%% file: iclr2027_conference.tex
\documentclass{article} 
\usepackage{iclr2027_conference,times}

\input{math_commands.tex}

\usepackage{hyperref}
\usepackage{url}

\usepackage[nolist]{acronym}
\input{acronyms}
\usepackage{graphicx}
\usepackage{subcaption}
\usepackage{enumitem}

\usepackage{booktabs,tabularx,array}
\newcolumntype{Y}{>{\raggedright\arraybackslash}X}

\usepackage{amsmath, amssymb, amsthm}
\usepackage{thmtools}
\usepackage{thm-restate}
\usepackage{algorithm,algorithmic}
\usepackage{multirow}

\newtheorem{lemma}{Lemma}

\newtheorem{remark}{Remark}
\newtheorem{corollary}{Corollary}

\title{Annealed Sinkhorn with Momentum: Certified Unregularized Optimal Transport in Linear Memory}

\author{Samuel J. K. Chin \\
MIT\\
\texttt{jkschin@mit.edu}
\And
Maximilian Schiffer\\
HEC Paris \\
\texttt{schiffer@hec.fr}
}

\iclrfinalcopy 
\begin{document}

\maketitle

\begin{abstract}
We characterize Bregman Douglas--Rachford splitting (BDRS) for unregularized discrete optimal transport and develop an anytime primal--dual certificate in linear memory. 
We first establish that BDRS coincides with warm-started Inexact Proximal point method for exact Optimal Transport (IPOT) using a single inner Sinkhorn iteration. 
By eliminating the primal transport plan from the updates, we derive an equivalent dual formulation that reveals BDRS as annealed Sinkhorn under an implicit inverse-linear temperature schedule, with an additional log-scaling momentum term and a cooler kernel. While this explains the role of the temperature parameter in BDRS as an initial temperature, it also reduces the solver's memory requirement from quadratic to linear.
Utilizing this annealing perspective, we introduce overrelaxed BDRS, which combines annealing and overrelaxed scaling within a single recursion.
We derive a primal--dual certificate for both methods that can be evaluated in linear memory without transport plan construction, thus providing a computable stopping rule.
On the DOTmark benchmark, combining momentum with the cooler kernel produces substantially smaller optimality gaps than annealed Sinkhorn under the same schedule.
For pixel-level color transfer between $1024\times1024$ images, BDRS attains a lower repaired transport cost than MDOT-TNT with a 9$\times$ speed up, reaching a relative duality gap of $1.59\%$ in 24 minutes. We further demonstrate a color transfer with $4238\times2365$ images, yielding 10 million pixels per image and approximately one hundred trillion implicit transport entries, reaching a best relative duality gap of $2.41\%$ and $2.80\%$ within 35 hours in each direction on a single NVIDIA L40S GPU.
\end{abstract}

\input{main_content}

\clearpage
\subsection*{AI use statement}
In this work, we used generative AI tools to provide critical ingredients for proving mathematical claims, assist in the writing of proofs, propose or refine hypotheses, design or provide feedback on research methodology or experiments, implement methods.
We have not used generative AI tools to generate synthetic datasets, help develop theoretical models or conceptual frameworks, assist with translation, clean and reformat dataset, support qualitative and thematic data analysis, interpret results are not applicable to this work.
Additionally, we used generative AI tools to discover research topics or identify gaps, brainstorming, edit a research paper to improve readability, identify relevant literature, suggest a structure for a research paper.
We have reviewed all AI-assisted work. 
For example, LLM-generated code was verified with unit tests and both the code and unit tests were checked.
For proof writing, LLMs assisted with the manipulation of equations and the authors ultimately checked the equations. We further tested the proofs with code.
We take responsibility for the final content of this work,
including text, claims or artifacts produced with the aid of generative AI.

\subsection*{Ethics statement}
We do not foresee and ethics issues with our work.

\subsection*{Reproducibility statement}
All code used to support this work will be released upon publication.
The assumptions underlying our theoretical results are stated in Sections~2 and~3, with proofs and supporting derivations provided in Appendices~A--F.
Algorithm~1 specifies the log-domain solver, while Algorithm~2 and Appendix~F detail the computation of the primal--dual certificate.
Appendix~G describes the datasets, preprocessing, parameter choices, implementation details, numerical precision, hardware, and evaluation procedures used in our experiments.
Additional classwise results for the momentum and kernel ablation are provided in Appendix~H.


\bibliography{references}
\bibliographystyle{iclr2027_conference}

\clearpage
\appendix

\input{appendices/app_non_recursive_primal_update}
\input{appendices/app_bdrs_ipot}
\input{appendices/app_pure_dual_state_proof}
\input{appendices/app_a_bdrs_update}
\input{appendices/app_pure_dual_a_bdrs_proof}
\input{appendices/app_dual_state_certificate}
\input{appendices/app_experiments}
\input{appendices/app_experiments_momentum_kernel_ablation_figures}

\end{document}

%% file: math_commands.tex
\usepackage{amsmath,amsfonts,bm}

\def\eqref#1{equation~\ref{#1}}

\def\1{\bm{1}}

\DeclareMathAlphabet{\mathsfit}{\encodingdefault}{\sfdefault}{m}{sl}
\SetMathAlphabet{\mathsfit}{bold}{\encodingdefault}{\sfdefault}{bx}{n}



%% file: acronyms.tex
\begin{acronym}
    \acro{BDRS}{Bregman Douglas--Rachford Splitting}
    \acro{BPRS}{Bregman Peaceman--Rachford Splitting}
    \acro{OT}{optimal transport}
    \acro{EOT}{entropic optimal transport}
    \acro{LP}{linear program}
    \acro{POT}{Python Optimal Transport}
    \acro{PR}{Peaceman--Rachford}
    \acro{DR}{Douglas--Rachford}
    \acro{ADEMM}{Alternating Direction Exponential Multiplier Method}
    \acro{IPOT}{Inexact Proximal point method for exact Optimal Transport}
    \acro{ML}{machine learning}
    \acro{BPPA}{Bregman proximal point algorithm}
    \acro{PPA}{proximal point algorithm}
    \acro{DRS}{Douglas--Rachford splitting}
    \acro{MD}{mirror descent}
    \acro{POT}{Python Optimal Transport}
    \acro{AS}{Annealed Sinkhorn}
    \acro{DAS}{Debiased Annealed Sinkhorn}
    \acro{LAMP}{Log-Averaged Mirror Prox}
    \acro{MDOT-TNT}{Mirror Descent Optimal Transport -- Truncated Newton}
\end{acronym}

%% file: main_content.tex
\section{Introduction}
The discrete \ac{OT} problem is a fundamental optimization problem underlying applications across machine learning, computer vision, economics, and computational biology, dating back to \citet{Monge1781MemoireRemblais} and its modern linear programming formulation by \citet{Kantorovich1942OnMasses}. The \emph{unregularized} problem is of particular interest because it is free of entropic bias: it admits sparse transport plans and exact transport costs. 
Classical methods for unregularized \ac{OT} include network simplex, combinatorial min-cost-flow methods \citep{Pele2009FastDistances}, and interior-point methods \citep{Lee2014PathFlow}.
Recent algorithms also achieve near-linear complexity under assumptions on the input geometry \citep{Agarwal2024FastSettings}.
For dense, large-scale applications, computational and memory efficiency remain central practical considerations.
Entropic OT (EOT), popularized by \citet{Cuturi2013SinkhornTransport}, restored scalability through the Sinkhorn algorithm \citep{Sinkhorn1967ConcerningMatrices}, whose iterations reduce to computing dense matrix--vector products, ideally suited to GPUs.

The efficiency of EOT comes with an inherent speed--accuracy trade-off: as the temperature decreases, the EOT minimizer approaches the unregularized solution, but iteration counts grow and numerical instability arises. A rich line of work navigates this trade-off. \citet{Altschuler2017Near-linearIteration} obtain near-linear-time approximation guarantees via Sinkhorn iterations followed by a rounding step that restores exact marginal feasibility; \citet{Dvurechensky2018ComputationalAlgorithm} establish an $\widetilde{\mathcal{O}}(n^2/\varepsilon^2)$ complexity bound for Sinkhorn and propose the accelerated APDAGD. \citet{Schmitzer2019StabilizedProblems} combines log-domain stabilization with $\varepsilon$-scaling to warm-start Sinkhorn across decreasing temperatures, while \citet{Thibault2021OverrelaxedTransport} accelerate convergence at a fixed temperature via overrelaxed projections. Most recently, \citet{Chizat2024AnnealedDebiasing} characterized the temperature schedules under which \emph{annealed} Sinkhorn converges to unregularized OT, quantified the relaxation error induced by annealing, and proposed a debiased variant. These methods reduce, but do not eliminate, the underlying trade-off.

A complementary line of work applies \acp{PPA} to unregularized OT directly \citep[see][]{Eckstein1993NonlinearProgramming,Parikh2014ProximalAlgorithms,Rockafellar1976MonotoneAlgorithm}. When the Bregman divergence generated by the entropy is used, each proximal subproblem is an EOT problem and can be addressed with Sinkhorn iterations: \ac{IPOT} \citep{Xie2020ADistance} runs a small fixed number of inner iterations (typically one), although its convergence analysis requires subproblem accuracies that a single iteration is not known to satisfy. Subsequent work therefore certifies or accelerates the inner solves: inexact stopping criteria and inertial variants \citep{Yang2022BregmanVariant}, dual block-coordinate solvers \citep{Chu2023AnProblems}, mirror-descent formulations with adaptive temperature schedules \citep{Kemertas2025ATransport,Kemertas2025EfficientGradients}, and Newton-type inner solvers \citep{Tang2024SafeTransport,Wu2026PINS:Transport,Pan2026InexactTransport}.

A third family attacks OT through operator splitting. Douglas--Rachford splitting (DRS) has a long history in convex optimization \citep{Eckstein1992OnOperators}, and has been specialized to OT in several forms: DROT \citep{Mai2021AImplementation} applies Euclidean DRS to unregularized OT, PDOT \citep{Lu2024PDOT:Transport} applies restarted primal--dual hybrid gradient, and \citet{Lindback2023BringingGPUs} adapt DRS to EOT on GPUs; \citet[Table~2]{Burns2026Log-AveragedSpace} catalogue first-order methods for OT in detail. Within this family, \citet{Ma2025BregmanMethod} recently showed that a special case of \ac{BDRS} can be applied directly to unregularized OT; hereafter we use BDRS to refer to this specific case. Overall, BDRS appears to be a distinct method, unrelated to the proximal point and Sinkhorn-based algorithms above. Moreover, \citet[Remark~5.1]{Ma2025BregmanMethod} observe empirically that BDRS performs well \emph{without} requiring the regularization parameter to be tuned toward zero --- an observation that has so far lacked an explanation.

In this paper, we present, among others, an exact algebraic characterization of BDRS that closes this gap. We first show that BDRS coincides with IPOT run with a single inner Sinkhorn iteration, connecting the operator splitting and proximal point views of the same matrix-scaling iteration. We then eliminate the primal transport plan from the iteration entirely and derive an equivalent \emph{dual} formulation: the resulting updates are precisely those of annealed Sinkhorn under an inverse-linear temperature schedule, augmented by a momentum-type extrapolation of the dual potentials. In other words, BDRS \emph{is} \ac{AS} with momentum and a cooler kernel.
This characterization reveals its hidden cooling even when its temperature parameter is held fixed.
We further propose a principled overrelaxation and an anytime optimality certificate that can be computed with linear memory.
Furthermore, it has three practical consequences that we develop in the remainder of the paper: the dual form needs only linear memory, the annealing view suggests a principled overrelaxation, and every sweep admits a computable optimality certificate. Our contributions are as follows.

\begin{itemize}[leftmargin=0pt, label={}, itemsep=2pt]
    \item \textbf{(1) Exact characterization and linear memory.} We establish that BDRS is equivalent to \ac{IPOT} \citep{Xie2020ADistance} with a single inner Sinkhorn iteration and derive a dual formulation of BDRS, revealing it to be \ac{AS} \citep{Chizat2024AnnealedDebiasing} under an implicit inverse-linear cooling schedule with a momentum term and cooler kernel. In Dual \ac{BDRS}, $\eta$ controls the initial temperature and eliminates the $\mathcal{O}(N^2)$ primal transport plan from the iteration, reducing persistent solver memory to $\mathcal{O}(N)$.
    \item \textbf{(2) Overrelaxation.} Guided by the annealing
    interpretation, we introduce Overrelaxed BDRS (O-BDRS), governed by a single parameter $\lambda \in (1,2)$ that steepens the implicit cooling schedule. Its dual formulation unifies the mechanics of annealed and overrelaxed Sinkhorn \citep{Thibault2021OverrelaxedTransport} in a single recursion, and setting $\lambda = 1$ recovers BDRS.
    \item \textbf{(3) Anytime certificate in linear memory.} We derive a computable primal--dual certificate for both Dual BDRS and O-BDRS without the need to compute the primal transport plan. The duality gap certifies additive $\tau$-optimality for the unregularized problem and provides a stopping rule that preserves linear memory. The certificate is valid independently of any convergence guarantee.
    \item \textbf{(4) Scale.} For pixel-level color transfer between $1024 \times 1024$ images, Dual BDRS reaches a relative duality gap of $1.6\%$ in $24$ minutes, attaining a lower feasible transport cost than MDOT-TNT \citep{Kemertas2025ATransport} with a $9\times$ speedup. We further demonstrate color transfer between $4238 \times 2365$ images, solving a single global OT problem with more than $10$ million atoms per image and $10^{14}$ implicit transport entries. 
    Both $1000$-iteration runs in each direction take approximately $35$ hours on a single NVIDIA L40S GPU, attaining best relative duality gaps of $2.41\%$ and $2.80\%$ respectively.
\end{itemize}

\section{Background and \ac{BDRS} Preliminaries}
We are interested in solving the unregularized discrete \ac{OT} problem. For brevity, all references to the \ac{OT} problem in this paper refer to this variant unless otherwise stated. The equations are given by \citep{Kantorovich1942OnMasses, Peyre2020ComputationalTransport}:

\begin{equation}
p^\star = \min_{X \in \Pi(r, c)} \langle C, X \rangle, \quad \Pi(r, c) = \{ X \ge 0 : X \mathbf{1}_n = r, \, X^\top \mathbf{1}_m = c \},
\label{eq:unregularized_ot_primal}
\end{equation}

where $X\in\mathbb{R}^{m\times n}_{+}$ is the transportation plan,
$r\in\mathbb{R}^{m}_{++}$ and $c\in\mathbb{R}^{n}_{++}$ are the source
and target marginals and have equal total mass of one, $\Pi$ is the transportation polytope and $C\in\mathbb{R}^{m\times n}$ is the transport cost matrix.
Entropic \ac{OT} replaces this objective by $\langle C, X \rangle - \eta H(X)$, where $\eta$ is a temperature that controls the approximation to the unregularized version and $H(X):=-\sum_{i=1}^{m}\sum_{j=1}^{n}X_{ij}\log X_{ij}$. The dual of \eqref{eq:unregularized_ot_primal} is given by:
\begin{equation}
    p^\star
    =
    \max_{f\in\mathbb{R}^m,\,g\in\mathbb{R}^n}
    \left\{
        \langle r,f\rangle+\langle c,g\rangle:
        f_i+g_j\le C_{ij}\ \text{for all }i,j
    \right\}.
    \label{eq:unregularized_ot_dual}
\end{equation}
Here, $f$ and $g$ are commonly called the source and target Kantorovich (or dual) potentials and every dual feasible pair $(f,g)$ provides a lower bound on $p^\star$.
\citet{Ma2025BregmanMethod} proposes a special case of \ac{BDRS} that targets the dual form of unregularized discrete \ac{OT}.
The update equations admit a natural primal-dual matrix-scaling interpretation and are given as follows:
\begin{subequations}
\label{eq:bdrs}
\begin{align}
    u^{k} &:= \frac{r}{(X^k \odot K)v^{k-1}} \label{eq:bdrs_u} \\
    v^{k} &:= \frac{c}{(X^k \odot K)^\top u^k} \label{eq:bdrs_v} \\
    X^{k+1} &:= \text{diag}(u^k)(X^k \odot K)\text{diag}(v^k) \label{eq:bdrs_x}
\end{align}
\end{subequations}
where $K = \exp(-C/\eta)$ and $\odot$ is element-wise multiplication. 
That is to say, that $K^{\odot (k+1)} = \exp(-C/\eta_{k+1})$, where $\eta_{k+1}=\frac{\eta}{k+1}$, is the kernel $K$ cooled by temperature $\eta_{k+1}$.
Throughout, division between vectors is understood element-wise, using the symbol $\oslash$ or as a fraction.
Let us define the cumulative row and column scaling vectors up to step $k$ as $a^k = \bigodot_{t=0}^k u^t$ and $b^k = \bigodot_{t=0}^k v^t$. 
Throughout this section, we adopt the initializations $a^{-1}=u^{-1}=\mathbf{1}_m$, $b^{-1}=v^{-1}=\mathbf{1}_n$, and $X^0=\mathbf{1}_{m\times n}$.
It is important to note that by definition, $(X^{k+1})^\top\mathbf{1}_m=c$, which is a property that we will be used extensively in the paper. 
We next state a result that underlies both the dual formulation of \ac{BDRS} and the certified primal--dual bound developed later.
\begin{restatable}[Non-recursive Primal Update]{lemma}{nonrecursiveprimalupdate}
\label{lem:non_recursive_primal_update}
The matrix $X^k$ generated by the \ac{BDRS} iterations at the start of any iteration $k \ge 0$ takes the following non-recursive form:
\begin{equation}
    X^{k+1} = \text{diag}(a^{k}) K^{\odot (k+1)} \text{diag}(b^{k}).\label{eq:bdrs_x_non_recursive}
\end{equation}
Equivalently,
\begin{equation}
    X_{ij}^{k+1}
    =
    \exp\left(
        \log a_i^k+\log b_j^k
        -\frac{k+1}{\eta}C_{ij}
    \right).
\end{equation}
where $X^k$ no longer relies on a recursive update.
\end{restatable}
We defer the proof to Appendix~\ref{app:non-recursive-primal-update} and state two important intuitions.
First, substituting the factorization into the \ac{BDRS} updates $u^k$ and $v^k$ eliminates the intermediate $X^k$ from the algorithm, leading to the dual formulation that we derive in
Section~\ref{sec:dual_bdrs}.
Second, the entrywise exponential form, together with column feasibility of $X^{k+1}$, enables the construction of a dual feasible pair $(f,g)$ and hence a certified dual lower bound in Section \ref{sec:certified-primal-dual-bound}.

\section{Main Results}
We present four main results in this section.
First, we establish that \ac{BDRS} is equivalent to \ac{IPOT} \citep{Xie2020ADistance} with a single Sinkhorn iteration.
Second, we show that primal-dual \ac{BDRS} can be formulated as dual \ac{BDRS}, which exposes its connection to \ac{AS} \citep{Chizat2024AnnealedDebiasing} and enables an $\mathcal{O}(n)$ memory implementation.
Third, we introduce O-\ac{BDRS} and dual O-\ac{BDRS}, which unifies overrelaxed and annealed Sinkhorn. 
Finally, we construct a stopping criterion based on primal--dual bounds instead of relying on row and column marginal errors.

\subsection{BDRS is equivalent to IPOT}
We first establish a connection between \ac{BDRS} and \ac{IPOT} \citep{Xie2020ADistance} for discrete \ac{OT}. 
Specifically, warm-started IPOT with a single step in the inner iteration is equivalent to \ac{BDRS}.
Although the two methods arise from fundamentally different optimization perspectives, they yield identical update equations when applied to unregularized \ac{OT}. 
Specifically, \citet{Xie2020ADistance} derives IPOT from the primal perspective using an inexact Bregman proximal point method, whereas \citet{Ma2025BregmanMethod} derives \ac{BDRS} from the dual perspective using Bregman Douglas--Rachford splitting. Despite these distinct derivations, setting the number of IPOT inner iterations to $L=1$ yields exactly the \ac{BDRS} updates. This equivalence provides a useful bridge between the proximal-point and operator-splitting interpretations of unregularized \ac{OT} and may provide insight into future convergence proofs under specific conditions, which is not the main focus of this paper. 
We refer the reader to Appendix~\ref{app:bdrs_ipot} for a detailed derivation of the equivalence.

\subsection{Dual BDRS}
\label{sec:dual_bdrs}
In the following, we first use Lemma \ref{lem:non_recursive_primal_update} to unroll BDRS into its dual form. Second, we show its connection to annealed Sinkhorn.
Observe in \eqref{eq:bdrs} that evaluating the sequence requires storing $X^k$ at each step, which can be prohibitively large. In this section, we establish that $X^k$ is entirely redundant for the progression of the dual variables. By algebraically unrolling the BDRS sequence, we demonstrate that the algorithm can be fully decoupled from the primal state, reducing the updates to a highly efficient sequence of alternating projections over dual scaling vectors only.

\begin{restatable}[Dual Formulation of BDRS]{proposition}{bdrsprop}
\label{prop:dual_bdrs}
The \ac{BDRS} iterations can be completely decoupled from the intermediate matrix $X^k$. The algorithm reduces exactly to a sequence of alternating projections over the cumulative scaling vectors $a$ and $b$, with initial conditions $a^{-1} = b^{-1} = b^{-2} = \mathbf{1}$:
\begin{subequations}
\label{eq:dual_bdrs}
\begin{align}
    a^{k} &= \frac{r}{K^{\odot (k+1)} \left[ \frac{(b^{k-1})^{\odot 2}}{b^{k-2}} \right]} \label{eq:dual_bdrs_u} \\
    b^{k} &= \frac{c}{(K^{\odot (k+1)})^\top a^{k}} \label{eq:dual_bdrs_v} \\
    X^{k+1} &= \text{diag}(a^k) K^{\odot (k+1)} \text{diag}(b^k) \label{eq:dual_bdrs_x}
\end{align}
\end{subequations}
\end{restatable}
The proof is deferred to Appendix~\ref{app:proof-pure-dual-bdrs}.
Dual BDRS indeed implicitly stores the cumulative scaling vectors $a^k$ and $b^k$, and it is important to note that the variables are arbitrary and can be simply replaced with $u^k$ and $v^k$ such that the form is analogous to Sinkhorn.

\begin{corollary}[Dual BDRS is annealed Sinkhorn with momentum]
\label{cor:bdrs_momentum}
Relative to the standard annealed Sinkhorn recursion of
\citet{Chizat2024AnnealedDebiasing}, BDRS makes exactly two
modifications to the row update:
\begin{equation}
    \underbrace{
        \frac{r}{K^{\odot k}b^{k-1}}
    }_{\text{Annealed Sinkhorn}}
    \quad\longrightarrow\quad
    \underbrace{
        \frac{r}{
            K^{\odot(k+1)}
            \left[(b^{k-1})^{\odot 2}\oslash b^{k-2}\right]
        }
    }_{\text{Dual BDRS}}.
    \label{eq:annealed_sinkhorn_to_bdrs}
\end{equation}
First, dual BDRS advances from $K^{\odot k}$ to
$K^{\odot(k+1)}$. Second, it replaces the
previous column scaling $b^{k-1}$ by the extrapolated scaling
\begin{equation}
    \widetilde b^{\,k-1}
    :=
    (b^{k-1})^{\odot 2}\oslash b^{k-2}.
\end{equation}
The column update and the completed matrix retain the annealed
Sinkhorn form shown in Proposition~\ref{prop:dual_bdrs}. Moreover,
\begin{equation}
    \log\widetilde b^{\,k-1}
    =
    \log b^{k-1}
    +
    \left(
        \log b^{k-1}-\log b^{k-2}
    \right),
\end{equation}
so the second modification is unit inertial extrapolation in the
column log scaling. In this precise sense, BDRS is annealed Sinkhorn
with log scaling momentum on a cooler kernel.
\end{corollary}

The result follows by direct comparison with
Proposition~\ref{prop:dual_bdrs}. The exact index matching and
initialization relative to \citet{Chizat2024AnnealedDebiasing} are
provided in Appendix~\ref{app:annealed-sinkhorn-comparison}.

\begin{remark}
Dual \ac{BDRS} reveals that primal-dual BDRS implicitly follows an inverse-linear temperature schedule even when $\eta$ is held fixed.
This provides a potential explanation for the observation of \citet[Remark~5.1]{Ma2025BregmanMethod} that the algorithm performs well without requiring $\eta$ to be close to zero.
A related cumulative-scaling recurrence appears in \citet{Ma2026ConvergenceAlgorithm}, where $C=0$ and $X^0=M>0$, contains the same second-order extrapolation.
For general transport costs, our characterization exposes the evolving kernel $K^{\odot(k+1)}$ and thereby identifies the implicit annealing.
Although \citet{Chizat2024AnnealedDebiasing} noted similarities between \ac{AS} and \ac{IPOT}, we prove that \ac{BDRS} and \ac{IPOT} are equivalent when $L=1$ and show precisely how \ac{AS} transforms into \ac{BDRS}/\ac{IPOT}.
We note that the convergence analysis from \ac{AS} does not immediately transfer to \ac{BDRS} and we defer this analysis to future work.
\end{remark}

\paragraph{Memory Efficiency of Dual BDRS}
The bottleneck of primal-dual \ac{BDRS} is its quadratic memory requirement, as the dense transport plan $X^k \in \mathbb{R}^{m \times n}$ must be explicitly stored throughout the iterations. For \ac{OT} problems with source and target distributions of size $N$, this results in $\mathcal{O}(N^2)$ memory complexity. 
Dual \ac{BDRS} eliminates this bottleneck by removing $X^k$ from the intermediate updates and requires only the one-dimensional vectors.
Consequently, the memory required for the algorithm is reduced from $\mathcal{O}(N^2)$ to $\mathcal{O}(N)$. 
The $O(N)$ bound assumes that pairwise costs can be evaluated on demand from a linear-size representation. An explicitly stored dense cost matrix requires $O(N^2)$ memory.
Regardless, each iteration still requires $O(N^2)$ pairwise work.

\subsection{Overrelaxed BDRS}
\label{sec:overrelaxed_bdrs}
We established that BDRS is an annealed Sinkhorn algorithm with a built-in momentum term in Corollary~\ref{cor:bdrs_momentum}.
We now show that there is a principled manner to obtain a faster annealing rate. To do so, we first observe that the standard BDRS state update (equation \ref{eq:bdrs_x}) can be factored as follows:
\begin{equation*}
    X^{k+1} = X^k \odot \Big( K \odot u^k (v^k)^\top \Big)
\end{equation*}
In this form, the mechanism is transparent: the update takes the previous state $X^k$ and applies an element-wise multiplicative correction factor, defined as $(K \odot u^k(v^k)^\top)$. Building directly on this geometric insight, we introduce Overrelaxed BDRS (O-BDRS) by applying an overrelaxation parameter $\lambda \in [1, 2)$ to explicitly stretch this correction step. Exponentiating the correction factor element-wise yields the O-BDRS primal update:
\begin{equation}
    X^{k+1} = X^k \odot \Big( K \odot u^k (v^k)^\top \Big)^{\odot \lambda} \label{eq:o_bdrs_x}
\end{equation}
This specific exponentiation holds a connection to \ac{BPRS}, the details of which we defer to Appendix \ref{app:a-bdrs-update}. Crucially, by altering only this primal update rule, we can again manipulate the overrelaxed algorithm into a dual formulation and show that it elegantly unifies overrelaxed and annealed Sinkhorn. All proofs in this section are deferred to Appendix~\ref{app:proof-pure-dual-a-bdrs}.

\begin{restatable}[O-BDRS as Overrelaxed Annealed Sinkhorn]
{proposition}{accbdrsprop}
\label{prop:dual-o-bdrs}
Define cumulative row and column scaling vectors exponentiated by $\lambda$ up to step $k$ as $\alpha^k = \bigodot_{t=0}^k (u^t)^{\odot \lambda}$ and $\beta^k = \bigodot_{t=0}^k (v^t)^{\odot \lambda}$, $\lambda\in[1,2)$, with initial conditions $\alpha^{-1}=\beta^{-2}=\beta^{-1}=\mathbf{1}$.
For $k\geq0$, define
$K_k^{\mathrm O}:=K^{\odot(\lambda k+1)}$ and
\[
    \widetilde\beta^{k-1}
    :=
    \beta^{k-1}
    \odot
    \left(
        \beta^{k-1}\oslash\beta^{k-2}
    \right)^{\odot1/\lambda}.
\]
Then O-BDRS admits the following equations
\begin{align}
    \widehat\alpha^k
    &:=
    r\oslash
    \left(
        K_k^{\mathrm O}\widetilde\beta^{k-1}
    \right),
    &
    \widehat\beta^k
    &:=
    c\oslash
    \left(
        (K_k^{\mathrm O})^\top\widehat\alpha^k
    \right),
    \label{eq:obdrs_sinkhorn_sweep}\\
    \alpha^k
    &:=
    (\alpha^{k-1})^{\odot(1-\lambda)}
    \odot
    (\widehat\alpha^k)^{\odot\lambda},
    &
    \beta^k
    &:=
    (\beta^{k-1})^{\odot(1-\lambda)}
    \odot
    (\widehat\beta^k)^{\odot\lambda}.
    \label{eq:obdrs_overrelaxation}
\end{align}
The corresponding O-BDRS state is
\begin{equation}
    X^{k+1}
    =
    \operatorname{diag}(\alpha^k)
    K^{\odot\lambda(k+1)}
    \operatorname{diag}(\beta^k).
    \label{eq:obdrs_dual_state}
\end{equation}
\end{restatable}

Consequently, observe how this is analogous to an annealed Sinkhorn iteration, followed by the overrelaxation form proposed by \citet{Thibault2021OverrelaxedTransport}. Setting $\lambda=1$
recovers Dual BDRS.
The effect of $\lambda$ is discussed in Section~\ref{sec:experiments} and Appendix~\ref{app:matched-temperature-ablation}.

\subsection{Certified Primal--Dual Bounds in Linear Memory}
\label{sec:certified-primal-dual-bound}

Many existing algorithms monitor the marginal residuals, which measure
proximity to the transport polytope but do not certify the transport cost.
We now construct a primal--dual certificate for unregularized OT using the
scalings already computed by Dual BDRS and Dual O-BDRS. The dual lower bound
requires only vector operations. The primal upper bound additionally uses
a cost-weighted reduction that can be evaluated alongside the column update.
Both bounds preserve the linear-memory formulation but require additional computation of the same order as a solver iteration.

\paragraph{Scalings shared by both methods.}
For any iteration $k\geq 0$, we write the scalings and effective
temperature at this iteration as
\begin{equation}
\begin{array}{c|cccc}
    \text{Method} & A & B & t & \varepsilon \\
    \hline
    \text{BDRS}
    & a^k & b^k & \widetilde b^{k-1} & \eta/(k+1) \\
    \text{O-BDRS}
    & \widehat\alpha^k & \widehat\beta^k
    & \widetilde\beta^{k-1} & \eta/(\lambda k+1)
\end{array}
\label{eq:cert_shared_scalings}
\end{equation}
where $t$ is the incoming column scaling used in the row update:
\[
    \widetilde b^{k-1}
    :=(b^{k-1})^{\odot 2}\oslash b^{k-2},
    \qquad
    \widetilde\beta^{k-1}
    :=\beta^{k-1}\odot
       (\beta^{k-1}\oslash\beta^{k-2})^{\odot 1/\lambda}.
\]
The stated initializations give $t=\mathbf{1}_n$ at $k=0$.
With $K_\varepsilon:=\exp(-C/\varepsilon)$, both methods satisfy
\begin{equation}
    A=r\oslash(K_\varepsilon t),
    \qquad
    B=c\oslash(K_\varepsilon^\top A).
    \label{eq:cert_normalizations}
\end{equation}
For O-BDRS, these are the unrelaxed scalings
$\widehat\alpha^k$ and $\widehat\beta^k$, evaluated before
overrelaxation at the effective temperature $\eta/(\lambda k+1)$.
Consider the implicit plan immediately after the row update,
\begin{equation}
    Z:=\operatorname{diag}(A)K_\varepsilon\operatorname{diag}(t).
    \label{eq:cert_intermediate_plan}
\end{equation}
By construction, its row marginal is $r$, and its column marginal is
\begin{equation}
    s:=Z^\top\mathbf{1}_m
      =t\odot(K_\varepsilon^\top A)
      =c\odot t\oslash B.
    \label{eq:cert_intermediate_columns}
\end{equation}
Observe that $c$, $t$, and $B$ are already available, so the column
marginal $s=c\odot t\oslash B$ of $Z$ can be computed using only
elementwise vector operations. One could instead use the
column-feasible matrix
$Y=\operatorname{diag}(A)K_\varepsilon\operatorname{diag}(B)$,
but evaluating its row marginal $A\odot(K_\varepsilon B)$ would
require an additional kernel--vector multiplication. We therefore
use $Z$ for the upper bound: its row constraint holds by
construction, and its column error can be evaluated directly
from the available scalings.

\begin{restatable}[Primal--Dual Certificate]{proposition}{primaldualcertificate}
\label{prop:dual_state_certificate}
Let $r\in\mathbb{R}_{++}^m$ and $c\in\mathbb{R}_{++}^n$ have total mass one,
let $C\in\mathbb{R}^{m\times n}$ have finite entries, and let
$A,B,t>0$ and $\varepsilon>0$ satisfy
\eqref{eq:cert_normalizations}. Let $\Omega$ be any known bound satisfying
$\Omega\geq\operatorname{osc}(C)$. Define
\begin{align}
    f&:=\varepsilon\log A,
    &g&:=\varepsilon\log(B\oslash c),
    \label{eq:cert_dual_potentials}\\
    \delta&:=\frac12\|s-c\|_1,
    &w&:=(C\odot K_\varepsilon)^\top A,
    \label{eq:cert_upper_statistics}
\end{align}
and
\begin{equation}
    L:=\langle r,f\rangle+\langle c,g\rangle,
    \qquad
    U:=\langle t,w\rangle+\Omega\delta.
    \label{eq:cert_bounds}
\end{equation}
Then $(f,g)$ is feasible for the unregularized Kantorovich dual.
Furthermore, the intermediate plan $Z$ admits a feasible repair
$\widehat Z\in\Pi(r,c)$ such that
\begin{equation}
    L\leq p^\star
    \leq\langle C,\widehat Z\rangle
    \leq U.
    \label{eq:cert_interval}
\end{equation}
Consequently,
\begin{equation}
    0\leq\langle C,\widehat Z\rangle-p^\star\leq U-L.
    \label{eq:cert_gap}
\end{equation}
\end{restatable}

The proof and the construction of $\widehat Z$, and the stabilized log-domain computation of $U$ are given in
Appendix~\ref{app:dual_state_certificate}. The upper bound has two terms:
$\langle t,w\rangle=\langle C,Z\rangle$ is the cost of the intermediate
plan, and $\Omega\delta$ bounds the cost of restoring its column marginal.
The repair is the row-feasible specialization of the rounding procedure
of \citet{Altschuler2017Near-linearIteration}.
For costs in $[0,1]$, one may take $\Omega=1$ without computing the exact
cost range.

\subsection{Algorithm and Log-Domain Execution}
\label{sec:algorithm}

Algorithm~\ref{alg:o_bdrs} synthesizes Dual O-BDRS directly in the log-domain using the LogSumExp (LSE) operator for numerical stability.
In each iteration, we define an implicit effective temperature $\eta_k$ and compute an extrapolated momentum variable $\log \tilde{\beta}^{k-1}$. 
The core iteration alternates between evaluating unrelaxed row and column targets ($\log \hat{\alpha}^k$ and $\log \hat{\beta}^k$) and applying geometric overrelaxation via $\lambda \in [1,2)$. 
Here, $\operatorname{LSE}$ reduces over the columns of its matrix argument: $[\operatorname{LSE}(M)]_p=\log\sum_q\exp(M_{pq})$. The column update uses the same operator with the transposed cost matrix.
Setting $\lambda=1$ perfectly recovers Dual BDRS.
When a convergence check is required, the anytime primal-dual certificate $U-L$ is evaluated periodically in linear memory using the available scalings, as detailed in Proposition~\ref{prop:dual_state_certificate} and Appendix~\ref{app:dual_state_certificate}.

\begin{algorithm}[htbp]
\caption{Dual O-BDRS (Log-Domain)}
\label{alg:o_bdrs}
\begin{algorithmic}[1]
\STATE \textbf{Initialize:} $\log \alpha^{-1}, \log \beta^{-2}, \log \beta^{-1} \leftarrow \mathbf{0}$
\FOR{$k = 0, 1, 2, \dots$}
    \STATE $\eta_k \leftarrow \eta / (\lambda k + 1)$ 
    \STATE $\log \tilde{\beta}^{k-1} \leftarrow \log \beta^{k-1} + \frac{1}{\lambda} (\log \beta^{k-1} - \log \beta^{k-2})$
    \STATE $\log \hat{\alpha}^k \leftarrow \log r - \text{LSE}\left(-\frac{C}{\eta_k} + \mathbf{1}_m (\log \tilde{\beta}^{k-1})^\top \right)$
    \STATE $\log \alpha^k \leftarrow (1-\lambda) \log \alpha^{k-1} + \lambda \log \hat{\alpha}^k$
    \STATE $\log \hat{\beta}^k \leftarrow \log c - \text{LSE}\left(-\frac{C^\top}{\eta_k} + \mathbf{1}_n (\log \hat{\alpha}^k)^\top \right)$
    \STATE $\log \beta^k \leftarrow (1-\lambda) \log \beta^{k-1} + \lambda \log \hat{\beta}^k$
    \IF{certificate requested}
        \STATE Compute gap $U - L$ via Proposition 3 (see Algorithm~\ref{alg:stable_upper_bound}, Appendix~\ref{app:dual_state_certificate})
        \STATE \textbf{if} $U - L \le \tau$ \textbf{then break}
    \ENDIF
\ENDFOR
\RETURN $(\log\alpha^k,\log\beta^k,
\log\hat{\alpha}^k,\log\hat{\beta}^k,
\log\widetilde{\beta}^{k-1},k)$
\end{algorithmic}
\end{algorithm}

\section{Experiments}
\label{sec:experiments}

We evaluate the two modifications that distinguish BDRS from annealed Sinkhorn, examine the contribution of overrelaxation, and assess the practical cost of certification. The DOTmark experiments use all 450 image pairs at resolution $64\times64$, with unit mass marginals and squared Euclidean costs normalized to $[0,1]$. Reference optimal values ($p^*$) are computed using POT's network simplex solver. Unless otherwise stated, we use log-domain implementations in \texttt{float64} on a single NVIDIA H200 GPU, set $\eta=1$, and allow up to $10^5$ iterations. Appendix~\ref{app:experiments} gives the preprocessing, implementation, and evaluation details.

\paragraph{Momentum and kernel ablation.}
We isolate the current row-update kernel and log-scaling momentum in a $2\times2$ ablation, and additionally compare against \ac{DAS} and O-\ac{BDRS} with $\lambda\in\{1.2,1.5,1.99\}$. Figure~1 shows that changing the kernel alone has little effect under the tested schedule. Momentum alone improves the result, but combining both modifications, as in BDRS, produces substantially smaller rounded primal optimality gaps. O-BDRS further improves the aggregate performance at fixed $\eta$. These comparisons use the inverse-linear temperature schedule intrinsic to BDRS; they do not compare individually optimized annealing schedules. Details and classwise results appear in Appendix~\ref{app:momentum-kernel-ablation} and Appendix~\ref{app:figures-for-momentum-and-kernel-ablation}.

\begin{figure}
    \centering
    \includegraphics[width=0.90\linewidth]{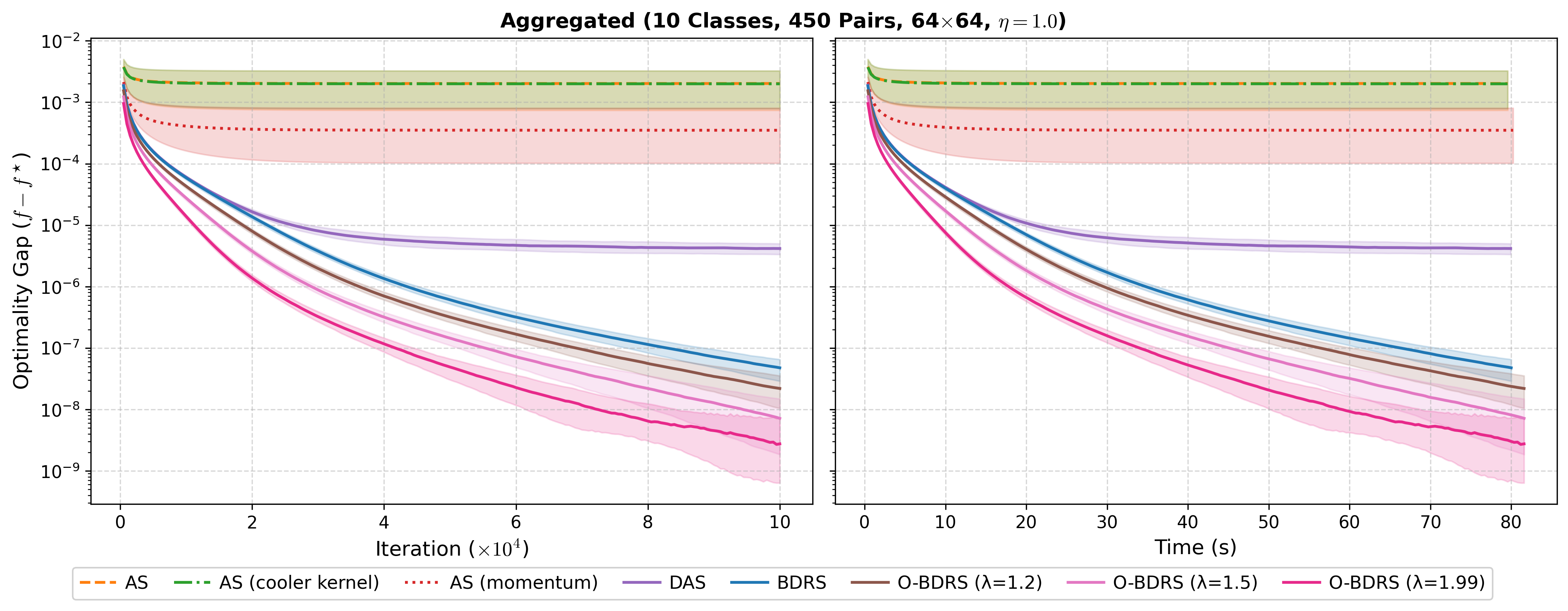}
    \caption{Momentum and kernel ablation on all 450 DOTmark pairs at resolution $64\times64$, with $\eta=1$. The optimality gap $f-f^\star$ is shown against iteration count (left) and median solver time, excluding feasibility rounding (right). Here $f$ is the cost after rounding and $f^\star$ is the reference value computed using POT's network simplex solver. Lines show medians across pairs; shaded bands indicate the interquartile range.}
    \label{fig:momentum-ablation-aggregated}
\end{figure}

\paragraph{What explains the gain from overrelaxation?}
The comparisons in Appendix~\ref{app:matched-temperature-ablation} indicate that much of the fixed-$\eta$ advantage is explained by reaching lower effective temperatures sooner. 
O-BDRS at approximately $k/\lambda$ iterations closely reproduces BDRS's rounded objective values at iteration $k$ at the later checkpoints. 
Lowering BDRS's initial $\eta$ to match O-BDRS's terminal update-kernel temperature largely removes the objective difference at the same iteration count. 
Thus, O-BDRS offers a reduction in iteration count at a common $\eta$, while these results do not establish a separate acceleration advantage over temperature-adjusted BDRS.
The relevant tables can be found in Table~2(a) and (b) in Appendix~\ref{app:matched-temperature-ablation}.

\paragraph{Effectiveness of the anytime certificate.}
This experiment uses a budget of $10^4$ iterations.
Figure~\ref{fig:delay-certificate-analysis} (left) shows that both methods reduce the certified gap during the run, with O-BDRS attaining a smaller aggregate certified gap at the later iterations. 
Of note is the region at $10^2$ iterations, where \ac{BDRS} actually has a slightly lower certified gap than O-\ac{BDRS} on average.
Figure~\ref{fig:delay-certificate-analysis} (right) shows the delay distributions and O-\ac{BDRS} arrives at a given tolerance earlier than \ac{BDRS} in general.
At $\tau=10^{-2}$, \ac{BDRS} typically requires roughly five to six times its oracle iteration count.
At $\tau=10^{-4}$, the difference widens again and the certification delay for \ac{BDRS} at this tolerance is not provided as the tolerance was not achieved by $10^4$ iterations.
For practitioners, this means that running O-\ac{BDRS} achieves the desired tolerance earlier than \ac{BDRS}. Appendix~\ref{app:effectiveness-of-anytime-certificate} gives the definitions, the experiment setup and classwise distributions.

\begin{figure}
    \centering
    \includegraphics[width=0.90\linewidth]{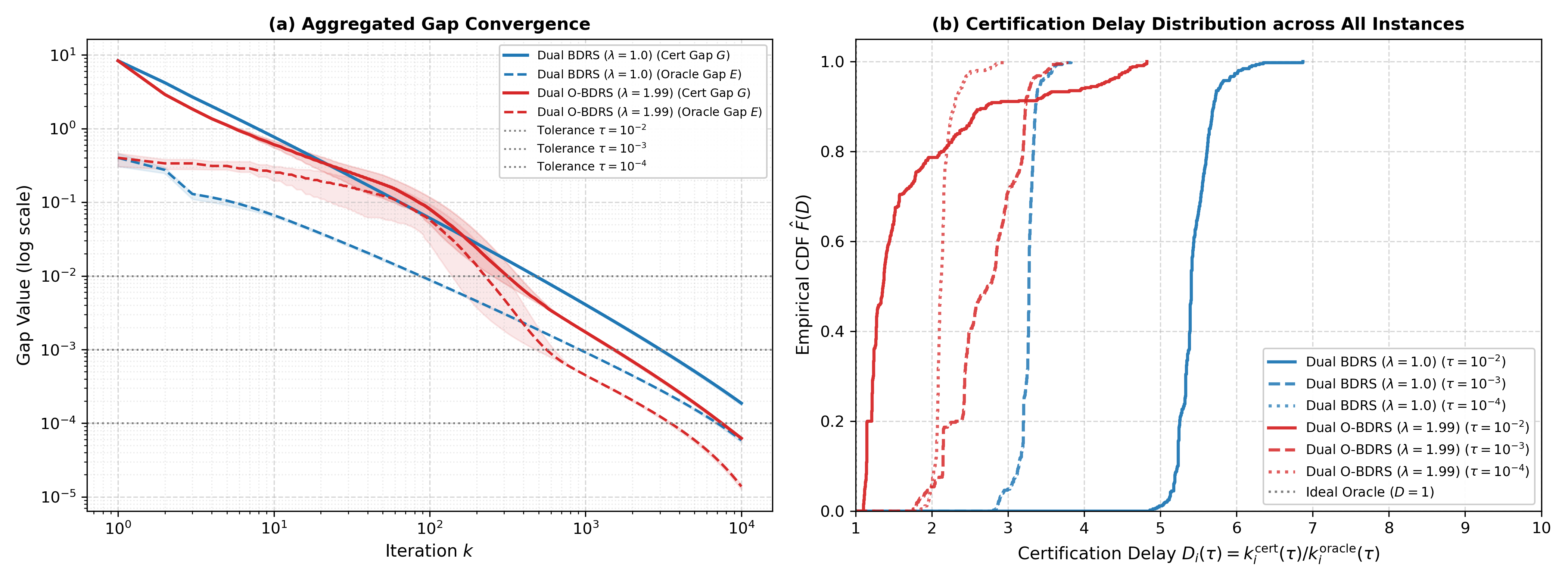}
    \caption{Anytime certification on 450 DOTmark pairs at resolution $64\times64$, comparing Dual BDRS ($\lambda=1$) and Dual O-BDRS ($\lambda=1.99$). Left: aggregated certified gaps $G^{\mathrm{cert}}=U^{\mathrm{best}}-L^{\mathrm{best}}$ (solid) and reference gaps $E^{\mathrm{oracle}}=U^{\mathrm{best}}-p^\star$ (dashed). 
    Right: empirical cumulative distributions of the certification delay $D_i(\tau)=k_i^{\mathrm{cert}}(\tau)/k_i^{\mathrm{oracle}}(\tau)$ for $\tau\in\{10^{-2},10^{-3},10^{-4}\}$.
    Smaller delays indicate closer agreement with each method's own oracle stopping rule.}
    \label{fig:delay-certificate-analysis}
\end{figure}

\paragraph{Large-Scale Pixel-Level Color Transfer}
The experiments here run in \texttt{float32} and use a single L40S GPU.
We compare with \ac{MDOT-TNT} \citep{Kemertas2025ATransport}, the current state-of-the-art in large-scale \ac{OT}, at $\gamma=2^{10}$ on a $10^6\times10^6$ \ac{OT} problem.
\ac{MDOT-TNT} required 3.5 hours and \ac{BDRS} finished in 24 minutes (9$\times$ speedup), while achieving a lower transport cost (0.0497 vs. 0.0507).
In addition, we run \ac{BDRS} on a $4238\times2365$ image, which translates to a $10^7\times10^7$ \ac{OT} problem, increasing the scale by an order of magnitude.
On this experiment, we ran \ac{BDRS} for 1000 iterations with $\eta=0.01$ over 35 hours and achieved a best duality gap of 2.41\%.
The outputs are shown in Figure~\ref{fig:color_transfer_bidirectional_4238x2365} and additional details are in Appendix~\ref{app:color-transfer}.

\begin{figure}[t]
    \centering

    \begin{minipage}[c]{0.04\linewidth}
        \raggedright
        \rotatebox{90}{\small Original Images}
    \end{minipage}%
    \hspace{0.01\linewidth}%
    \begin{minipage}[c]{0.44\linewidth}
        \includegraphics[width=\linewidth]{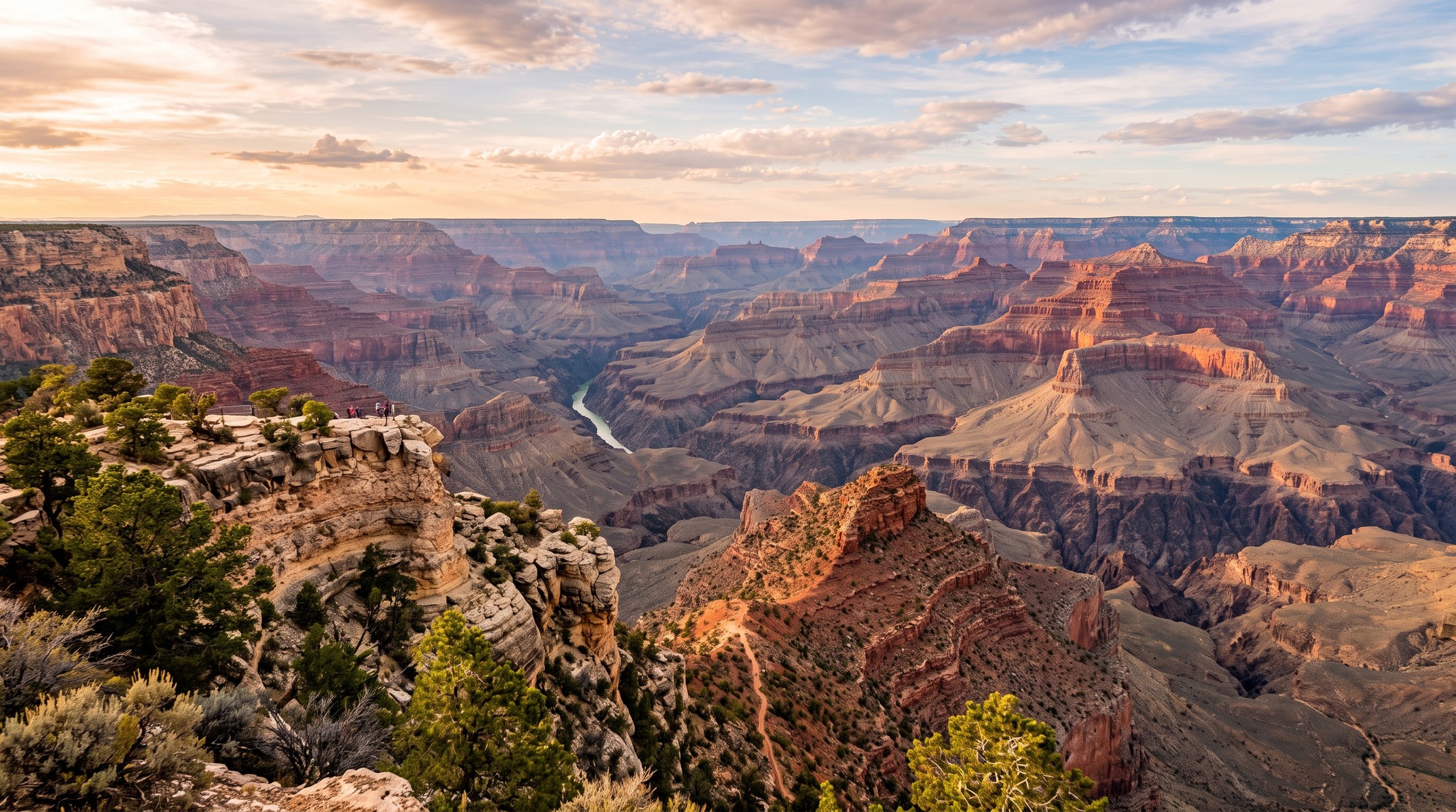}
    \end{minipage}%
    \hspace{0.01\linewidth}%
    \begin{minipage}[c]{0.44\linewidth}
        \includegraphics[width=\linewidth]{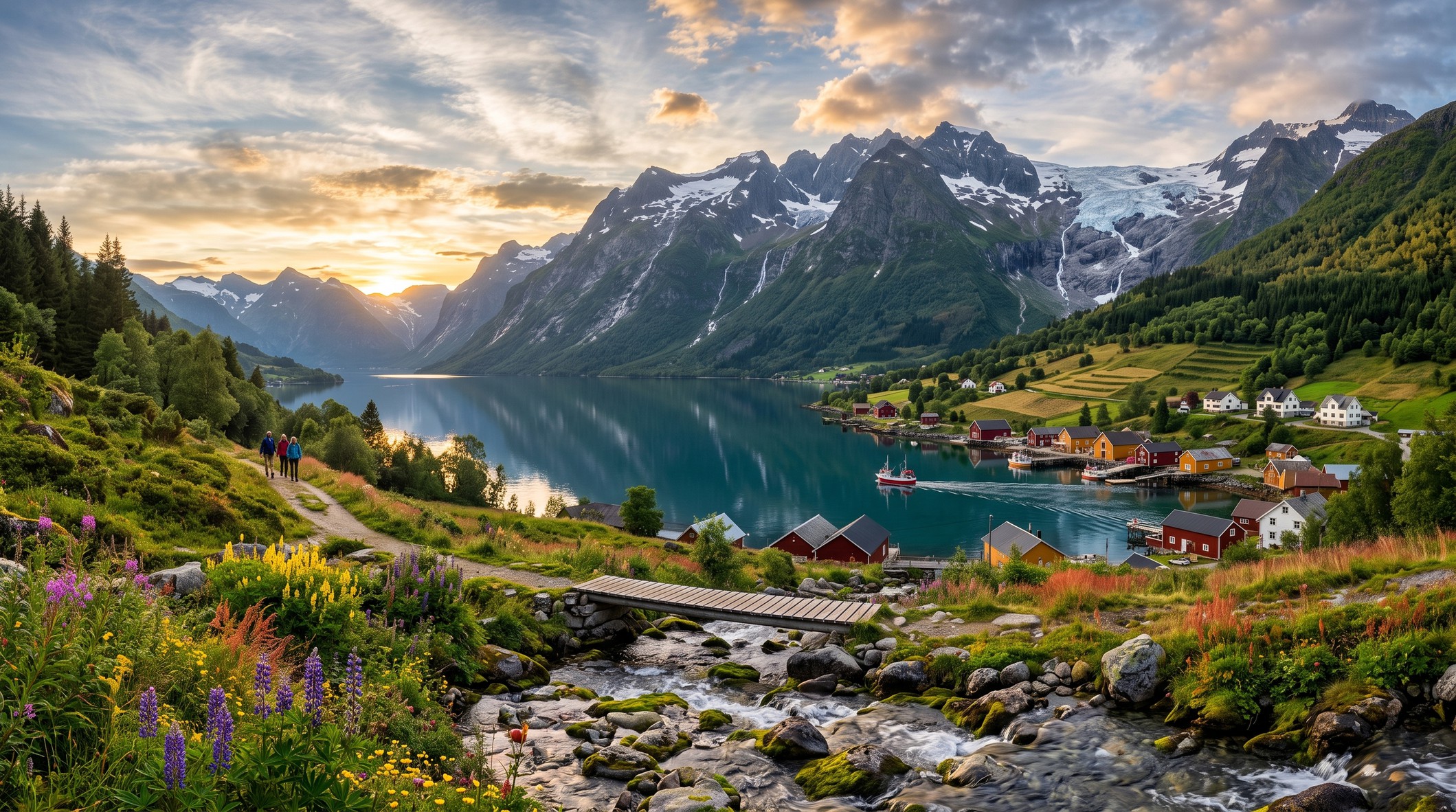}
    \end{minipage}

    \par\vspace{3pt}\nointerlineskip

    \begin{minipage}[c]{0.04\linewidth}
        \raggedright
        \rotatebox{90}{\small Color Transferred}
    \end{minipage}%
    \hspace{0.01\linewidth}%
    \begin{minipage}[c]{0.44\linewidth}
        \includegraphics[width=\linewidth]{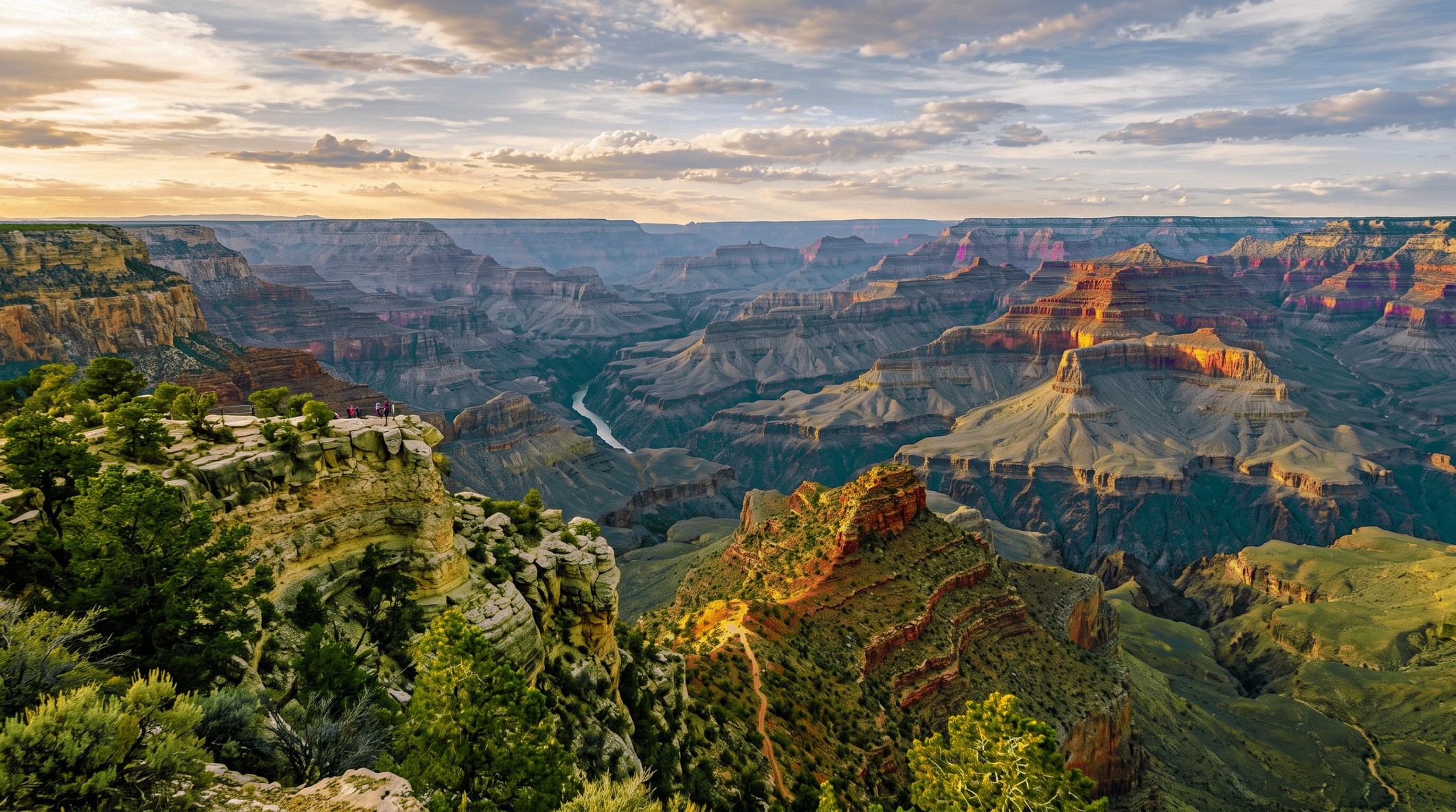}
    \end{minipage}%
    \hspace{0.01\linewidth}%
    \begin{minipage}[c]{0.44\linewidth}
        \includegraphics[width=\linewidth]{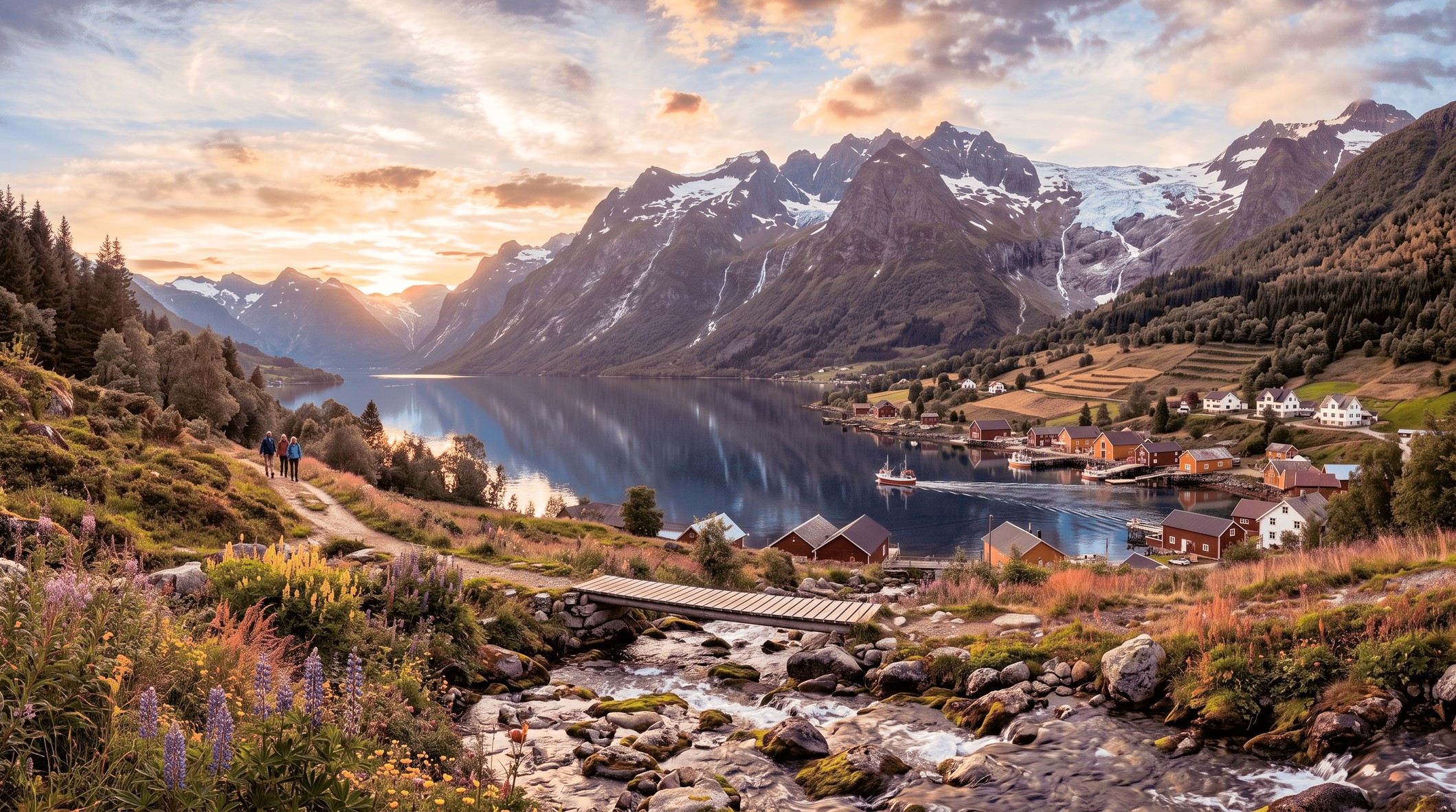}
    \end{minipage}

    \caption{Large-scale pixel-level color transfer between two $4238 \times 2365$ images using Dual BDRS. Top: original images A (left) and B (right). Bottom: transfers from A to B (left) and B to A (right).
    Each direction runs for 1000 iterations in \texttt{float32} on a single NVIDIA L40S GPU and takes approximately 35 hours. The best relative duality gaps $(U-L)/U$ are $2.41\%$ and $2.80\%$ respectively. Transferred images are obtained by barycentric projection.}
    \label{fig:color_transfer_bidirectional_4238x2365}
\end{figure}

\section{Conclusion}

We established that \ac{BDRS} is identical to \ac{IPOT} with a single inner Sinkhorn iteration and derived an equivalent dual formulation. 
This formulation reveals \ac{BDRS} as annealed Sinkhorn with log scaling momentum and a cooler kernel, while eliminating the need to store the primal transport plan, enabling linear memory.
We further introduced O-BDRS and a primal--dual certificate that can be evaluated in linear memory using the available scalings.
The certificate provides a computable stopping rule for unregularized \ac{OT} whenever the requested gap is reached.
The linear memory requirement of both the algorithm and the stopping rule thus enables unregularized \ac{OT} on dense matrices at scale.

The DOTmark experiments show that the two modifications of momentum and a cooler kernel underlying BDRS work together: neither modification alone reproduces the observed performance of the combined update. 
The matched-temperature experiments further indicate that much of O-BDRS's advantage at fixed $\eta$ is explained by faster cooling, since temperature-adjusted BDRS attains closely matching objective values. 
Together, these results clarify the relationship of unregularized \ac{OT} viewed through proximal point algorithms \citep{Xie2020ADistance}, operator splitting \citep{Ma2025BregmanMethod}, and annealed \citep{Chizat2024AnnealedDebiasing} and overrelaxed Sinkhorn \citep{Thibault2021OverrelaxedTransport} and provide a way to both compute the solution and monitor solution quality in linear memory.

\paragraph{Limitations.}
Our algebraic equivalences and certificate do not establish global convergence or an a-priori convergence rate for BDRS or O-BDRS. In particular, a valid certificate does not guarantee that a prescribed tolerance will be reached, and restricting $\lambda$ to $[1,2)$ does not by itself establish convergence. 
The ablations use the schedule intrinsic to \ac{BDRS} and O-\ac{BDRS} and do not establish superiority over independently optimized schedules. Linear memory also does not remove the $O(NM)$ pairwise work per iteration for dense problems with $N$ source and $M$ target atoms. Finally, the certificate is proved in exact arithmetic and thus its floating-point evaluation and attainable tolerances require numerical validation.

%% file: appendices/app_non_recursive_primal_update.tex
\section{Proof: Non-Recursive Primal Update}
\label{app:non-recursive-primal-update}

\nonrecursiveprimalupdate*

\begin{proof}
We prove by induction the equivalent index-shifted identity and assume the form holds for arbitrary $k$
\begin{equation*}
    X^k
    =
    \text{diag}(a^{k-1})
    K^{\odot k}
    \text{diag}(b^{k-1}),
    \qquad k\geq 0.
\end{equation*}
For the base case $k=0$, using
$K^{\odot 0}=\mathbf{1}_{m\times n}$, we have
\begin{equation*}
X^0 = \text{diag}(a^{-1}) K^{\odot 0} \text{diag}(b^{-1}) = \text{diag}(\mathbf{1}) K^{\odot 0} \text{diag}(\mathbf{1}) = \mathbf{1}_{m \times n} 
\end{equation*}
We substitute this assumed form into the intermediate core matrix $(X^k \odot K)$ which appears in the BDRS state update:
\begin{equation*}
    X^k \odot K = \left( \text{diag}(a^{k-1}) K^{\odot k} \text{diag}(b^{k-1}) \right) \odot K
\end{equation*}
Because pre- and post-multiplying by diagonal matrices only scales the rows and columns, an element-wise multiplication by $K$ passes through the diagonals and directly multiplies the inner kernel. Thus, $K^{\odot k} \odot K$ becomes $K^{\odot (k+1)}$:
\begin{equation*}
    X^k \odot K = \text{diag}(a^{k-1}) K^{\odot (k+1)} \text{diag}(b^{k-1}).
\end{equation*}
We now substitute this resolved core matrix directly into the state update \eqref{eq:bdrs_x} for the next step:
\begin{align*}
    X^{k+1} &= \text{diag}(u^k) \big( X^k \odot K \big) \text{diag}(v^k) \\
    &= \text{diag}(u^k) \Big[ \text{diag}(a^{k-1}) K^{\odot (k+1)} \text{diag}(b^{k-1}) \Big] \text{diag}(v^k).
\end{align*}
Matrix multiplication is associative, and the product of diagonal matrices is exactly the diagonal matrix of their element-wise product (i.e., $\text{diag}(a)\text{diag}(b) = \text{diag}(a \odot b)$). Regrouping the left and right terms yields:
\begin{equation*}
    X^{k+1} = \text{diag}(u^k \odot a^{k-1}) K^{\odot (k+1)} \text{diag}(b^{k-1} \odot v^k).
\end{equation*}
Applying the definition of the cumulative vectors, $a^k = a^{k-1} \odot u^k$ and $b^k = b^{k-1} \odot v^k$, we arrive at the exact unrolled form:
\begin{equation*}
    X^{k+1} = \text{diag}(a^k) K^{\odot (k+1)} \text{diag}(b^k).
\end{equation*}
This completes the induction. To show the equivalent representation, we can take the $(i,j)$th entry of
\eqref{eq:bdrs_x_non_recursive}:
\[
    X_{ij}^{k+1}
    =
    a_i^k
    \left(K_{ij}\right)^{k+1}
    b_j^k.
\]
Since $K_{ij}=\exp(-C_{ij}/\eta)$, we have
\[
    \left(K_{ij}\right)^{k+1}
    =
    \exp\left(-\frac{k+1}{\eta}C_{ij}\right).
\]
Therefore,
\[
    X_{ij}^{k+1}
    =
    \exp\left(
        \log a_i^k+\log b_j^k
        -\frac{k+1}{\eta}C_{ij}
    \right),
\]
which establishes the equivalent entrywise form.

\end{proof}
Because $K = \exp(-C/\eta)$, we have $K^{\odot (k+1)} = \exp(-(k+1) C/\eta)$. 
Lemma \ref{lem:non_recursive_primal_update} proves that the \ac{BDRS} state matrix implicitly uses an effective kernel $K^{\odot (k+1)} = \exp(-(k+1)C/\eta)$ and the recursive form of \eqref{eq:bdrs_x} can be substituted with the non-recursive form.

%% file: appendices/app_bdrs_ipot.tex
\section{BDRS Equivalence with IPOT and Relation to MDOT}
\label{app:bdrs_ipot}

The \ac{IPOT} algorithm uses the generalized proximal point method, which traditionally aims to minimize a convex objective function $f(x)$ by iteratively evaluating a proximal operator structured around a regularization term. In the context of optimal transport, IPOT tackles the exact Wasserstein distance by using a Bregman divergence—specifically, the Kullback-Leibler divergence generated by the entropy function.

By framing the primal transportation matrix as a Bregman proximal point method with the transportation polytope constraints of unregularized \ac{OT}, the subproblem at each iteration $t$ mirrors an entropic \ac{OT} problem and thus enables the use of the Sinkhorn algorithm. Solving this subproblem to convergence at every step is computationally expensive and \ac{IPOT} circumvents this by introducing an \textit{inexact} evaluation scheme that runs only a single Sinkhorn iteration ($L=1$).

Counterintuitively, this inexact primal approximation yields an update sequence that is identical to the \textit{exact} application of BDRS to the dual optimal transport problem \citep{Ma2025BregmanMethod}. BDRS requires no relaxation or truncation and its exact alternating updates produce the same sequence as IPOT with $L=1$.

\subsection{Equivalence to IPOT}
To explicitly demonstrate this link, we map the BDRS variables to their IPOT counterparts \citep[cf.][Algorithm~1]{Xie2020ADistance}: the marginals $r, c$ correspond to $\mu, \nu$; the cost kernel $K$ corresponds to $G$; the scaling vectors $u^k, v^k$ correspond to $a^{(t)}, b^{(t)}$; and the transport plan $X^k$ corresponds to $\Gamma^{(t)}$. As shown below, the single-iteration inner loop of IPOT is equivalent to the alternating dual updates of BDRS.

\vspace{1em}
\noindent
\begin{minipage}[t]{0.48\textwidth}
\textbf{BDRS (Operator Splitting)}
\begin{subequations}
\begin{align*}
    u^{k} &:= \frac{r}{(X^k \odot K)v^{k-1}} \\
    v^{k} &:= \frac{c}{(X^k \odot K)^\top u^k} \\
    X^{k+1} &:= \text{diag}(u^k)(X^k \odot K)\text{diag}(v^k)
\end{align*}
\end{subequations}
\end{minipage}
\hfill
\begin{minipage}[t]{0.48\textwidth}
\textbf{IPOT (Proximal Point Algorithm)}
\begin{subequations}
\begin{align*}
    a^{(t)} &:= \frac{\mu}{(\Gamma^{(t)} \odot G)b^{(t-1)}} \\
    b^{(t)} &:= \frac{\nu}{(\Gamma^{(t)} \odot G)^\top a^{(t)}}  \\
    \Gamma^{(t+1)} &:= \text{diag}(a^{(t)})(\Gamma^{(t)} \odot G)\text{diag}(b^{(t)})
\end{align*}
\end{subequations}
\end{minipage}
\vspace{1em}

This equivalence provides two complementary interpretations of the same matrix-scaling iteration.
Specifically, what appears as a truncated inner Sinkhorn solve with a single iteration in \ac{IPOT} is simultaneously a complete operator splitting iteration in \ac{BDRS}.
Though this equivalence was established, it is important not to immediately assume the convergence analysis can be directly applied.
The convergence theory from \citet{Xie2020ADistance} requires conditions on the accumulated error of the inner solve with Sinkhorn and are not shown to hold for the practical choice of $L=1$.
Nonetheless, we are hopeful that by drawing this connection, a convergence proof can eventually emerge in future.

\subsection{Relationship with MDOT}

Given the close connection between MDOT and IPOT \citep[cf.][Appendix~G]{Kemertas2025EfficientGradients}, the equivalence of BDRS to IPOT at $L=1$ connects BDRS and MDOT.
In particular, the unregularized \ac{OT} objective
\[
f(X)=\langle C,X\rangle
\]
is linear, so that $\nabla f(X^k)=C$. Consequently, the negative entropy \ac{MD} subproblem
\[
X^{k+1}\in\arg\min_{X\in\Pi(r,c)}
\left\{
\langle C,X\rangle
+\frac{1}{\Delta\gamma_k}
D_{\mathrm{KL}}(X|X^k)
\right\}
\]
coincides exactly with the KL-proximal subproblem underlying IPOT when $\Delta\gamma_k=1/\beta_k$. Therefore, under matched initialization and exact subproblem solutions, exact MDOT and exact IPOT trace the same entropic-regularization path, with $\gamma_{k+1}=\gamma_k+\Delta\gamma_k$.

While these three methods are similar, their equivalences should not be conflated. 
BDRS interprets the single Sinkhorn iteration as a complete operator splitting iteration and thus does not solve the subproblem exactly.
IPOT runs with a fixed $\beta$ and each subproblem is solved by a prescribed number of Sinkhorn iterations, commonly $L=1$.
MDOT specifies an inverse temperature schedule $\{\gamma_k\}$ and approximately solves each subproblem to a chosen accuracy, potentially using Sinkhorn or any convex optimization solver.
It is not surprising that these three methods are similar since it is well established that mirror descent, proximal point algorithms and operator splitting are all closely related.
The focus here is to connect \ac{BDRS}, IPOT and MDOT.

%% file: appendices/app_pure_dual_state_proof.tex
\section{Dual BDRS Extended Analysis}
\label{app:dual-bdrs-extended-analysis}

\subsection{Proof: Dual Formulation of BDRS}
\label{app:proof-pure-dual-bdrs}
In this section, we prove that BDRS can be represented in a dual formulation such that $X^k$ is not necessary in the intermediate steps and thus need not be persisted in memory. We first prove Lemma~\ref{lem:cum_duals}, and then use the results from Lemma~\ref{lem:non_recursive_primal_update} and ~\ref{lem:cum_duals} to prove Proposition~\ref{prop:dual_bdrs}.

For convenience, we reproduce the BDRS update equations for discrete \ac{OT} \citep{Ma2025BregmanMethod}:
\begin{align*}
u^{k} &:= \frac{r}{(X^k \odot K)v^{k-1}} \\
v^{k} &:= \frac{c}{(X^k \odot K)^\top u^k} \\
X^{k+1} &:= \text{diag}(u^k)(X^k \odot K)\text{diag}(v^k)
\end{align*}
where $K = \exp(-C/\eta)$ and $\odot$ denotes element-wise multiplication. In the following, we analyze the algorithm through the algebraic unrolling of the update equations. 

\begin{lemma}[Cumulative Dual Updates]
\label{lem:cum_duals}
Given the non-recursive form of $X^k$ from Lemma \ref{lem:non_recursive_primal_update}, the row and column updates in \ac{BDRS} can be rewritten as a sequence strictly over their cumulative counterparts:
\begin{subequations}
\begin{align}
    a^k &= \frac{r}{K^{\odot (k+1)} (b^{k-1} \odot v^{k-1})} \label{eq:lem_cum_u} \\
    b^k &= \frac{c}{(K^{\odot (k+1)})^\top a^k} \label{eq:lem_cum_v}
\end{align}
\end{subequations}
\end{lemma}

\begin{proof}
From Lemma~\ref{lem:non_recursive_primal_update},
\[
X^{k+1} = \text{diag}(a^{k}) K^{\odot k+1} \text{diag}(b^{k}),
\]
therefore, 
\[
X^k \odot K = \text{diag}(a^{k-1}) K^{\odot (k+1)} \text{diag}(b^{k-1})
\]
and substituting into the row update $u^k$, we obtain:
\begin{equation*}
    u^k = \frac{r}{\left[ \text{diag}(a^{k-1}) K^{\odot (k+1)} \text{diag}(b^{k-1}) \right] v^{k-1}}.
\end{equation*}
In the denominator, multiplying the diagonal matrix $\text{diag}(b^{k-1})$ by the column vector $v^{k-1}$ simplifies to the element-wise product $b^{k-1} \odot v^{k-1}$. This yields:
\begin{equation*}
    u^k = \frac{r}{\text{diag}(a^{k-1}) \left[ K^{\odot (k+1)} (b^{k-1} \odot v^{k-1}) \right]}.
\end{equation*}
Because pre-multiplying by the diagonal matrix $\text{diag}(a^{k-1})$ scales the rows of the resulting denominator vector, dividing the numerator $r$ by this scaled vector is mathematically equivalent to multiplying both sides of the equation element-wise by $a^{k-1}$:
\begin{equation*}
    a^{k-1} \odot u^k = \frac{r}{K^{\odot (k+1)} (b^{k-1} \odot v^{k-1})} \\
\end{equation*}
Recognizing by definition that $a^{k-1} \odot u^k = a^k$ yields exactly \eqref{eq:lem_cum_u}. Applying an identical substitution and factorization to the column update $v^k$ yields \eqref{eq:lem_cum_v}, completing the proof.
\end{proof}

By isolating the dynamics in Lemma \ref{lem:cum_duals}, a profound structural reality of \ac{BDRS} emerges. The intermediate matrix $X^k$ is entirely redundant, and though not obvious now, tracking the step-wise vectors ($v^{k-1}$) alongside the cumulative vectors is unnecessary. We can now formalize the complete reduction of \ac{BDRS} into a sequence of alternating projections analogous to the Sinkhorn algorithm, revealing its exact equivalence with annealed Sinkhorn.

\bdrsprop*

\begin{proof}
By definition, the step-wise dual vector $v^{k-1}$ is the element-wise ratio between the current and previous cumulative vectors:
\begin{equation}
    v^{k-1} = \frac{b^{k-1}}{b^{k-2}} \label{eq:app_v_ratio}
\end{equation}
We substitute this ratio into the target denominator term from Lemma \ref{lem:cum_duals}:
\begin{equation}
    b^{k-1} \odot v^{k-1} = b^{k-1} \odot \left( \frac{b^{k-1}}{b^{k-2}} \right) = \frac{(b^{k-1})^{\odot 2}}{b^{k-2}} \label{eq:app_target_sub}
\end{equation}
Next, we substitute the resolved term from \eqref{eq:app_target_sub} directly into the cumulative row update \eqref{eq:lem_cum_u}. This isolates the update rule so that it depends exclusively on the sequence of cumulative column vectors:
\begin{equation}
    a^k = \frac{r}{K^{\odot (k+1)} \left[ \frac{(b^{k-1})^{\odot 2}}{b^{k-2}} \right]} \label{eq:app_cum_u_isolated}
\end{equation}
The cumulative column update \eqref{eq:lem_cum_v} already relies entirely on the cumulative vectors, remaining unchanged:
\begin{equation}
    b^k = \frac{c}{(K^{\odot (k+1)})^\top a^k} \label{eq:app_cum_v_isolated}
\end{equation}
At this stage, the step-wise variables $u$ and $v$ have been eliminated from the system. Because the alternating dynamics of \eqref{eq:app_cum_u_isolated} and \eqref{eq:app_cum_v_isolated} are now entirely self-contained, the cumulative origin of $a$ and $b$ is no longer relevant to the forward progress of the algorithm, though it still inherently contains the property of being cumulative vectors. In other words, we can easily substitute $a$ and $b$ with $u$ and $v$ to obtain a form analogous to the Sinkhorn algorithm. Finally, we substitute the unrolled state equation from Lemma~\ref{lem:non_recursive_primal_update} and the proof is complete.
\end{proof}

We further verified dual \ac{BDRS} computationally and the iterates match up to machine precision with primal-dual \ac{BDRS}.
This reformulation preserves the iterates of \ac{BDRS} and warm-started \ac{IPOT} with $L=1$ and removes the need for the dense primal state.
In Appendix~\ref{app:dual_state_certificate}, when costs are evaluated implicitly, both the updates and the primal--dual certificate can be computed in linear memory, without materializing the transport plan or its feasible repair.

\subsection{Annealed Sinkhorn Detailed Comparison}
\label{app:annealed-sinkhorn-comparison}
\paragraph{Annealed Sinkhorn.} \citet{Chizat2024AnnealedDebiasing} analyzes standard annealed Sinkhorn and show that a temperature sequence of $\eta_k = \Theta(1/\sqrt{k})$ balances the entropic error and relaxation error.
Furthermore, under specific conditions, annealed Sinkhorn converges to unregularized \ac{OT}.
The sequence of updates proceeds as follows:
\begin{subequations}
\label{eq:annealed_sinkhorn}
\begin{align}
    u^{k} &= \frac{r}{K_{k-1} v^{k-1}} \label{eq:annealed_u} \\
    v^{k} &= \frac{c}{K_{k}^\top u^{k}} \label{eq:annealed_v} \\
    X^{k+1} &= \text{diag}(u^k) K_k \text{diag}(v^k). \label{eq:annealed_state}
\end{align}
\end{subequations}
Notably, if one were to prescribe an inverse linear temperature schedule where $\eta_k = \eta/(k + 1)$, then the kernel at step $k$ is given by: 
\begin{equation}
    K_k = \exp\left(-\frac{C}{\eta_k}\right) = \exp\left(-\frac{(k+1)C}{\eta}\right) = \left[\exp\left(-\frac{C}{\eta}\right)\right]^{\odot (k+1)} = K^{\odot (k+1)}, \label{eq:K_k}
\end{equation}
with $k=0$ as the initial step and the $\odot$ symbol to denote element-wise (Hadamard) operations. Because the exponential function maps $C$ element-wise, scaling $\eta$ by a factor of $\frac{1}{k+1}$ is equivalent to raising the base kernel $K$ to the $(k+1)$-th Hadamard power.
This notation reveals a crucial property: annealing the temperature inverse linearly is identical to repeatedly multiplying the kernel $K$ by itself element-wise. Under this specific schedule, substituting $K_k$ with $K^{\odot (k+1)}$ in \eqref{eq:annealed_state} is equivalent to \eqref{eq:dual_bdrs_x}.
We now compare annealed Sinkhorn with dual \ac{BDRS} side-by-side:
\par\smallskip
\noindent
\begin{minipage}[t]{0.48\textwidth}
    \centering
    \textbf{Annealed Sinkhorn}
    \begin{align*}
        u^{k} &= \frac{r}{K^{\odot k} v^{k-1}} \vphantom{\frac{r}{K^{\odot (k+1)} \left[ \frac{(v^{k-1})^{\odot 2}}{v^{k-2}} \right]}} \\
        v^{k} &= \frac{c}{(K^{\odot (k+1)})^\top u^{k}} \\
        X^{k+1} &= \text{diag}(u^k) K^{\odot (k+1)} \text{diag}(v^k)
    \end{align*}
\end{minipage}%
\hfill
\vrule
\hfill
\begin{minipage}[t]{0.48\textwidth}
    \centering
    \textbf{Dual BDRS}
    \begin{align*}
        u^{k} &= \frac{r}{K^{\odot (k+1)} \left[ \frac{(v^{k-1})^{\odot 2}}{v^{k-2}} \right]} \\
        v^{k} &= \frac{c}{(K^{\odot (k+1)})^\top u^{k}} \\
        X^{k+1} &= \text{diag}(u^k) K^{\odot (k+1)} \text{diag}(v^k)
    \end{align*}
\end{minipage}
\par\smallskip
Corollary~\ref{cor:bdrs_momentum} then immediately follows.

\paragraph{Debiased Annealed Sinkhorn.}
\citet{Chizat2024AnnealedDebiasing} further proposes \emph{Debiased Annealed Sinkhorn}, which modifies the row-scaling update to reduce the relaxation error induced by annealing. Define
\begin{equation*}
\theta_k
:=
1-\frac{\eta_{k-1}}{\eta_{(k-2)\vee 0}}
=
\begin{cases}
0, & k=1,\\[2mm]
\dfrac{1}{k}, & k\geq 2.
\end{cases}
\end{equation*}
Under this schedule, Debiased Annealed Sinkhorn becomes
\begin{align*}
u^{k}
=
\frac{
\left(u^{k-1}\right)^{\odot\theta_k}
\odot r
}{
K^{\odot k}v^{k-1}
}, \quad\quad
v^{k}
=
\frac{
c
}{
\left(K^{\odot(k+1)}\right)^\top u^{k}
}, \quad\quad
X^{k+1}
=
\operatorname{diag}(u^k)
K^{\odot(k+1)}
\operatorname{diag}(v^k).
\end{align*}
Here, \(a\vee b:=\max\{a,b\}\), so \((k-2)\vee0\) prevents the temperature schedule from being evaluated at a negative index. The factor \(\left(u^{k-1}\right)^{\odot\theta_k}\) is the debiasing term: it modifies $r$ using the previous scaling vector, with a strength determined by the relative decrease in temperature. This modification is designed to reduce the leading relaxation error introduced by annealing. At \(k=1\), \(\theta_1=0\), so the correction is inactive and the numerator reduces to \(r\).

%% file: appendices/app_a_bdrs_update.tex
\section{Overrelaxed BDRS Update and relation to BPRS}
\label{app:a-bdrs-update}
In the following, we restructure the BDRS state update:
$$X^{k+1} = \text{diag}(u^k)(X^k \odot K)\text{diag}(v^k)$$
such that the multiplicative nature in some form of an update rule with respect to $X^k$ is revealed.
Because pre-multiplying a matrix by a diagonal matrix scales its rows, and post-multiplying scales its columns, this operation is equivalent to taking the element-wise (Hadamard) product of the inner matrix with the outer product of the scaling vectors. Applying this identity yields:
\[X^{k+1} = (X^k \odot K) \odot \big( u^k (v^k)^\top \big)\]
Because the Hadamard product is both associative and commutative, we can regroup these terms to isolate the previous state:
\[X^{k+1} = X^k \odot \Big( K \odot u^k (v^k)^\top \Big)\]
In this form, the underlying mechanics of the algorithm become immediately transparent. 
The update simply takes the previous state $X^k$ and applies an element-wise multiplication against a step-wise correction term, defined as $\big(K \odot u^k (v^k)^\top\big)$.
Building directly on this geometric insight, we can introduce a parameter $\lambda \in (1, 2)$ to explicitly tune the magnitude of this correction. By element-wise exponentiating the correction term, we can produce an accelerated update:
$$X^{k+1} = X^k \odot \Big( K \odot u^k (v^k)^\top \Big)^{\odot \lambda}$$

Interestingly, this update equation has a relation to \ac{PR}. The \ac{PR} equations with half steps, extracted from Theorem 4.2 in \citet{Ma2025BregmanMethod} are given as follows:
\begin{subequations}
\begin{align}
u^{k} &:= \frac{r}{(X^k \odot K)v^{k-1}} \label{eq:bprs_u_half} \\
X^{k+\frac{1}{2}} &:= \text{diag}(u^k)(X^k \odot K)\text{diag}(v^{k-1}) \label{eq:bprs_x_half} \\
v^{k} &:= \frac{c}{(X^{k+\frac{1}{2}} \odot K)^\top u^k} \label{eq:bprs_v_half} \\
X^{k+1} &:= \text{diag}(u^k)(X^{k+\frac{1}{2}} \odot K)\text{diag}(v^k) \label{eq:bprs_x_full}
\end{align}
\end{subequations}

The half-step equations can be reduced to a single step equation as follows:
\begin{subequations}
\begin{align}
u^{k} &:= \frac{r}{(X^k \odot K)v^{k-1}} \label{eq:bprs_u_free} \\
v^{k} &:= \frac{c}{v^{k-1} \odot \left( (X^k \odot K^{\odot 2})^\top (u^k)^{\odot 2} \right)} \label{eq:bprs_v_free} \\
X^{k+1} &:= \text{diag}((u^k)^{\odot 2}) \left( X^k \odot K^{\odot 2} \right) \text{diag}(v^{k-1} \odot v^k) \label{eq:bprs_x_free}
\end{align}
\end{subequations}

It turns out that \eqref{eq:bprs_x_free} can be rewritten as follows:
\[X^{k+1} = X^k \odot \Big( K \odot u^k (v^{k-1})^\top \Big)^{\odot (\lambda-1)} \odot \Big( K \odot u^k (v^k)^\top \Big)\]
such that $\lambda=1$ yields the primal update rule for BDRS and $\lambda=2$ yields the primal update rule for BPRS.
However, the update rule that we propose, guided by the insights of faster annealing, is somewhat similar to BPRS, with the only exception of $v^{k-1}$ replaced by $v^k$ in the rewrite of \eqref{eq:bprs_x_free} to obtain:
\[X^{k+1} = X^k \odot \Big( K \odot u^k (v^{k})^\top \Big)^{\odot (\lambda-1)} \odot \Big( K \odot u^k (v^k)^\top \Big)\]

At $\lambda=1$, the proposed update reduces to BDRS. At $\lambda=2$, it differs from the BPRS state update as described above. 
This similarity motivates further investigation of methods that potentially interpolate between BDRS and BPRS.

%% file: appendices/app_pure_dual_a_bdrs_proof.tex
\section{Proof: Dual Overrelaxed BDRS}
\label{app:proof-pure-dual-a-bdrs}

Recall from Section \ref{sec:overrelaxed_bdrs} and Appendix \ref{app:a-bdrs-update} that the factored BDRS primal update isolates a step-wise correction term, $(K \odot u^k (v^k)^\top)$. We accelerate this sequence by introducing an overrelaxation parameter $\lambda \in (1, 2)$, which corresponds to the element-wise exponentiation of this correction term:
\begin{equation*}
    X^{k+1} = X^k \odot \Big( K \odot u^k (v^k)^\top \Big)^{\odot \lambda}
\end{equation*}
This is the only change for A-BDRS and the dual updates remain the same. In the following, we restate Proposition \ref{prop:dual-o-bdrs} and the proof follows.

\accbdrsprop*

\begin{proof}
We define the cumulative row and column scaling vectors exponentiated by $\lambda$ up to step $k$ as follows:
\begin{equation*}
    \alpha^k = \bigodot_{t=0}^k (u^t)^{\odot \lambda} \quad \text{and} \quad \beta^k = \bigodot_{t=0}^k (v^t)^{\odot \lambda}
\end{equation*}

Assuming the same initialization where the primal state is a matrix of ones ($X^0 = \mathbf{1}_{m \times n}$), the overrelaxed state matrix unrolls inductively to carry an accelerated kernel exponent:
\begin{equation*}
    X^k = \text{diag}(\alpha^{k-1}) K^{\odot \lambda k} \text{diag}(\beta^{k-1})
\end{equation*}

When evaluating the local row and column projections, the algorithm evaluates the core matrix $(X^k \odot K)$. Substituting the unrolled state matrix into this core evaluation yields:
\begin{equation*}
    X^k \odot K = \text{diag}(\alpha^{k-1}) K^{\odot (\lambda k + 1)} \text{diag}(\beta^{k-1})
\end{equation*}

This intermediate result is critical: it proves that the $\lambda$-exponentiated primal update forces the algorithm to evaluate its projections against a cooling kernel accelerated by $\lambda$, merging overrelaxation with an accelerated annealing schedule.

To eliminate the dense matrix $X^k$ from the algorithmic loop, we must express the step-wise projections strictly in terms of the cumulative vectors. From our definition of $\beta^k$, the previous cumulative vector is updated via $\beta^{k-1} = \beta^{k-2} \odot (v^{k-1})^{\odot \lambda}$. By isolating the step-wise multiplier $v^{k-1}$, we obtain:
\begin{equation*}
    v^{k-1} = \left( \frac{\beta^{k-1}}{\beta^{k-2}} \right)^{\odot \frac{1}{\lambda}}
\end{equation*}

We now substitute the unrolled core matrix and the isolated $v^{k-1}$ multiplier into the denominator of the standard row projection, $u^k = r / ((X^k \odot K)v^{k-1})$:
\begin{equation*}
    u^k = \frac{r}{\alpha^{k-1} \odot \left[ K^{\odot (\lambda k + 1)} \left( \beta^{k-1} \odot \left( \frac{\beta^{k-1}}{\beta^{k-2}} \right)^{\odot \frac{1}{\lambda}} \right) \right]}
\end{equation*}

To execute the cumulative update $\alpha^k = \alpha^{k-1} \odot (u^k)^{\odot \lambda}$, we must raise this entire projection to the power of $\lambda$ and multiply it element-wise with the previous state $\alpha^{k-1}$. When the denominator is exponentiated, the isolated $\alpha^{k-1}$ scaling term becomes $(\alpha^{k-1})^{\odot \lambda}$. Dividing the outer $\alpha^{k-1}$ multiplier by this exponentiated denominator naturally produces the geometric averaging term $(\alpha^{k-1})^{\odot (1-\lambda)}$. This algebraic simplification results in the final unrolled row update:
\begin{equation*}
    \alpha^{k} = (\alpha^{k-1})^{\odot (1-\lambda)} \odot \left( \frac{r}{K^{\odot (\lambda k + 1)} \left[ \beta^{k-1} \odot \left(\frac{\beta^{k-1}}{\beta^{k-2}}\right)^{\odot \frac{1}{\lambda}} \right]} \right)^{\odot \lambda}
\end{equation*}

By applying the symmetric algebraic substitutions to the standard column projection, $v^k = c / ((X^k \odot K)^{\top}u^{k})$, we obtain the complementary column update:
\begin{equation*}
    \beta^{k} = (\beta^{k-1})^{\odot (1-\lambda)} \odot \left( \frac{c}{(K^{\odot (\lambda k + 1)})^\top \left[ \alpha^{k-1} \odot \left(\frac{\alpha^{k}}{\alpha^{k-1}}\right)^{\odot \frac{1}{\lambda}} \right]} \right)^{\odot \lambda}
\end{equation*}

We note that substituting $\lambda = 1$ reduces the geometric averaging exponents $(1-\lambda)$ to zero and the momentum exponents $1/\lambda$ to one, perfectly recovering the dual BDRS formulation. Again, the choice of variables $\alpha$ and $\beta$ is arbitrary and $u$ and $v$ can be used to obtain a form that is analogous to overrelaxed and annealed Sinkhorn. At this point, the complete equations for dual O-BDRS are given as follows:
\begin{subequations}
\begin{align}
    \alpha^{k} &= (\alpha^{k-1})^{\odot (1-\lambda)} \odot \left( \frac{r}{K^{\odot (\lambda k + 1)} \left[ \beta^{k-1} \odot \left(\frac{\beta^{k-1}}{\beta^{k-2}}\right)^{\odot \frac{1}{\lambda}} \right]} \right)^{\odot \lambda} \label{eq:dual-o-bdrs-alpha} \\
    \beta^{k} &= (\beta^{k-1})^{\odot (1-\lambda)} \odot \left( \frac{c}{(K^{\odot (\lambda k + 1)})^\top \left[ \alpha^{k-1} \odot \left(\frac{\alpha^{k}}{\alpha^{k-1}}\right)^{\odot \frac{1}{\lambda}} \right]} \right)^{\odot \lambda} \label{eq:dual-o-bdrs-beta} \\
    X^{k+1} &= \text{diag}(\alpha^k) K^{\odot \lambda(k+1)} \text{diag}(\beta^k) \label{eq:dual-o-bdrs-x}
\end{align}
\end{subequations}
Next, we show how the recursion can be reduced further to the stated form. Recall the definitions
\begin{align}
K_k^{\mathrm O}
&:=
K^{\odot(\lambda k+1)},
&
\widetilde{\beta}^{k-1}
&:=
\beta^{k-1}
\odot
\left(
\beta^{k-1}\oslash\beta^{k-2}
\right)^{\odot 1/\lambda},
\notag\
\\
\widehat{\alpha}^{k}
&:=
r\oslash
\left(
K_k^{\mathrm O}\widetilde{\beta}^{k-1}
\right),
&
\widehat{\beta}^{k}
&:=
c\oslash
\left(
(K_k^{\mathrm O})^\top\widehat{\alpha}^{k}
\right).
\notag
\end{align}

Substituting $\widehat{\alpha}^k$ into equation~\ref{eq:dual-o-bdrs-alpha} directly gives
\begin{align*}
    \alpha^k
    &=
    (\alpha^{k-1})^{\odot(1-\lambda)}
    \odot
    \left[
        r\oslash
        \left(
            K_k^{\mathrm{O}}\widetilde \beta^{k-1}
        \right)
    \right]^{\odot\lambda}\\
    &=
    (\alpha^{k-1})^{\odot(1-\lambda)}
    \odot
    (\widehat \alpha^k)^{\odot\lambda},
\end{align*}
which immediately proves the first equality in
\eqref{eq:obdrs_overrelaxation}.
It remains to simplify the vector appearing inside the column update
\eqref{eq:dual-o-bdrs-beta}. Define
\begin{equation*}
    s^k
    :=
    \alpha^{k-1}
    \odot
    \left(
        \alpha^k\oslash \alpha^{k-1}
    \right)^{\odot1/\lambda}.
\end{equation*}
We show that $s^k=\widehat \alpha^k$. From the overrelaxation derived
above,
\begin{equation*}
    \alpha^k
    =
    (\alpha^{k-1})^{\odot(1-\lambda)}
    \odot
    (\widehat \alpha^k)^{\odot\lambda}.
\end{equation*}
Therefore,
\begin{align*}
    \alpha^k\oslash \alpha^{k-1}
    &=
    (\alpha^{k-1})^{\odot(1-\lambda)}
    \odot
    (\widehat \alpha^k)^{\odot\lambda}
    \oslash \alpha^{k-1}\\
    &=
    (\alpha^{k-1})^{\odot(-\lambda)}
    \odot
    (\widehat \alpha^k)^{\odot\lambda}\\
    &=
    \left(
        \widehat \alpha^k\oslash \alpha^{k-1}
    \right)^{\odot\lambda}.
\end{align*}
All entries are strictly positive, so taking the elementwise
$1/\lambda$ power yields
\begin{equation*}
    \left(
        \alpha^k\oslash \alpha^{k-1}
    \right)^{\odot1/\lambda}
    =
    \widehat \alpha^k\oslash \alpha^{k-1}.
\end{equation*}
Consequently,
\begin{align}
    s^k
    &=
    \alpha^{k-1}
    \odot
    \left(
        \widehat \alpha^k\oslash \alpha^{k-1}
    \right)
    =
    \widehat \alpha^k.
    \label{eq:obdrs_row_substitution}
\end{align}
Substituting \eqref{eq:obdrs_row_substitution} into the original column
update \eqref{eq:dual-o-bdrs-beta} gives
\begin{align*}
    \beta^k
    &=
    (\beta^{k-1})^{\odot(1-\lambda)}
    \odot
    \left[
        c\oslash
        \left(
            (K_k^{\mathrm{O}})^\top\widehat \alpha^k
        \right)
    \right]^{\odot\lambda}\\
    &=
    (\beta^{k-1})^{\odot(1-\lambda)}
    \odot
    (\widehat \beta^k)^{\odot\lambda},
\end{align*}
which immediately proves the second equality in \eqref{eq:obdrs_overrelaxation}.
Finally, when $\lambda=1$, we have
\[
    K_k^{\mathrm{O}}=K^{\odot(k+1)},
    \qquad
    \widetilde b^{k-1}
    =
    (b^{k-1})^{\odot2}\oslash b^{k-2},
\]
and \eqref{eq:obdrs_overrelaxation} reduces to
$\alpha^k=\widehat \alpha^k$ and $\beta^k=\widehat \beta^k$. These are exactly the
dual BDRS updates.
\end{proof}

%% file: appendices/app_dual_state_certificate.tex
\section{Proof and Evaluation of the Primal--Dual Certificate}
\label{app:dual_state_certificate}

We first restate Proposition~\ref{prop:dual_state_certificate} here for convenience.

\primaldualcertificate*

We prove Proposition~\ref{prop:dual_state_certificate} using the
normalization identities in \eqref{eq:cert_normalizations}, reproduced here for convenience:
\begin{equation*}
    A=r\oslash(K_\varepsilon t),
    \qquad
    B=c\oslash(K_\varepsilon^\top A).
\end{equation*}
These identities hold for both choices in \eqref{eq:cert_shared_scalings},
so the same argument applies to BDRS and O-BDRS, reproduced here for convenience:
\begin{equation*}
\begin{array}{c|cccc}
    \text{Method} & A & B & t & \varepsilon \\
    \hline
    \text{BDRS}
    & a^k & b^k & \widetilde b^{k-1} & \eta/(k+1) \\
    \text{O-BDRS}
    & \widehat\alpha^k & \widehat\beta^k
    & \widetilde\beta^{k-1} & \eta/(\lambda k+1)
\end{array}
\end{equation*}

\begin{proof}[Proof of Proposition~\ref{prop:dual_state_certificate}]
We first identify the marginals of the intermediate plan, then establish
the dual lower bound and construct a feasible plan for the primal upper
bound.

\paragraph{Intermediate marginals.}
From $A=r\oslash(K_\varepsilon t)$, the plan
$Z=\operatorname{diag}(A)K_\varepsilon\operatorname{diag}(t)$ satisfies
\begin{equation}
    Z\mathbf{1}_n=A\odot(K_\varepsilon t)=r.
    \label{eq:cert_proof_rows}
\end{equation}
Its column marginal is
\begin{equation}
    Z^\top\mathbf{1}_m
    =t\odot(K_\varepsilon^\top A)
    =c\odot t\oslash B=s,
    \label{eq:cert_proof_columns}
\end{equation}
where the final equality uses the column update for $B$.
All entries of $Z$ and $s$ are positive. Since the row marginal is $r$,
$Z$ has total mass one, and hence
\begin{equation}
    \mathbf{1}_n^\top s
    =\mathbf{1}_m^\top r
    =\mathbf{1}_n^\top c=1.
    \label{eq:cert_equal_mass}
\end{equation}

\paragraph{Dual lower bound.}
Taking the $(i,j)$th entry of $Z$ gives
\[
    Z_{ij}=A_i\exp(-C_{ij}/\varepsilon)t_j.
\]
Because $s_j=c_jt_j/B_j$, we obtain
\begin{align}
    f_i+g_j-C_{ij}
    &=\varepsilon\left(
        \log A_i+\log\frac{B_j}{c_j}
        -\frac{C_{ij}}{\varepsilon}
      \right)\nonumber\\
    &=\varepsilon\log\frac{Z_{ij}}{s_j}
    \leq 0.
    \label{eq:cert_dual_feasibility_proof}
\end{align}
The inequality follows because $s_j=\sum_i Z_{ij}$, so every entry in
column $j$ is at most its column sum. Thus, $(f,g)$ is dual feasible,
and weak duality gives $L\leq p^\star$.

\paragraph{Feasible primal repair.}
We note that this is simply a one-sided repair that can be adapted from \citet{Altschuler2017Near-linearIteration} and state this in full for convenience and for the analysis later.
Since the row marginal of $Z$ already equals $r$, we now restore the column
marginal to $c$. 
Since $s$ and $c$ have equal total mass,
\begin{equation}
    \sum_j(s_j-c_j)_+
    =\sum_j(c_j-s_j)_+
    =\frac12\|s-c\|_1=\delta,
    \label{eq:cert_surplus_deficit}
\end{equation}
where $(x)_+:=\max\{x,0\}$.
If $\delta=0$, then $s=c$ and $Z$ is already feasible; set
$\widehat Z=Z$. Suppose now that $\delta>0$.

First, clip each column whose mass exceeds its target. Define
\begin{equation}
    e_j:=\min\left\{1,\frac{c_j}{s_j}\right\},
    \qquad
    \widetilde Z:=Z\operatorname{diag}(e).
    \label{eq:cert_column_clipping}
\end{equation}
Then $\widetilde Z\leq Z$ entrywise and the mass in column $j$ becomes
$\min\{s_j,c_j\}$. This removes exactly $\delta$ units of mass.
Define the remaining row and column deficits by
\begin{equation}
    p:=r-\widetilde Z\mathbf{1}_n,
    \qquad
    q:=c-\widetilde Z^\top\mathbf{1}_m.
    \label{eq:cert_repair_deficits}
\end{equation}
Both vectors are nonnegative: column clipping only decreases the row
masses, and every clipped column has mass at most its target.
Moreover, $q_j=(c_j-s_j)_+$ and
\begin{equation}
    \mathbf{1}_m^\top p
    =\mathbf{1}_n^\top q
    =1-\sum_j\min\{s_j,c_j\}
    =\delta.
    \label{eq:cert_deficit_mass}
\end{equation}

To fill these deficits, distribute the missing mass $q_j$ in each
column among the rows in proportions $p_i/\delta$. This gives
\begin{equation}
    \widehat Z:=\widetilde Z+\frac{pq^\top}{\delta}.
    \label{eq:cert_repaired_plan}
\end{equation}
Indeed, the correction has row marginal $p$ and column marginal $q$:
\[
    \frac{pq^\top}{\delta}\mathbf{1}_n=p,
    \qquad
    \left(\frac{pq^\top}{\delta}\right)^\top\mathbf{1}_m=q.
\]
Consequently, $\widehat Z\geq 0$,
$\widehat Z\mathbf{1}_n=r$, and
$\widehat Z^\top\mathbf{1}_m=c$, so
$\widehat Z\in\Pi(r,c)$.

\paragraph{Primal upper bound.}
Let $R:=Z-\widetilde Z$ denote the mass removed by clipping and let
$T:=pq^\top/\delta$ denote the correction. Both matrices are nonnegative
and have total mass $\delta$. Writing
$C_{\min}:=\min_{i,j}C_{ij}$ and
$C_{\max}:=\max_{i,j}C_{ij}$ therefore gives
\begin{align}
    \langle C,\widehat Z-Z\rangle
    &=\langle C,T\rangle-\langle C,R\rangle\nonumber\\
    &\leq C_{\max}\delta-C_{\min}\delta\nonumber\\
    &=\operatorname{osc}(C)\delta
    \leq\Omega\delta.
    \label{eq:cert_repair_cost}
\end{align}
In words, the repair relocates $\delta$ units of mass, and relocating
one unit increases the cost by at most the range of $C$.
For $\delta=0$, the same inequality holds because $\widehat Z=Z$.

Finally, the cost of the intermediate plan is
\begin{align}
    \langle C,Z\rangle
    &=\sum_j t_j\sum_i C_{ij}A_i(K_\varepsilon)_{ij}\nonumber\\
    &=\langle t,(C\odot K_\varepsilon)^\top A\rangle
      =\langle t,w\rangle.
    \label{eq:cert_implicit_cost}
\end{align}
Since $\widehat Z$ is feasible, we conclude that
\[
    L\leq p^\star
    \leq\langle C,\widehat Z\rangle
    \leq\langle C,Z\rangle+\Omega\delta
    =U.
\]
Subtracting $p^\star$ and using $L\leq p^\star$ proves
\eqref{eq:cert_gap}.
\end{proof}

\paragraph{Evaluating the cost in the log domain.}
Key to the primal upper bound is computing the variable $w$.
While it is simple to state mathematically
\[
 w:=(C\odot K_\varepsilon)^\top A,
\]
and with the indexed form
\begin{align*}
    w_j
    &=\sum_i C_{ij}A_i\exp(-C_{ij}/\varepsilon)\\
    &=\sum_i C_{ij}
        \exp\left(\log A_i-\frac{C_{ij}}{\varepsilon}\right).
\end{align*}
Our implementation is already in the log domain, and thus $\log A_i$ is readily available.
However, directly exponentiating this expression can still cause underflow, even when the entries of $A$ sum to one and the costs lie in $[0,1]$. For example, consider two rows with $A_1=A_2=1/2$ and a column $j$ with $C_{1j}=C_{2j}=1$. At $\varepsilon=10^{-3}$, both exponential weights are \[ \exp\left(\log A_i-\frac{C_{ij}}{\varepsilon}\right) =\frac12 e^{-1000}, \qquad i\in\{1,2\}. \] These values underflow to zero in both single and double precision, so normalizing them by their sum would produce $0/0$. Yet their relative contributions are equal: each should receive half the weight when computing the mean cost within this column. We can preserve these relative contributions by subtracting the largest log weight before exponentiating. In this example, both shifted logarithms become zero, giving weights $(1,1)$, which normalize to $(1/2,1/2)$. 
This common rescaling leaves the weighted mean cost unchanged and motivates the reformulation of the computation of interest, $\langle t,w\rangle$, which we discuss next.

Recall that $s_j=c_jt_j/B_j$.
Using the column update
$B_j=c_j/(K_\varepsilon^\top A)_j$, we obtain
\[
    s_j
    =\frac{c_jt_j}{B_j}
    =t_j(K_\varepsilon^\top A)_j
    =\sum_i Z_{ij}.
\]
Thus $s_j$ is the total mass in column $j$ of the intermediate
plan $Z$.
Define $\mu_j$ as the mean transport cost within this column.
Substituting $Z_{ij}=A_i(K_\varepsilon)_{ij}t_j$ gives
\begin{align*}
    \mu_j
    &:=\frac{\sum_i C_{ij}Z_{ij}}{s_j}\\
    &=\frac{\sum_i C_{ij}A_i(K_\varepsilon)_{ij}t_j}
            {t_j(K_\varepsilon^\top A)_j}\\
    &=\frac{\sum_i C_{ij}A_i(K_\varepsilon)_{ij}}
            {(K_\varepsilon^\top A)_j}\\
    &=\frac{w_j}{(K_\varepsilon^\top A)_j}.
\end{align*}
Multiplying the mean cost by the column mass therefore gives
$s_j\mu_j=t_jw_j$.
Summing over columns, we obtain
\[
    \langle t,w\rangle
    =\langle s,\mu\rangle
    =\langle C,Z\rangle.
\]

The column masses $s$ are readily computed from the available
log scalings.
We now turn to $\mu$.
We first rewrite its numerator $w_j$ and define
\[
    \ell_{ij}:=\log A_i-\frac{C_{ij}}{\varepsilon},
    \qquad
    M_j:=\max_i\ell_{ij}.
\]
Then, for each column $j$, we subtract $M_j$ from every log weight.
Since $M_j$ does not depend on the row index $i$, we can factor
out $e^{M_j}$ to obtain
\begin{align*}
    w_j
    &=e^{M_j}\sum_i C_{ij}\exp(\ell_{ij}-M_j),\\
    (K_\varepsilon^\top A)_j
    &=\sum_i\exp(\ell_{ij})
      =e^{M_j}\sum_i\exp(\ell_{ij}-M_j),
\end{align*}
and thus achieve the necessary rescaling effect.
However, the common factor $e^{M_j}$ still carries the original
scale and can underflow to zero.
In the example above, $M_j=\log(1/2)-1000$, so
$e^{M_j}=\tfrac12e^{-1000}$ is numerically zero.
Multiplying the stabilized sums by this factor would therefore
reintroduce the same $0/0$ problem when computing $\mu_j$.
We avoid this by canceling $e^{M_j}$ algebraically between the
numerator and denominator, so this factor is never evaluated.

Starting from the definition of $\mu_j$, we obtain
\begin{align}
    \mu_j
    &=\frac{w_j}{(K_\varepsilon^\top A)_j}
    \notag\\
    &=\frac{\sum_i C_{ij}A_i(K_\varepsilon)_{ij}}
            {\sum_i A_i(K_\varepsilon)_{ij}}
    \notag\\
    &=\frac{\sum_i C_{ij}\exp(\ell_{ij})}
            {\sum_i\exp(\ell_{ij})}
    \notag\\
    &=\frac{\sum_i C_{ij}e^{M_j}\exp(\ell_{ij}-M_j)}
            {\sum_i e^{M_j}\exp(\ell_{ij}-M_j)}
    \notag\\
    &=\frac{e^{M_j}\sum_i C_{ij}\exp(\ell_{ij}-M_j)}
            {e^{M_j}\sum_i\exp(\ell_{ij}-M_j)}
    \notag\\
    &=\frac{\sum_i C_{ij}\exp(\ell_{ij}-M_j)}
            {\sum_i\exp(\ell_{ij}-M_j)}.
    \label{eq:cert_column_mean_cost}
\end{align}
Here we first expand $w_j$ and the column sum, then use
$A_i(K_\varepsilon)_{ij}=\exp(\ell_{ij})$.
Since $M_j$ is constant within column $j$, the factor $e^{M_j}$
can be taken outside both sums and canceled.
We evaluate the final expression directly, so neither $e^{M_j}$
nor $w_j$ is formed numerically.

Define the shifted column sum, which is essentially the denominator of $\mu_j$ as
\[
    D_j:=\sum_i\exp(\ell_{ij}-M_j).
\]
Observe that $1\leq D_j\leq m$, since at least one shifted exponential equals $1$ and every term lies in $[0,1]$. Thus, when the log inputs are finite, the denominator cannot underflow to zero. Moreover, $\mu_j$ is a weighted average of the costs in column $j$, and therefore
$$
\min_i C_{ij}\leq \mu_j\leq \max_i C_{ij}.
$$
The computation of $\mu$ is thus stable.

Now, we address the computation of $s$. Recall
\[
    (K_\varepsilon^\top A)_j=e^{M_j}D_j.
\]
Taking logarithms, we obtain
\begin{align}
    d_j
    &:=\log\bigl((K_\varepsilon^\top A)_j\bigr)
    \notag\\
    &=\log\left(e^{M_j}D_j\right)
    \notag\\
    &=\log\left(e^{M_j}\right)+\log D_j
    \notag\\
    &=M_j+\log D_j.
    \label{eq:cert_log_column_sum}
\end{align}
Finally, we have
\begin{align*}
    s_j
    &=t_j(K_\varepsilon^\top A)_j\\
    &=\exp\left(
        \log t_j+\log((K_\varepsilon^\top A)_j)
    \right)\\
    &=\exp(\log t_j+d_j).
\end{align*}
The transport cost is then obtained as
\[
    \langle C,Z\rangle=\sum_j s_j\mu_j.
\]
To summarize, the primal upper bound can be evaluated using
Algorithm~\ref{alg:stable_upper_bound}.
Here $A$ and $t$ are the scalings defining the row feasible
intermediate plan $Z$; for O-BDRS, $A$ is the unrelaxed row scaling.

\begin{algorithm}[t]
\caption{Log-domain evaluation of the primal upper bound}
\label{alg:stable_upper_bound}
\begin{algorithmic}[1]
\REQUIRE Log scalings $\log A,\log t$, marginal $c$,
temperature $\varepsilon>0$, access to costs $C_{ij}$,
and $\Omega\geq\operatorname{osc}(C)$.
\FOR{$j=1,\ldots,n$}
    \STATE $\ell_i \gets \log A_i-C_{ij}/\varepsilon$
    for $i=1,\ldots,m$
    \STATE $M_j \gets \max_i \ell_i$
    \STATE $E_i \gets \exp(\ell_i-M_j)$
    for $i=1,\ldots,m$
    \STATE $D_j \gets \sum_i E_i$,
    \quad $N_j \gets \sum_i C_{ij}E_i$
    \STATE $\mu_j \gets N_j/D_j$
    \STATE $d_j \gets M_j+\log D_j$
    \STATE $s_j \gets \exp(\log t_j+d_j)$
\ENDFOR
\STATE $\delta \gets \frac12\sum_j |s_j-c_j|$
\STATE $U \gets \sum_j s_j\mu_j+\Omega\delta$
\RETURN $U$
\end{algorithmic}
\end{algorithm}

%% file: appendices/app_experiments.tex
\section{Experiment Details}
\label{app:experiments}

\paragraph{Data and ground cost.}
We use the DOTmark benchmark \citep{Schrieber2016DotmarkaTransport}, which
contains ten classes of ten images at each resolution. For every class,
we evaluate all $\binom{10}{2}=45$ unordered image pairs, yielding 450
instances per resolution. Each image is normalized to have total mass
one and we do not remove zero mass pixels. For an $s\times s$ image, let
$z_i\in\{0,\ldots,s-1\}^2$ denote the coordinate of pixel $i$.
Following DOTmark, we use the squared Euclidean ground cost on the pixel grid.
We additionally normalize the ground cost by the maximum cost as follows:
\[
    C_{ij}
    =
    \frac{\lVert z_i-z_j\rVert_2^2}{2(s-1)^2}.
\]
Consequently, $C_{ij}\in[0,1]$ and $\lVert C\rVert_\infty=1$.
The main experiments use $64\times64$ instances, for which we compute the optimal $p^\star_{\mathrm{LP}}$ with the network simplex implementation in \ac{POT} \citep{Flamary2021POT:Transport} with $10^7$ maximum iterations, and verify successful termination on every instance.

\paragraph{Methods.}
Our main comparison is with \ac{AS}, reflecting the paper's focus on understanding \ac{BDRS} through its connection to annealed Sinkhorn.
We include standard \ac{AS}, \ac{AS} with a cooler kernel, \ac{AS} with momentum, and the debiased variant \citep{Chizat2024AnnealedDebiasing}.
For the large-scale color-transfer experiment, we additionally compare with \ac{MDOT-TNT} \citep{Kemertas2025ATransport}, a state-of-the-art second-order method with demonstrated scalability to approximately one million pixels per image.

\paragraph{Implementation.}
All methods use the log domain implementation for numerical stability.
Lazy pairwise reductions are only used for the large-scale experiments.
Computations use float64 and a single H200 GPU for all experiments except for the large-scale color transfer task, where we opt to use float32 and a single L40S GPU to show the broad applicability of our method.
We observed that the first run of any method normally has a longer runtime of a few milliseconds, and thus warm start the GPUs.
Following the spirit of \ac{BDRS}, where $\eta$ does not have to be tuned, we simply set $\eta=1$ for all experiments except for the color transfer task and maximum iterations to $10^5$.
For the color transfer task, we tuned $\eta$ as we wanted to obtain lower duality gaps and the details are provided in Section~\ref{app:color-transfer}.
The goal is to analyze full scale \ac{OT} and thus zero mass values in the DOTmark benchmark are not removed, and we give those pixels a small mass of $1e^{-15}$.

\subsection{Momentum and Kernel Ablation}
\label{app:momentum-kernel-ablation}
We isolate the two modifications that distinguish \ac{BDRS} from standard \ac{AS}.
We reproduce the equation in Corollary~\ref{cor:bdrs_momentum} for convenience:
\begin{equation*}
    \underbrace{
        \frac{r}{K^{\odot k}b^{k-1}}
    }_{\text{Annealed Sinkhorn}}
    \quad\longrightarrow\quad
    \underbrace{
        \frac{r}{
            K^{\odot(k+1)}
            \left[(b^{k-1})^{\odot 2}\oslash b^{k-2}\right]
        }
    }_{\text{Dual BDRS}}.
\end{equation*}
The first modification replaces the old kernel
$K^{\odot k}$ in the row update with the current kernel
$K^{\odot(k+1)}$. The second replaces $b^{k-1}$ with the
unit log-scaling extrapolation
\begin{equation*}
    \widetilde b^{k-1}
    :=
    (b^{k-1})^{\odot 2}\oslash b^{k-2}.
\end{equation*}
The resulting controlled ablations form the $2\times2$ design in
Table~\ref{tab:momentum_kernel_ablation}. We additionally compare against \ac{DAS}, whose row update under the
same inverse-linear temperature schedule is
\begin{equation*}
    u^k
    =
    \frac{
        (u^{k-1})^{\odot\theta_k}\odot r
    }{
        K^{\odot k}v^{k-1}
    },
    \qquad
    \theta_k
    =
    \begin{cases}
        0, & k=1,\\[1mm]
        1/k, & k\geq 2.
    \end{cases}
\end{equation*}
This method retains the old kernel and does not extrapolate the
incoming column scaling; instead, it modifies the row scaling target
using the previous row scaling.
Note that in this ablation, we deliberately use the inverse-linear temperature schedule that is exactly induced by dual \ac{BDRS}, rather than a schedule selected to optimize \ac{AS} or \ac{DAS}.
In particular, the inverse temperature increment does not vanish so the convergence conditions established by \citet{Chizat2024AnnealedDebiasing} do not apply to this linear schedule.
Accordingly, the experiment here should be interpreted as a controlled comparison of Table~\ref{tab:momentum_kernel_ablation} and \ac{DAS} under schedule intrinsic to \ac{BDRS}, and not a claim that \ac{BDRS} dominates annealed methods under their individually optimized schedules.
Replacing this schedule would define a different method that is no longer \ac{BDRS} or O-\ac{BDRS} studied here.
We leave the design and analysis of such schedule modified variants to future work.

In this ablation, we ran a total of eight methods: four methods in Table~\ref{tab:momentum_kernel_ablation}, \ac{DAS}, and O-\ac{BDRS} with $\lambda=[1.2, 1.5, 1.99]$ in float64 precision.
Each method was run for 100K iterations and rounding from \citet{Altschuler2017Near-linearIteration} was used at every iteration to compute $f$, the cost of the feasible primal transport plan and $f^*$ was obtained from \ac{POT}.
For wall-clock timing, the time it took to round the result was excluded.
Figure~\ref{fig:momentum-ablation-aggregated} shows that replacing the cooler kernel alone has little effect under this schedule. Adding momentum alone reduces the gap, but combining momentum with the cooler kernel, as in \ac{BDRS}, produces substantially smaller errors. In addition, \ac{DAS} \citep{Chizat2024AnnealedDebiasing} does not outperform \ac{BDRS}. O-\ac{BDRS} further lowers the aggregate rounded primal gaps at the fixed $\eta$. The corresponding classwise results are shown in Figures~\ref{fig:momentum-ablation-cauchy-classic}--\ref{fig:momentum-ablation-shapes-noise} in Appendix~\ref{app:figures-for-momentum-and-kernel-ablation}.

\begin{table}[t]
\centering
\caption{Controlled ablation of kernel timing and log-scaling
momentum. All variants share the same column update, annealing
schedule, initialization, and transport plan reconstruction formula.}
\label{tab:momentum_kernel_ablation}
\small
\begin{tabular}{lcc}
\toprule
Row-update kernel
& No momentum: $b^{k-1}$
& Momentum: $\widetilde b^{k-1}$ \\
\midrule
Old: $K^{\odot k}$
& Annealed Sinkhorn
& Annealed Sinkhorn with momentum \\
Current: $K^{\odot(k+1)}$
& Current-kernel Annealed Sinkhorn
& Dual BDRS \\
\bottomrule
\end{tabular}
\end{table}

\subsection{Matched Temperature Ablation}
\label{app:matched-temperature-ablation}

The fixed-$\eta$ comparison in Appendix~\ref{app:momentum-kernel-ablation} combines two effects: O-BDRS changes the scaling dynamics and reaches lower effective temperatures sooner. We now examine how much of its observed objective improvement remains when the kernel temperatures are matched. We use the same 450 DOTmark pairs, numerical precision, initialization, and feasibility rounding as in Appendix~\ref{app:momentum-kernel-ablation}, and consider $\lambda\in\{1.2,1.5,1.99\}$.

For BDRS and O-BDRS, respectively, the temperatures of the kernels used in the scaling updates are
\[
    \varepsilon_k^{\mathrm{B}}
    =\frac{\eta_{\mathrm{B}}}{k+1},
    \qquad
    \varepsilon_k^{\mathrm{O}}
    =\frac{\eta_{\mathrm{O}}}{\lambda k+1}.
\]
We compare the rounded feasible-plan costs using
\[
    \Delta_i
    =\left|f_{\mathrm{O},i}-f_{\mathrm{B},i}\right|,
\]
and report the mean and standard deviation of $\Delta_i$ across instances in Table~\ref{tab:matched_temperature_exact}. This measures agreement between the two objective values, rather than either method's error relative to the optimum.

\paragraph{Regime 1: common initial temperature and approximately matched effective temperatures.}
We keep $\eta_{\mathrm{B}}=\eta_{\mathrm{O}}=1$ and compare BDRS at iteration $k$ with O-BDRS at $k_{\mathrm{O}}=\lfloor k/\lambda\rfloor$. This asks whether O-BDRS can attain a comparable objective value using approximately $1/\lambda$ as many iterations. The update-kernel temperatures are close, but are exactly equal only when $\lambda k_{\mathrm{O}}=k$. In particular,
\[
    \frac{\varepsilon_{k_{\mathrm{O}}}^{\mathrm{O}}}
         {\varepsilon_k^{\mathrm{B}}}-1
    =\frac{k-\lambda k_{\mathrm{O}}}{\lambda k_{\mathrm{O}}+1},
    \qquad
    0\leq k-\lambda k_{\mathrm{O}}<\lambda.
\]
Thus, the relative temperature mismatch decreases with the iteration budget.

Table~2(a) shows that the rounded objective values become closely aligned at the later checkpoints. At $k=10^3$, the mean discrepancies are of order $10^{-4}$, whereas at $k=10^5$ they range from $2.3\times10^{-11}$ to $9.0\times10^{-10}$. For $\lambda=1.99$, the final comparison uses $50{,}251$ O-BDRS iterations against $100{,}000$ BDRS iterations. This supports an approximately twofold reduction in iteration count for comparable late-stage objective values under the common initial $\eta$; it does not by itself measure wall-clock or certified-stopping speedup.

\paragraph{Regime 2: matched terminal temperature at a common iteration budget.}
We fix $\eta_{\mathrm{O}}=1$ and lower BDRS's initial parameter to
\[
    \eta_{\mathrm{B}}
    =\frac{T+1}{\lambda T+1},
    \qquad T=10^5.
\]
Both methods are evaluated at the same iteration $k$, and their kernel temperatures agree exactly at the terminal checkpoint $T$. Their earlier temperature schedules are not identical. This comparison asks whether adjusting BDRS's initial temperature reproduces the objective improvement seen with overrelaxation.

Table~2(b) shows very small objective discrepancies, with terminal means ranging from $1.8\times10^{-11}$ to $1.8\times10^{-9}$. The agreement is already close at the earlier checkpoints, although the discrepancy and its variation are larger for $\lambda=1.99$ than for the smaller relaxation parameters at several checkpoints. These results indicate that much of O-BDRS's fixed-$\eta$ objective advantage can be reproduced by choosing a smaller initial temperature for BDRS.

\paragraph{Interpretation.}
The two regimes support a common explanation: O-BDRS reaches lower effective temperatures sooner, which accounts for much of its observed improvement at a fixed initial $\eta$. When $\eta$ is held common across methods, overrelaxation provides a practical reduction in iteration count. When $\eta$ is adjusted to approximately align the cooling schedules, the rounded objective values become very similar. This is an empirical observation about the tested instances and checkpoints, rather than an equivalence of transport plans, numerical stability, or certification behavior. It motivates separating temperature selection from the effect of the update dynamics in future comparisons and in the tuned color-transfer experiment below.

\begin{table}[t]
\centering
\caption{Rounded transport-cost discrepancies under temperature matching on 450 DOTmark pairs at resolution $64\times64$. Entries report the mean and standard deviation of $\Delta_i=|f_{\mathrm{O},i}-f_{\mathrm{B},i}|$, where both costs are evaluated after feasibility rounding. (a) A common initial $\eta=1$ with approximately matched kernel temperatures at $k$ and $\lfloor k/\lambda\rfloor$. (b) A smaller initial $\eta$ for BDRS, chosen to match O-BDRS's update-kernel temperature exactly at the terminal checkpoint $T=10^5$, with both methods evaluated at the same iteration.}
\label{tab:matched_temperature_exact}
\small

\textbf{(a) Common initial temperature and approximately matched effective temperatures} \\[3pt]
\begin{tabular}{r | ccc}
\toprule
\textbf{BDRS ($k$)} 
  & $\bm{\lambda = 1.2}$ \scriptsize ($k_{\text{O}} = \lfloor k / 1.2 \rfloor$) 
  & $\bm{\lambda = 1.5}$ \scriptsize ($k_{\text{O}} = \lfloor k / 1.5 \rfloor$) 
  & $\bm{\lambda = 1.99}$ \scriptsize ($k_{\text{O}} = \lfloor k / 1.99 \rfloor$) \\
\midrule
$1{,}000$ & $(1.0 \pm 0.0) \times 10^{-4}$ & $(1.0 \pm 0.0) \times 10^{-4}$ & $(7.3 \pm 9.2) \times 10^{-5}$ \\
$10{,}000$ & $(8.3 \pm 0.4) \times 10^{-8}$ & $(1.3 \pm 0.1) \times 10^{-7}$ & $(4.5 \pm 0.2) \times 10^{-8}$ \\
$50{,}000$ & $(2.2 \pm 0.5) \times 10^{-10}$ & $(3.4 \pm 0.7) \times 10^{-10}$ & $(9.3 \pm 6.5) \times 10^{-10}$ \\
$100{,}000$ & $(2.3 \pm 1.0) \times 10^{-11}$ & $(5.3 \pm 2.4) \times 10^{-11}$ & $(9.0 \pm 14.1) \times 10^{-10}$ \\
\bottomrule
\end{tabular}

\vspace{12pt}

\textbf{(b) Matched terminal temperature and equal iteration counts} \\[3pt]
\begin{tabular}{r | ccc}
\toprule
\textbf{Iteration} 
  & \textbf{Pair 1: $\bm{\lambda = 1.2}$} 
  & \textbf{Pair 2: $\bm{\lambda = 1.5}$} 
  & \textbf{Pair 3: $\bm{\lambda = 1.99}$} \\
\textbf{($k$)} 
  & \scriptsize vs BDRS($\eta = \frac{100{,}001}{120{,}001}$) 
  & \scriptsize vs BDRS($\eta = \frac{100{,}001}{150{,}001}$) 
  & \scriptsize vs BDRS($\eta = \frac{100{,}001}{199{,}001}$) \\
  & \scriptsize ($\epsilon_K = 8.33 \times 10^{-6}$) 
  & \scriptsize ($\epsilon_K = 6.67 \times 10^{-6}$) 
  & \scriptsize ($\epsilon_K = 5.03 \times 10^{-6}$) \\
\midrule
$1{,}000$ & $(2.2 \pm 0.3) \times 10^{-9}$ & $(3.5 \pm 0.8) \times 10^{-9}$ & $(5.9 \pm 11.0) \times 10^{-7}$ \\
$10{,}000$ & $(2.1 \pm 0.4) \times 10^{-10}$ & $(5.1 \pm 1.7) \times 10^{-10}$ & $(9.5 \pm 7.8) \times 10^{-10}$ \\
$50{,}000$ & $(3.5 \pm 1.3) \times 10^{-11}$ & $(9.1 \pm 4.0) \times 10^{-11}$ & $(7.8 \pm 13.3) \times 10^{-10}$ \\
$100{,}000$ & $(1.8 \pm 1.3) \times 10^{-11}$ & $(4.3 \pm 5.6) \times 10^{-11}$ & $(1.8 \pm 3.1) \times 10^{-9}$ \\
\bottomrule
\end{tabular}
\end{table}

\subsection{Effectiveness of Anytime Certificate}
\label{app:effectiveness-of-anytime-certificate}

\paragraph{Experiment Setup.}
We evaluate the practical usefulness of the anytime certificate described in Section~\ref{sec:certified-primal-dual-bound} and use an iteration budget of $10^4$ steps.
Suppose first that the optimal value $p_i^\star$ of
instance $i$ is available from an exact \ac{LP} solve.
We can then evaluate the \emph{certification delay}, which measures how many iterations are required before the duality gap achieves a prescribed tolerance $\tau$ compared with an oracle stopping rule that knows $p_i^\star$ and terminates when the optimality gap is within $\tau$.
To reflect how the certificate would be used
in practice, we retain the best bounds observed up to iteration $k$:
\begin{equation}
    U_{i,k}^{\mathrm{best}}
    :=
    \min_{t\leq k}U_{i,t},
    \qquad
    L_{i,k}^{\mathrm{best}}
    :=
    \max_{t\leq k}L_{i,t}.
\end{equation}
By weak duality,
\begin{equation}
    L_{i,k}^{\mathrm{best}}
    \leq
    p_i^\star
    \leq
    U_{i,k}^{\mathrm{best}}.
\end{equation}

If we indeed have access to $p_i^\star$, one can evaluate the optimality gap
\begin{equation}
             E_{i,k}^{\mathrm{oracle}}
    :=
    U_{i,k}^{\mathrm{best}}-p_i^\star.
\end{equation}

In practice, however, $p_i^\star$ is unavailable, and termination must
instead rely on the duality gap
\begin{equation}
    G_{i,k}^{\mathrm{cert}}
    :=
    U_{i,k}^{\mathrm{best}}-L_{i,k}^{\mathrm{best}}.
\end{equation}
The two quantities satisfy
\begin{equation}
    G_{i,k}^{\mathrm{cert}}
    =
    E_{i,k}^{\mathrm{oracle}}
    +
    \left(p_i^\star-L_{i,k}^{\mathrm{best}}\right)
    \geq
    E_{i,k}^{\mathrm{oracle}}.
\end{equation}
Consequently, $G_{i,k}^{\mathrm{cert}}\leq\tau$ guarantees that the
returned plan is within $\tau$ of the optimal transport cost, although
its true error may be smaller.

For a prescribed additive tolerance $\tau$, define the oracle and
certified stopping iterations by
\begin{align}
    k_i^{\mathrm{oracle}}(\tau)
    &:=
    \min\left\{
        k:
        E_{i,k}^{\mathrm{oracle}}\leq\tau
    \right\},\\
    k_i^{\mathrm{cert}}(\tau)
    &:=
    \min\left\{
        k:
        G_{i,k}^{\mathrm{cert}}\leq\tau
    \right\}.
\end{align}
Because the certified gap upper bounds the true optimality gap,
\begin{equation}
    k_i^{\mathrm{cert}}(\tau)
    \geq
    k_i^{\mathrm{oracle}}(\tau).
\end{equation}
We measure the conservatism of the certificate through the
\emph{certification delay}
\begin{equation}
    D_i(\tau)
    :=
    \frac{
        k_i^{\mathrm{cert}}(\tau)
    }{
        k_i^{\mathrm{oracle}}(\tau)
    }
    \geq 1.
\end{equation}
A value of $D_i(\tau)=1$ means that the certificate identifies the
first iterate satisfying the requested accuracy, while
$D_i(\tau)=2$ means that certification requires twice as many
iterations as an oracle stopping rule.
We report the empirical cumulative distribution of
$D_i(\tau)$ over the DOTmark instances for
$\tau\in\{10^{-2},10^{-3},10^{-4}\}$. 
The iteration budget used here is $10^{5}$.
If the oracle reaches the requested
tolerance but the certificate does not do so within the terminal
budget, we set $D_i(\tau)=+\infty$.

\paragraph{Results.}
We supplement the analysis in the main text with Figure~\ref{fig:certificate-delay-by-class}, where it shows how the behavior of the delay certification varies across image classes. The reduction in delay at $\tau=10^{-2}$ is visible throughout the benchmark, whereas the $\tau=10^{-3}$ distributions are much closer for some classes, including GRF Smooth and Shapes. These results show that the delay depends on the method, tolerance, and problem class.

\begin{figure}
    \centering
    \includegraphics[width=\linewidth]{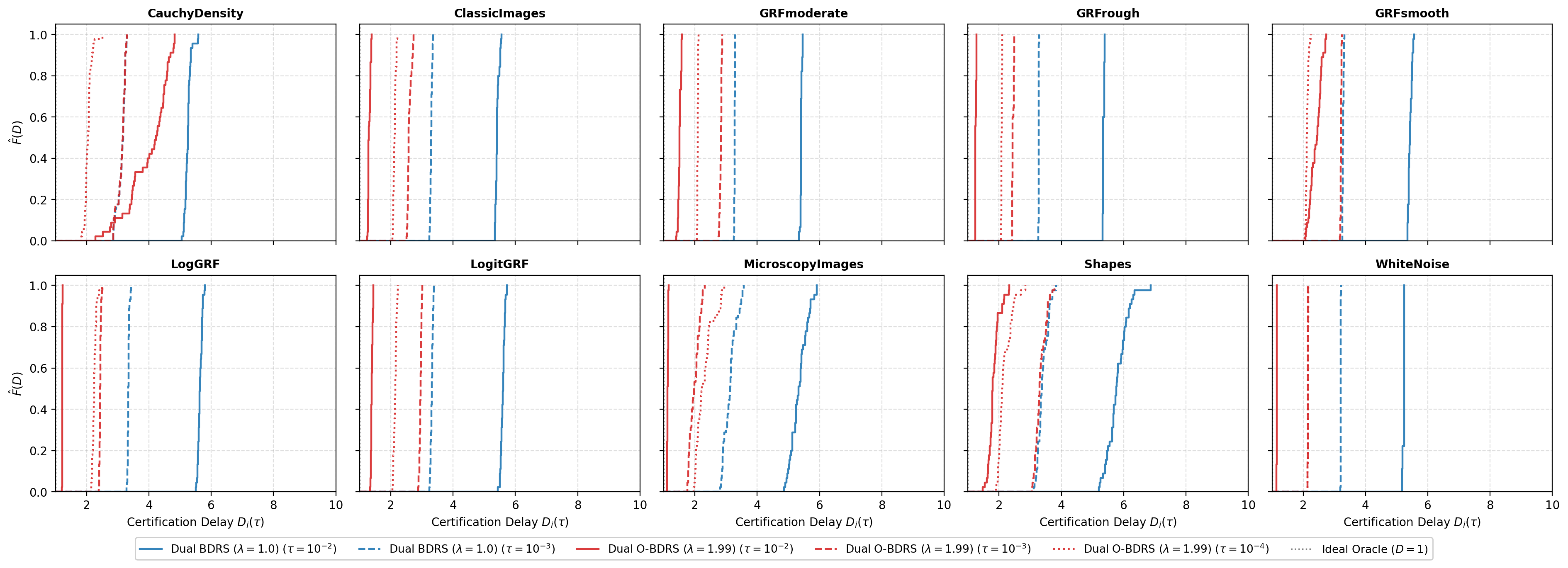}
    \caption{Classwise certification-delay distributions for Dual BDRS ($\lambda=1$) and Dual O-BDRS ($\lambda=1.99$) on DOTmark. Each panel contains 45 image pairs at resolution $64\times64$. Solid and dashed curves correspond to $\tau=10^{-2}$ and $\tau=10^{-3}$, respectively; dotted red curves additionally show O-BDRS at $\tau=10^{-4}$, with BDRS omitted as certification was not obtained within the iteration budget. The horizontal axis is $D_i(\tau)=k_i^{\mathrm{cert}}(\tau)/k_i^{\mathrm{oracle}}(\tau)$ and the vertical axis is its empirical cumulative distribution. The reference $D_i=1$ denotes agreement with the oracle stopping iteration. Curves further to the left indicate less additional iteration cost relative to each method's own oracle.}
    \label{fig:certificate-delay-by-class}
\end{figure}

\subsection{Large-scale color transfer}
\label{app:color-transfer}
The preceding experiments examine the update structure and certification behavior on problems for which a reference optimum is available. We now consider a large scale problem in which storing the transport plan becomes the main obstacle. 
To this end, pixel-level color transfer allows us to assess two practical consequences of the dual formulation: the time required to reach a prescribed certified accuracy and the problem sizes that become accessible with a linear memory implementation.

The matched temperature results in Appendix~\ref{app:matched-temperature-ablation} motivate using Dual BDRS with $\lambda=1$ as the primary solver while tuning its initial temperature for this application. This evaluates the performance of the tuned solver, whereas the fixed $\eta$ DOTmark study isolates the effect of overrelaxation and targets the general case where tuning is unnecessary. 
Our initial experiments compares against \ac{MDOT-TNT}~\citep{Kemertas2025ATransport} at a common approximate transport cost.
Following that, we run large-scale \ac{OT} on $4238\times2365$ resolution images, comprising a total of $10^7$ pixels, establishing a new state-of-the-art for large-scale \ac{OT}.

\paragraph{Pixel-level OT formulation.}
Let $I_s$ denote the source image whose spatial arrangement is retained and $I_p$ the palette image whose colors are transferred to $I_s$. After flattening the images, let $x_i=\operatorname{RGB}(I_s,i)\in[0,255]^3$ for $i=1,\ldots,N$, and $y_j=\operatorname{RGB}(I_p,j)\in[0,255]^3$ for $j=1,\ldots,M$. Each pixel has uniform mass, so
\[
    r_i=\frac{1}{N},\qquad c_j=\frac{1}{M},\qquad
    C_{ij}=\frac{\|x_i-y_j\|_2^2}{\max_{p,q}\|x_p-y_q\|_2^2},
\]
and $C_{ij}$ is normalized by the maximum cost, thus $C_{ij}\in[0,1]$. Every pixel in each input image is retained as a separate atom, including pixels with repeated RGB values. 
We do not aggregate equal colors or replace pixels with quantized color representatives as we want to evaluate the solver without compression.
At each tested resolution, the solver treats the complete source and palette images as one global OT problem with $NM$ implicit transport entries.

\paragraph{Implicit implementation and output.}
We use log-domain Dual BDRS and evaluate pairwise reductions lazily with PyKeOps. The cost matrix, kernel, and transport plan are never materialized, including during certificate evaluation and image reconstruction. The large-scale implementation uses \texttt{float32} on a single NVIDIA L40S GPU. 
We log all primal upper and dual lower bounds to show the trajectory of the algorithm but retain the state with the lowest relative duality gap for barycentric projection.
The transferred image is obtained by barycentric projection of the chosen implicit output plan $X$:
\[
    \widehat y_i
    =\frac{\sum_{j=1}^{M}X_{ij}y_j}{\sum_{j=1}^{M}X_{ij}}.
\]
The numerator and denominator are evaluated using lazy pairwise reductions. We then clip the RGB values to $[0,255]$ and reshape them to the source image dimensions.
In our experiments, both row and column residual have extremely low errors and thus we do not round the transport plan before barycentric projection.
Therefore, the visual output is \emph{not} the transformation from a feasible transport plan, but a means for us to visualize the optimality of the transport plan.
The primal-dual guarantee applies to the transport cost of the repaired plan, if feasible solutions are required.

\paragraph{$\eta$ tuning.}
We first use $512\times512$ instances to select the initial temperature.
Specifically, we started with $\eta=1$, and gradually reduced the temperature until we observe oscillations in the primal-dual gap.
We found that $\eta=0.1$ is approximately the limit where oscillations start to occur.
In the spirit of implicit annealing in \ac{BDRS}, we did not further tune $\eta$ and $\eta=1$ gives a duality gap curve that resembles a monotonically decreasing curve.
At $\eta < 0.1$, we start to observe some oscillations in the duality gap.
Despite that, lower $\eta$ values can still arrive at a given optimality gap with fewer iterations.

\paragraph{Comparison with existing state-of-the-art.}
We compare \ac{BDRS} with the released low-memory implementation of \ac{MDOT-TNT} \citep{Kemertas2025ATransport} on the same pair of $1024\times1024$ images, using all pixels and the same normalized squared Euclidean cost in RGB space.
\citet{Kemertas2025ATransport} report approximately 10 hours for their \texttt{float64} PyKeOps implementation at a final inverse temperature of $\gamma=2^{10}$, without reporting the corresponding linear transport cost.
We ran the provided code in \texttt{float64} with $\gamma=2^{10}$ and obtained a solution within 7 hours on a single L40S GPU.
We also ran the provided code in \texttt{float32} with $\gamma=2^{10}$ and obtained a solution within 3.5 hours on a single L40S GPU.
The final costs obtained on both runs was identical.
In the following experiments, we use \texttt{float32} to reflect operational realities where speed is critical.

To this end, we evaluate the baseline directly, with $\gamma=2^{5}$, followed by $\gamma=2^{10}$.
We choose these two settings to represent different computational budgets: $\gamma=2^{5}$ targets scenarios requiring a solution quickly, while $\gamma=2^{10}$ targets those permitting longer computations, such as overnight runs.
With $\gamma=2^{5}$, \ac{MDOT-TNT} takes $1685.76$s of solver time and yields an unrepaired transport cost of $0.072147$.
Our \texttt{float32} \ac{BDRS} implementation, with $\eta=0.1$ and $\lambda=1$, completes 10 iterations in $13.55$s of solver time, or $15.71$s including setup and certificate evaluation.
The primal upper bound at iteration 10 is $0.070403$, which is less than the unrepaired transport cost of \ac{MDOT-TNT}.
We recognize that the baseline's unreported marginal residuals limit the interpretation of this cost comparison, and therefore focus on the visual results in Figure~\ref{fig:color_transfer_comparison}, which clearly show \ac{BDRS} with 10 iterations yielding a sharper result than \ac{MDOT-TNT}.

\begin{figure}[t]
    \centering
    \begin{minipage}[t]{0.32\linewidth}
        \centering
        \includegraphics[width=\linewidth]{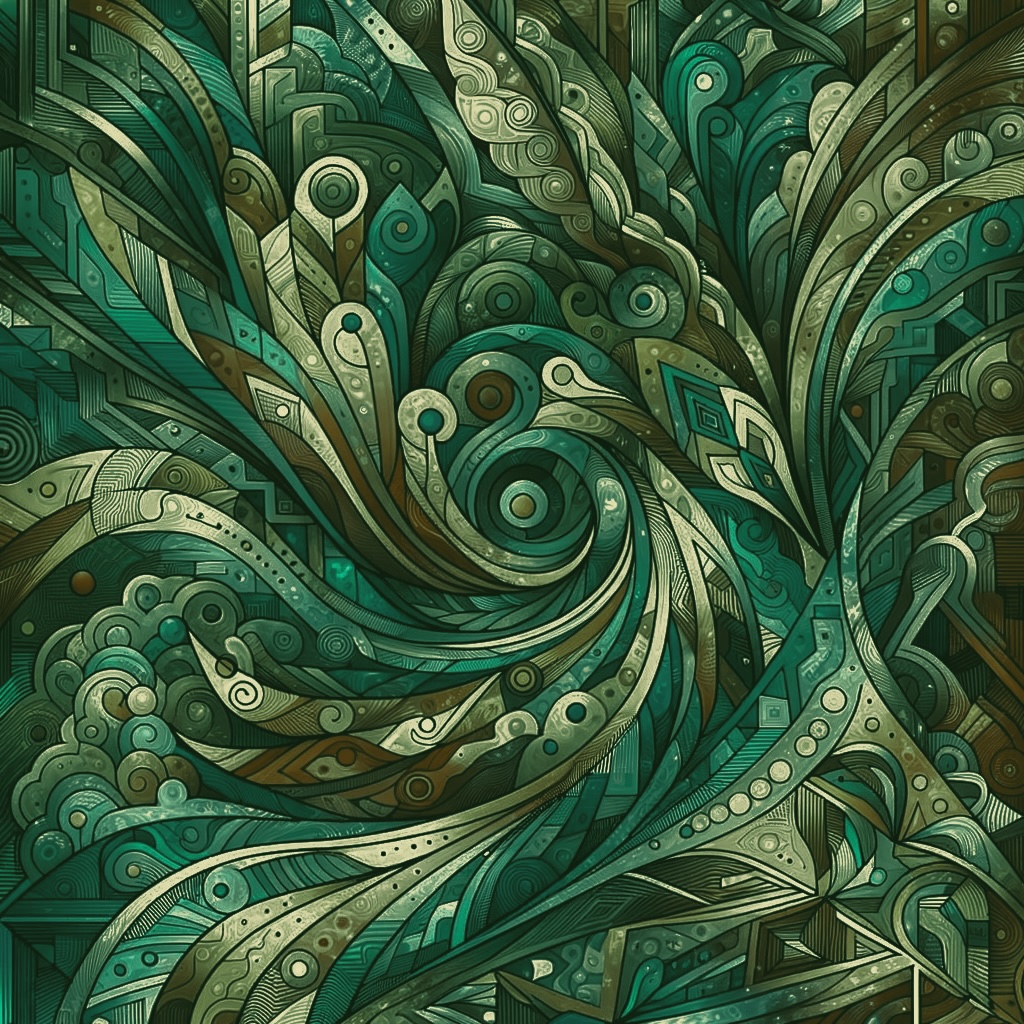}
        \par\smallskip
        {\small\textbf{(a) BDRS: 10 iterations}}
    \end{minipage}
    \hfill
    \begin{minipage}[t]{0.32\linewidth}
        \centering
        \includegraphics[width=\linewidth]{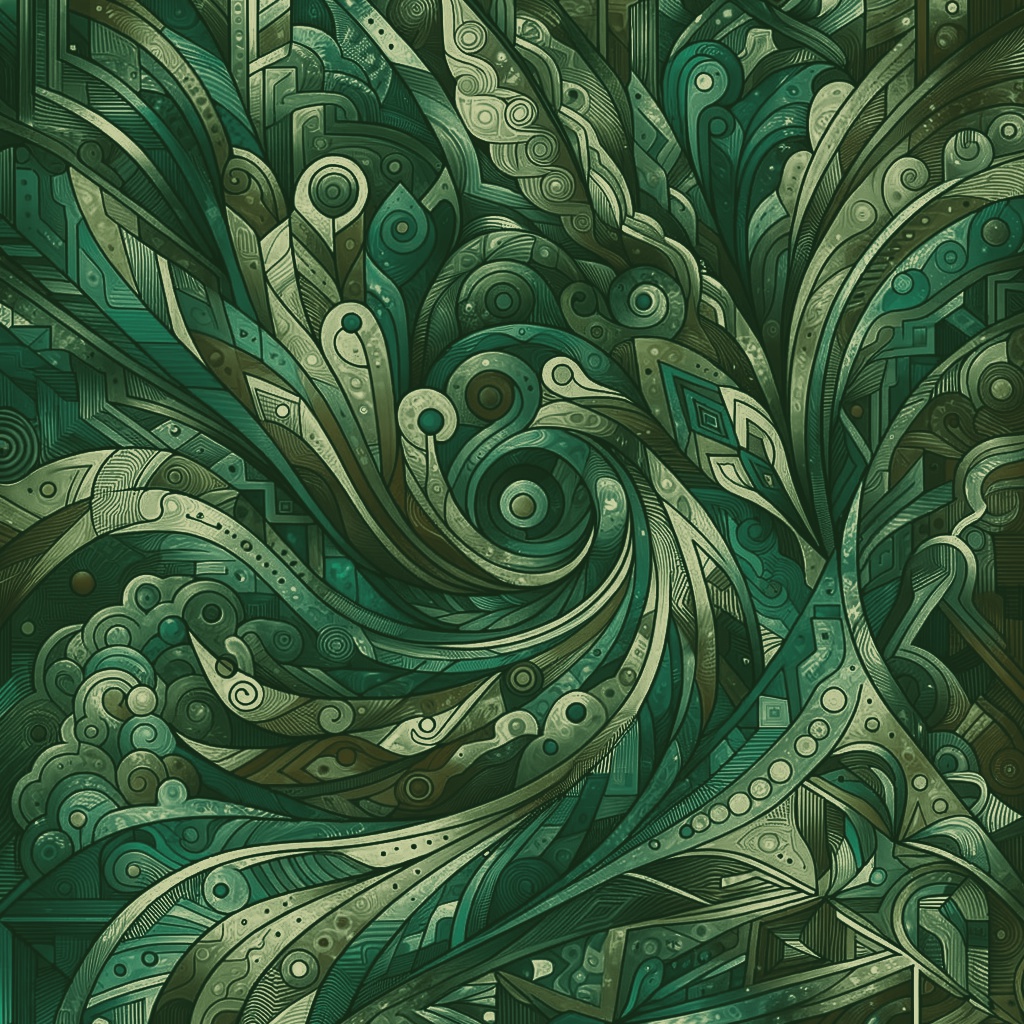}
        \par\smallskip
        {\small\textbf{(b) MDOT-TNT}}
    \end{minipage}
    \hfill
    \begin{minipage}[t]{0.32\linewidth}
        \centering
        \includegraphics[width=\linewidth]{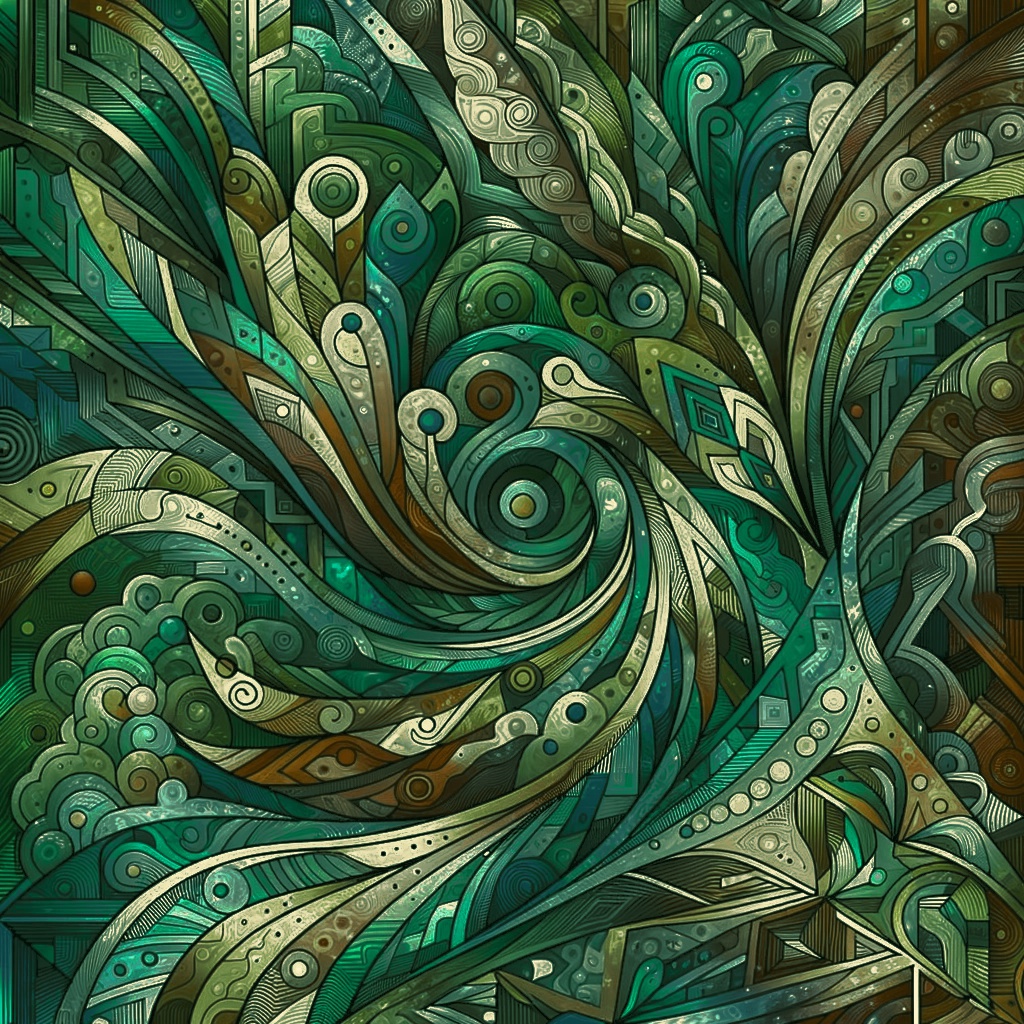}
        \par\smallskip
        {\small\textbf{(c) BDRS reference}}
    \end{minipage}
    \caption{
        Visual comparison on the same $1024\times1024$ color-transfer
        problem. All panels show barycentric reconstructions using
        all source and target pixels, \texttt{float32}, and the same
        normalized squared Euclidean cost in RGB space.
        \textbf{(a)} BDRS after 10 iterations, requiring $13.55$\,s
        of solver time ($15.71$\,s including setup and certificate
        evaluation).
        \textbf{(b)} MDOT-TNT at final inverse temperature
        $\gamma=2^5$, requiring $1685.76$\,s of solver time.
        \textbf{(c)} BDRS after 1,000 iterations, used as a reference
        solution with a reported relative duality gap of $1.59\%$.
        BDRS uses $\eta=0.1$ and $\lambda=1$.
        Solver times exclude image reconstruction.
        Observe how \ac{BDRS} with 10 iterations is closer to the reference, where the brown parts of the image are more pronounced, whereas \ac{MDOT-TNT} is still generally green, an indication of a diffuse transport plan.
        On the 10 iterations comparison, \ac{BDRS} provides a 107$\times$ speedup over \ac{MDOT-TNT}.
        Best viewed digitally.
    }
    \label{fig:color_transfer_comparison}
\end{figure}

We then compare \ac{BDRS} with \ac{MDOT-TNT} at $\gamma=2^{10}$.
We first run \ac{BDRS} for 1,000 iterations to assess both optimization accuracy and marginal feasibility.
We ran both the forward transfer and the backward transfer and Table~\ref{tab:color_transfer_accuracy} summarizes the bounds and marginal errors.
For the unrepaired implicit plan used in barycentric reconstruction, the row and column marginal $\ell_1$ errors are both below $7.34\times10^{-5}$.
Each source and target pixel has a prescribed mass of $1/2^{20}\approx9.54\times10^{-7}$.
The maximum absolute row and column marginal errors are $2.06\times10^{-10}$ and $7.87\times10^{-10}$, respectively.
Dividing these errors by the prescribed pixel mass gives maximum relative errors of $0.0216\%$ and $0.0826\%$.
Thus, every row and column mass deviates from its prescribed value by less than $0.083\%$, indicating that the unrepaired transport plan is approximately feasible in practice.
\ac{BDRS} requires only 156MB and takes approximately 24 minutes to run to achieve a visual similarity, shown in Figure~\ref{fig:color_transfer_bidirectional}, that is indistinguishable to the naked eye from what was reported by \citet{Kemertas2025ATransport}.
The matched run on the same hardware of \ac{MDOT-TNT} also confirms this result, where it ran for 3.5 hours and returned a repaired cost of $0.0507$.
\ac{BDRS} thus achieves a $9\times$ speed up, while achieving a lower cost of $0.0497$ than \ac{MDOT-TNT}.

\begin{figure}[t]
    \centering

    \begin{subfigure}[t]{0.48\linewidth}
        \centering
        \includegraphics[width=\linewidth]{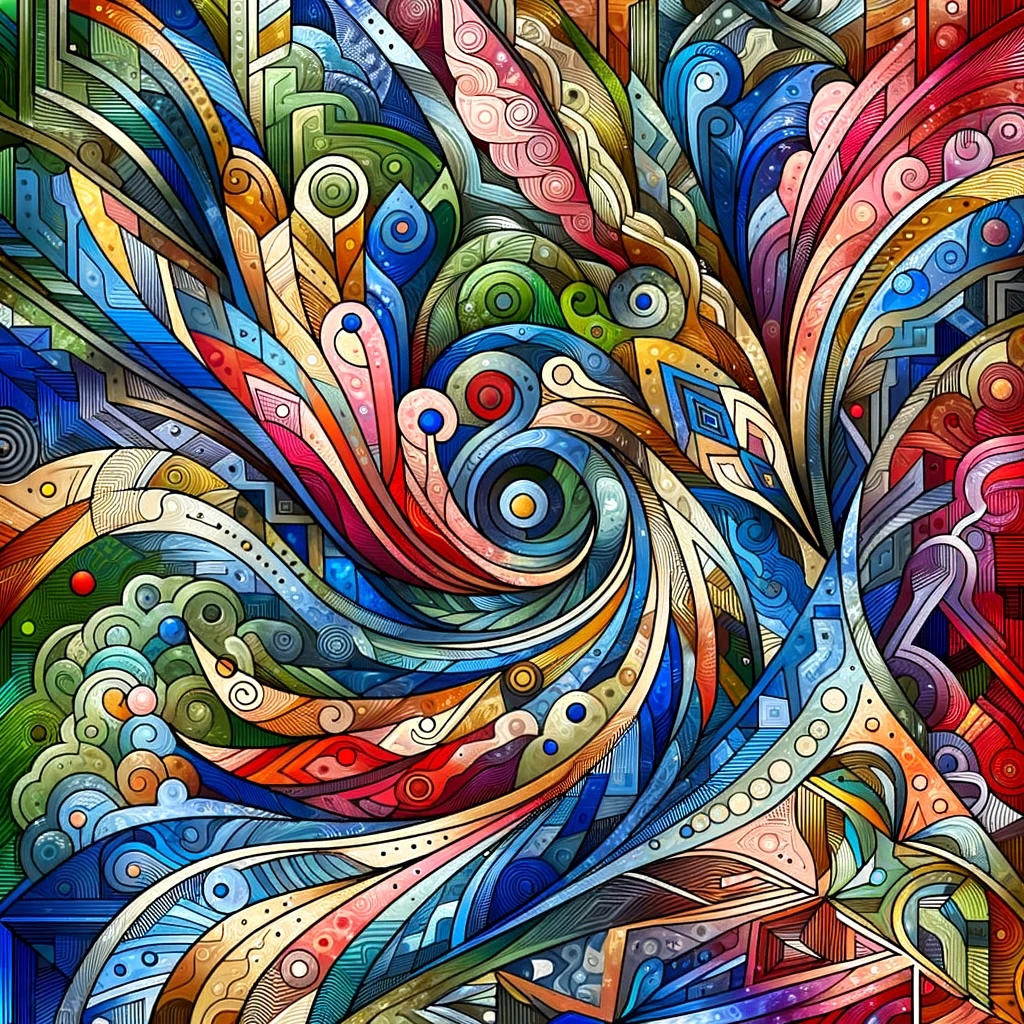}
        \caption{Original Image A}
        \label{fig:color_original_1}
    \end{subfigure}
    \hfill
    \begin{subfigure}[t]{0.48\linewidth}
        \centering
        \includegraphics[width=\linewidth]{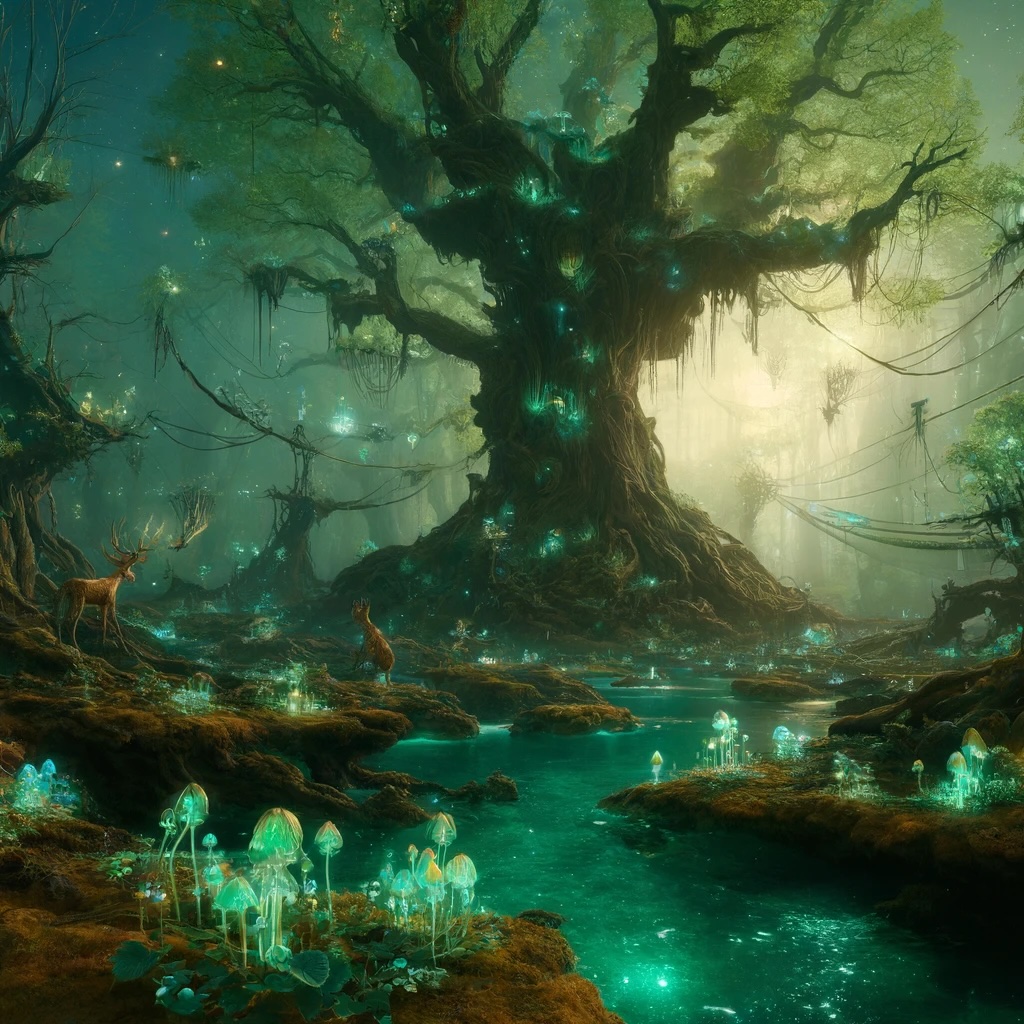}
        \caption{Original Image B}
        \label{fig:color_original_2}
    \end{subfigure}

    \par\medskip

    \begin{subfigure}[t]{0.48\linewidth}
        \centering
        \includegraphics[width=\linewidth]{assets/large-scale/kemertas1-1024x1024_to_kemertas2-1024x1024_eta0.1_lam1_iters1000_float32_application_cert20.jpg}
        \caption{Forward transfer (A to B).}
        \label{fig:color_forward}
    \end{subfigure}
    \hfill
    \begin{subfigure}[t]{0.48\linewidth}
        \centering
        \includegraphics[width=\linewidth]{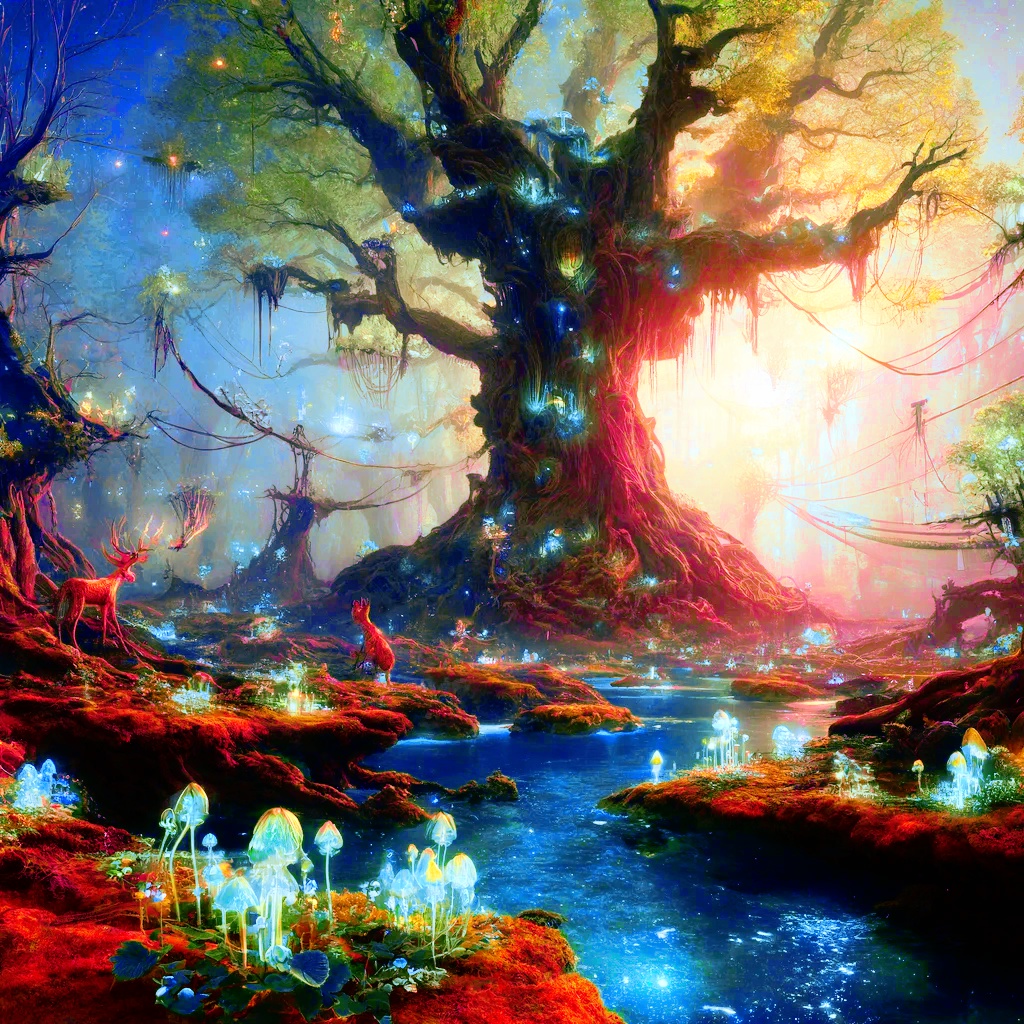}
        \caption{Backward transfer (B to A).}
        \label{fig:color_backward}
    \end{subfigure}

    \caption{
        Bidirectional color transfer between two $1024\times1024$
        images using BDRS.
        \textbf{(a, b)} Original images.
        \textbf{(c)} Image 1 transferred to the color distribution
        of image 2.
        \textbf{(d)} Image 2 transferred to the color distribution
        of image 1.
        Both runs use all source and target pixels and perform
        1,000 iterations in \texttt{float32}, with $\eta=0.1$
        and $\lambda=1$.
        The reported relative duality gaps are $1.59\%$ and
        $1.58\%$ for the forward and backward directions,
        respectively.
        Each direction ran in approximately 24 minutes with 156MB of memory.
        Transferred images are obtained by barycentric projection
        of the unrepaired transport plans.
    }
    \label{fig:color_transfer_bidirectional}
\end{figure}

\begin{table}[t]
    \centering
    \small
    \caption{
        Accuracy, runtime, and memory of BDRS after 1,000 iterations
        for both $1024\times1024$ color transfer directions in
        \texttt{float32}.
        Marginal errors refer to the unrepaired plan; paired entries
        are ordered as (row, column).
        Relative marginal errors are measured against the prescribed
        pixel mass $2^{-20}$.
        Runtime includes setup and certificate evaluation but excludes
        image reconstruction.
        Memory is the peak GPU allocation tracked by PyTorch,
        including reconstruction and validation, and excludes
        allocations outside PyTorch.
    }
    \label{tab:color_transfer_accuracy}
    \begin{tabular}{@{}lrr@{}}
        \toprule
        Statistic & Forward ($1\to2$) & Backward ($2\to1$) \\
        \midrule
        Primal upper bound $U$
            & $0.0496885$
            & $0.0496843$ \\
        Dual lower bound $L$
            & $0.0488973$
            & $0.0488973$ \\
        Relative duality gap $(U-L)/U$
            & $1.59\%$
            & $1.58\%$ \\
        Marginal $\ell_1$ error
            & $(2.94\times10^{-5},\,6.35\times10^{-5})$
            & $(2.54\times10^{-5},\,7.34\times10^{-5})$ \\
        Maximum absolute marginal error
            & $(2.06\times10^{-10},\,7.87\times10^{-10})$
            & $(2.06\times10^{-10},\,7.87\times10^{-10})$ \\
        Maximum relative marginal error
            & $(0.0216\%,\,0.0826\%)$
            & $(0.0216\%,\,0.0826\%)$ \\
        Runtime including setup and certification
            & $1433.06$\,s
            & $1421.04$\,s \\
        Peak allocated GPU memory
            & $156.13$\,MiB
            & $156.13$\,MiB \\
        \bottomrule
    \end{tabular}
\end{table}

\ac{LAMP}~\citep{Burns2026Log-AveragedSpace} also provides a linear-memory approach to \ac{OT}.
We attempted to run the authors' Julia implementation but encountered a GPU kernel compilation error (\texttt{InvalidIRError}) in our environment.
Our best-effort PyTorch reimplementation was tractable for color transfer between $512\times512$ images, but not between $1024\times1024$ images in our setup.
We therefore omit a quantitative comparison with \ac{LAMP}.
These implementation difficulties do not establish a scalability limit for \ac{LAMP}. 
The absence of a direct comparison with a validated \ac{LAMP} implementation remains a limitation of our evaluation.

\paragraph{Scaling pixel-level \ac{OT} to ten million atoms per image.}
We further apply Dual BDRS to bidirectional color transfer between two $4238\times2365$ images. Each image contains $10{,}022{,}870$ pixels, yielding a single global \ac{OT} problem with approximately $10^{14}$ implicit transport entries. We use $\eta=0.01$, $\lambda=1$, and \texttt{float32} pairwise computations on a single NVIDIA L40S GPU, evaluating the certificate every 20 iterations. Both directions complete 1,000 iterations in approximately 35 hours each, with a peak allocated GPU memory of $1.58$\,GiB. The selected iterates, both at iteration 840, attain relative primal--dual gaps of $2.41\%$ and $2.80\%$ for the forward and backward transfers, respectively. Figure~\ref{fig:color_transfer_bidirectional_4238x2365} in the main text shows the barycentric reconstructions from the unrepaired implicit plans, while Table~\ref{tab:large_color_transfer_accuracy} reports the bounds, marginal errors, runtimes, and memory usage. The cost matrix, kernel, and transport plan are never materialized.

\begin{table}[t]
    \centering
    \small
    \caption{
        Accuracy, runtime, and memory of BDRS for both
        $4238\times2365$ color transfer directions, with
        $\eta=0.01$, $\lambda=1$, and \texttt{float32}
        pairwise computations.
        Images 1 and 2 are the canyon and outdoor images, respectively.
        Each run completes 1,000 iterations; reported accuracy
        corresponds to the selected iterate at iteration 840.
        Marginal errors refer to the unrepaired intermediate plan;
        paired entries are ordered as (row, column).
        Relative marginal errors are measured against the prescribed
        pixel mass $1/10{,}022{,}870$.
        Runtime covers the full run, including setup, certificate
        evaluation, and final verification, but excludes image
        reconstruction.
        Memory is the peak GPU allocation tracked by PyTorch,
        including reconstruction and validation, and excludes
        allocations outside PyTorch.
    }
    \label{tab:large_color_transfer_accuracy}
    \begin{tabular}{@{}lrr@{}}
        \toprule
        Statistic & Forward ($1\to2$) & Backward ($2\to1$) \\
        \midrule
        Primal upper bound $U$
            & $0.0101706$
            & $0.0102118$ \\
        Dual lower bound $L$
            & $0.00992558$
            & $0.00992557$ \\
        Relative duality gap $(U-L)/U$
            & $2.41\%$
            & $2.80\%$ \\
        Marginal $\ell_1$ error
            & $(1.16\times10^{-4},\,3.29\times10^{-4})$
            & $(1.17\times10^{-4},\,4.28\times10^{-4})$ \\
        Maximum absolute marginal error
            & $(4.68\times10^{-11},\,1.50\times10^{-9})$
            & $(4.68\times10^{-11},\,1.34\times10^{-9})$ \\
        Maximum relative marginal error
            & $(0.0469\%,\,1.5023\%)$
            & $(0.0469\%,\,1.3414\%)$ \\
        Runtime including setup and certification
            & $125582.11$\,s
            & $124820.21$\,s \\
        Peak allocated GPU memory
            & $1613.97$\,MiB
            & $1613.97$\,MiB \\
        \bottomrule
    \end{tabular}
\end{table}

\paragraph{Floating-point evaluation.}
The reported primal and dual bounds are floating-point evaluations of the exact arithmetic certificate.
In the largest experiments, \texttt{float32} computations may introduce errors in normalization and certificate evaluation.
The reported marginal residuals quantify mass-conservation errors, but do not bound the numerical errors in the certificate itself.

%% file: appendices/app_experiments_momentum_kernel_ablation_figures.tex
\section{Figures for Momentum and Kernel Ablation}
In this section, we show all ten graphs for each class in the DOTmark dataset.
We put all figures for the momentum and kernel ablation in a single section to improve readability.

\label{app:figures-for-momentum-and-kernel-ablation}
\begin{figure}[tbp]
    \centering
    \includegraphics[width=\textwidth]{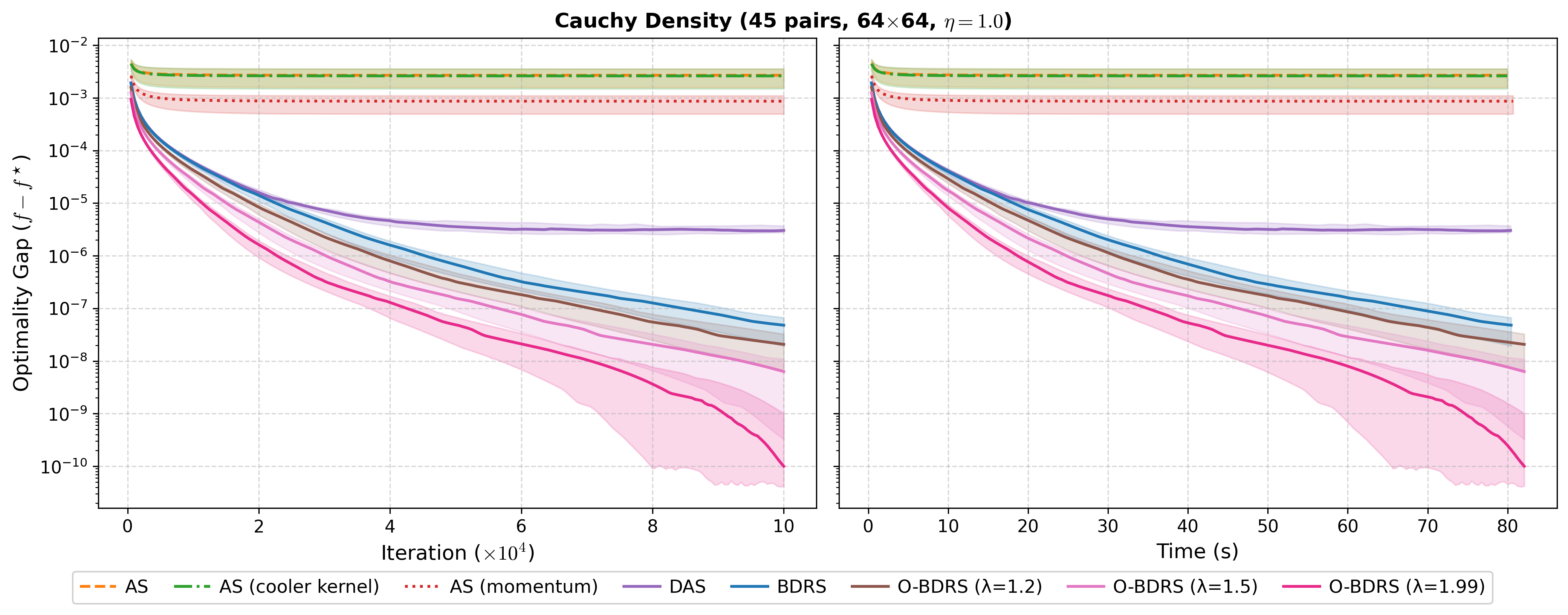}
    \par\medskip
    \includegraphics[width=\textwidth]{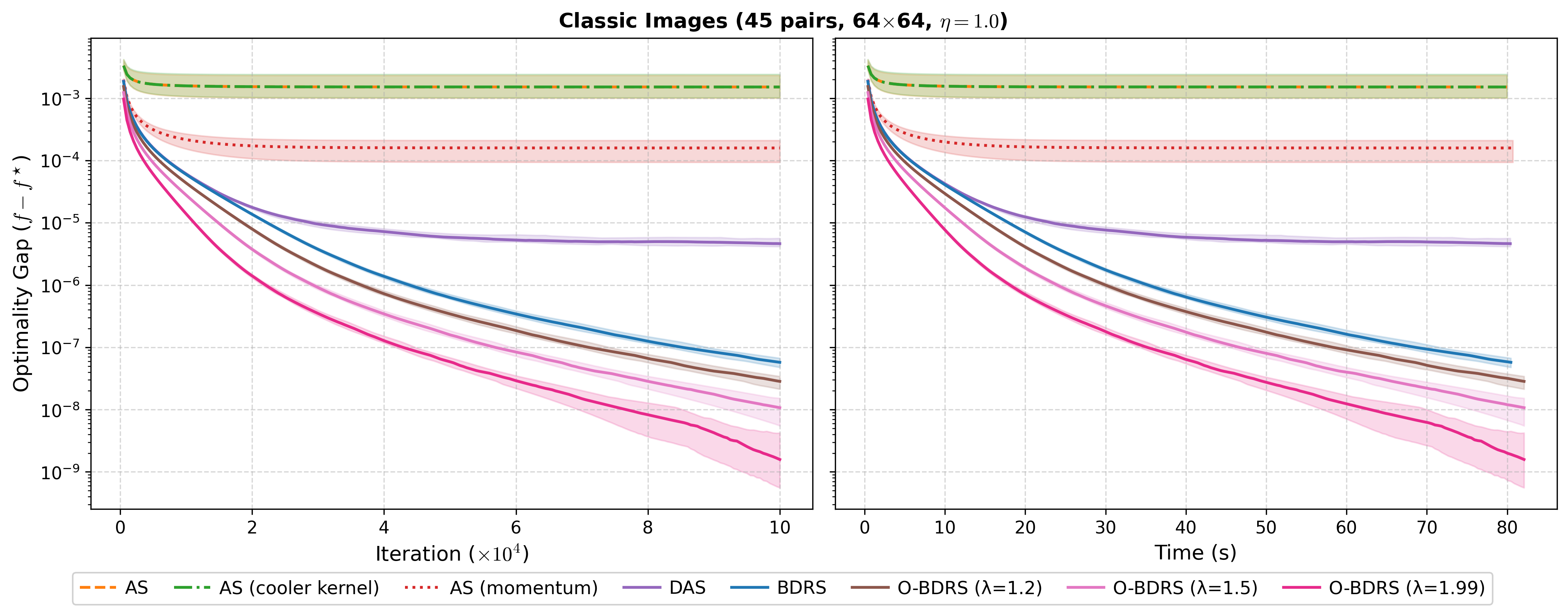}

    \caption{
        Momentum and kernel ablations on the Cauchy Density (top)
        and Classic images (bottom) classes of DOTmark.
        Each class contains 45 image pairs at resolution
        $64\times64$, with $\eta=1$.
        The left panels show the transport cost difference
        $f-f^\star$ against iteration count, and the right panels
        show the same quantity against median solver time,
        excluding rounding.
        Here $f$ is the evaluated transport cost and $f^\star$
        is the reference value computed using POT's network
        simplex solver.
        Lines show medians across image pairs; shaded bands
        indicate the 25th--75th percentiles.
    }
    \label{fig:momentum-ablation-cauchy-classic}
\end{figure}

\begin{figure}[tbp]
    \centering
    \includegraphics[width=\textwidth]{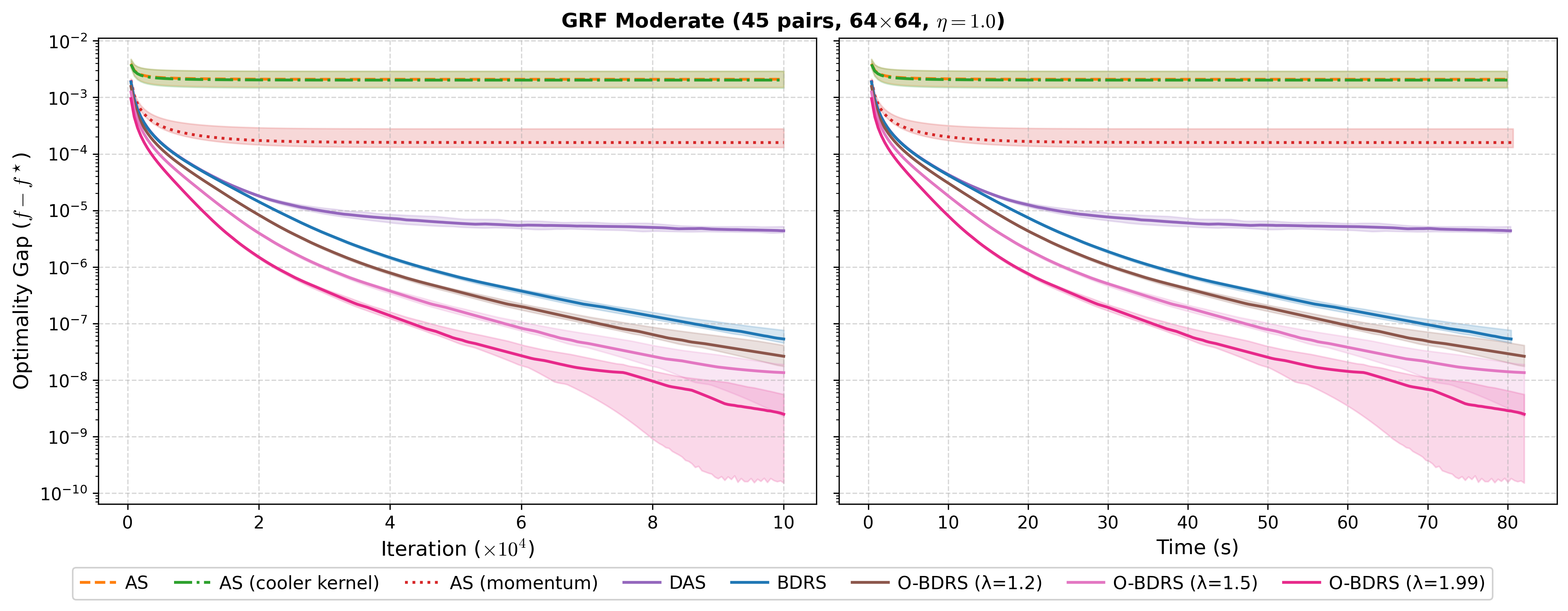}
    \par\medskip
    \includegraphics[width=\textwidth]{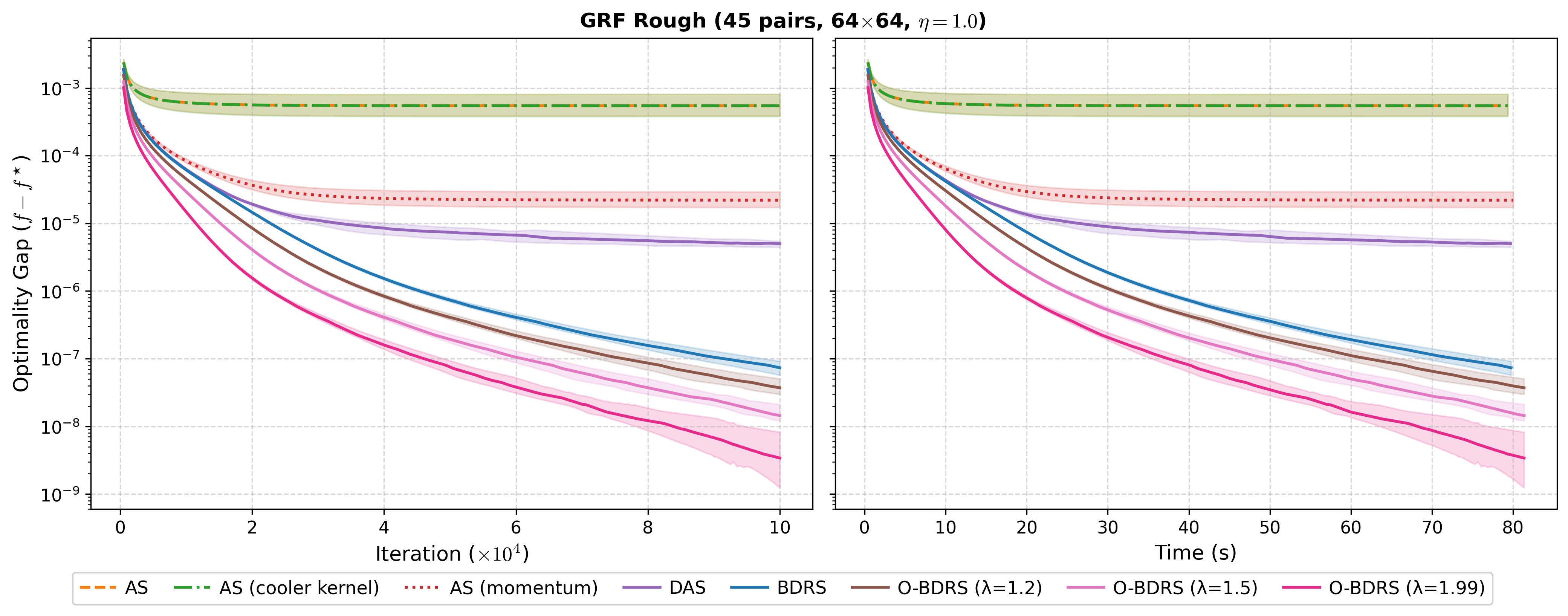}

    \caption{
        Momentum and kernel ablations on the GRF Moderate (top)
        and GRF Rough (bottom) classes of DOTmark.
        Each class contains 45 image pairs at resolution
        $64\times64$, with $\eta=1$.
        The left panels show the transport cost difference
        $f-f^\star$ against iteration count, and the right panels
        show the same quantity against median solver time,
        excluding rounding.
        Here $f$ is the evaluated transport cost and $f^\star$
        is the reference value computed using POT's network
        simplex solver.
        Lines show medians across image pairs; shaded bands
        indicate the 25th--75th percentiles.
    }
    \label{fig:momentum-ablation-grf-moderate-rough}
\end{figure}

\begin{figure}[tbp]
    \centering
    \includegraphics[width=\textwidth]{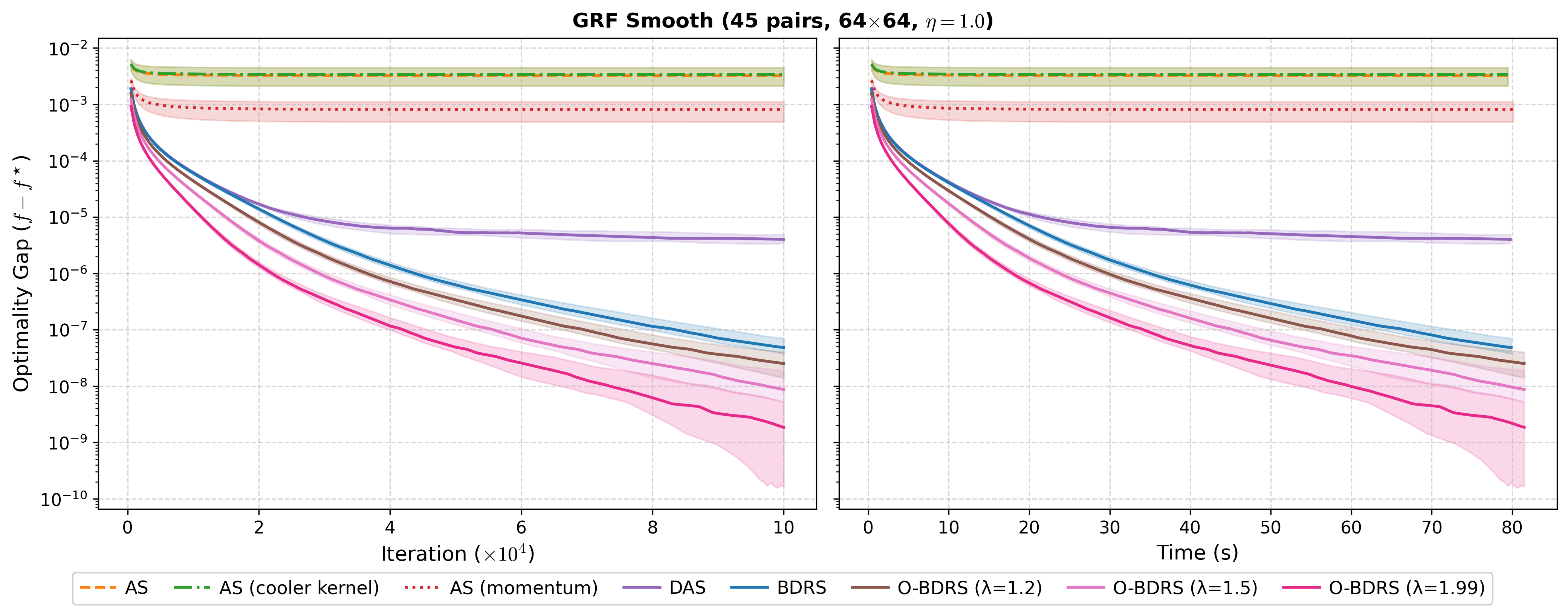}
    \par\medskip
    \includegraphics[width=\textwidth]{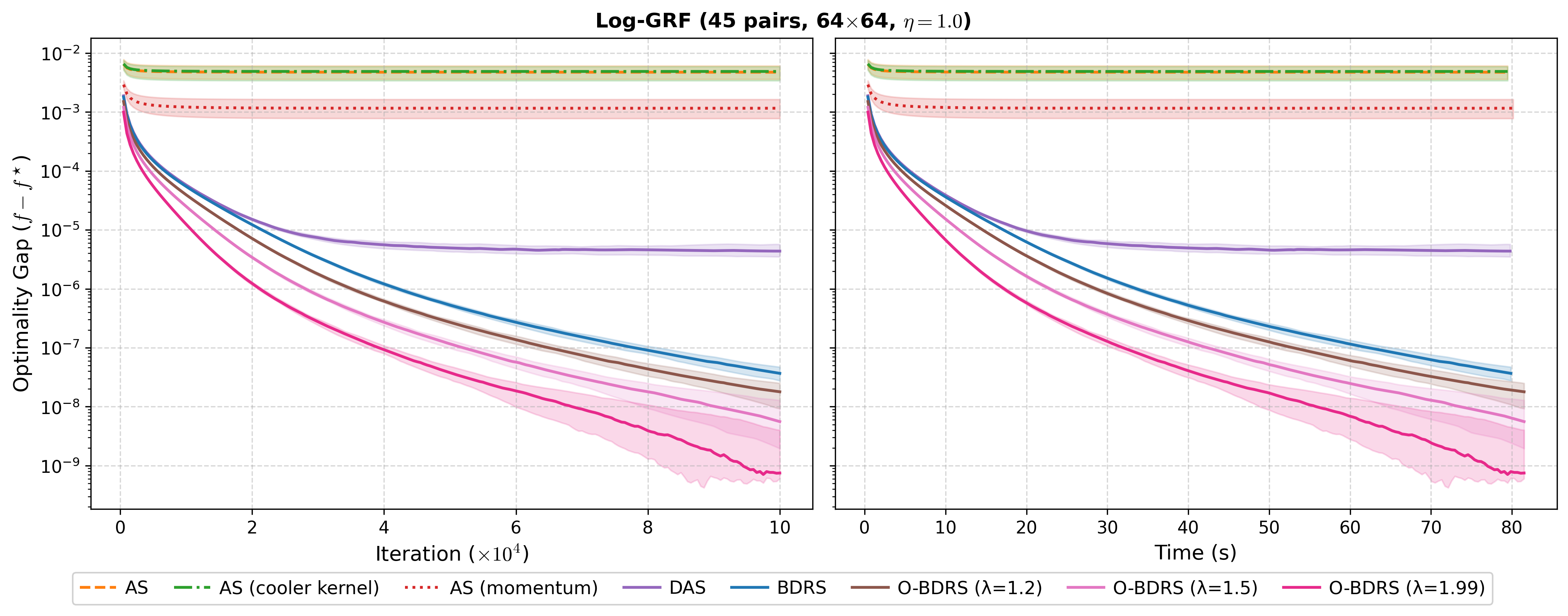}

    \caption{
        Momentum and kernel ablations on the GRF Smooth (top)
        and Log-GRF (bottom) classes of DOTmark.
        Each class contains 45 image pairs at resolution
        $64\times64$, with $\eta=1$.
        The left panels show the transport cost difference
        $f-f^\star$ against iteration count, and the right panels
        show the same quantity against median solver time,
        excluding rounding.
        Here $f$ is the evaluated transport cost and $f^\star$
        is the reference value computed using POT's network
        simplex solver.
        Lines show medians across image pairs; shaded bands
        indicate the 25th--75th percentiles.
    }
    \label{fig:momentum-ablation-grf-smooth-log}
\end{figure}

\begin{figure}[tbp]
    \centering
    \includegraphics[width=\textwidth]{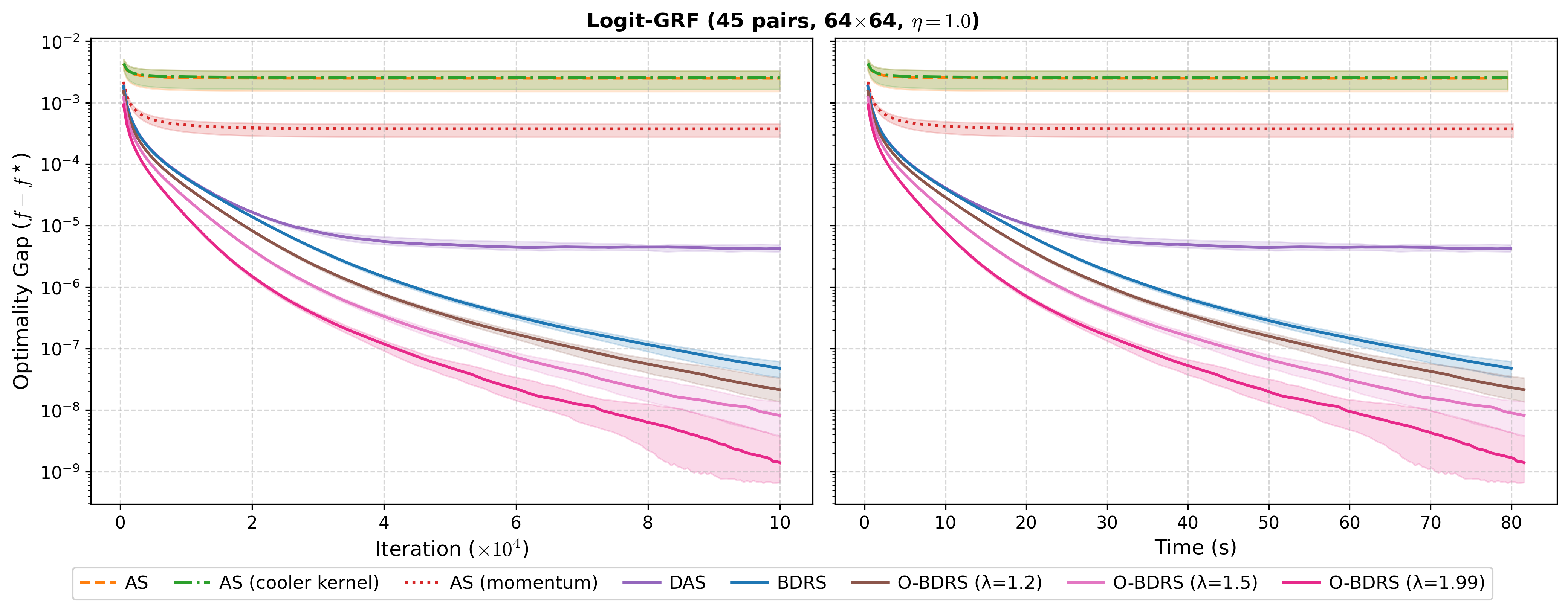}
    \par\medskip
    \includegraphics[width=\textwidth]{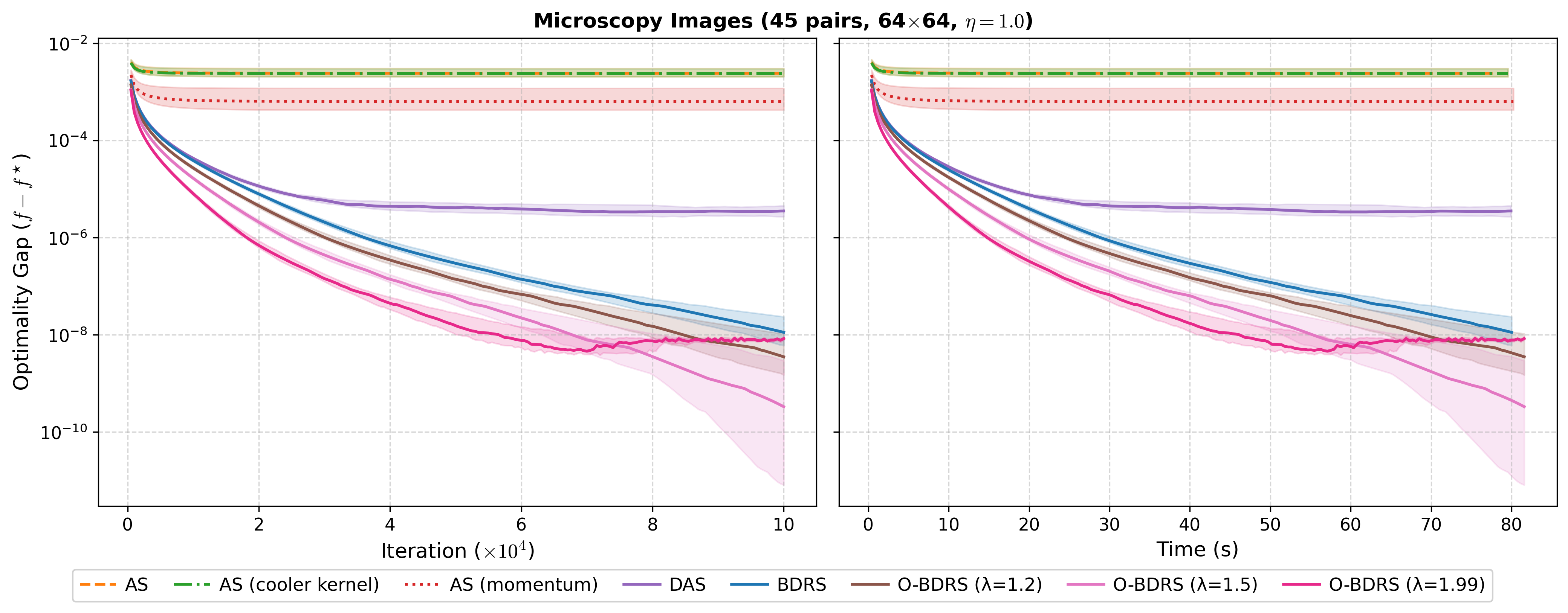}

    \caption{
        Momentum and kernel ablations on the Logit-GRF (top)
        and Microscopy Images (bottom) classes of DOTmark.
        Each class contains 45 image pairs at resolution
        $64\times64$, with $\eta=1$.
        The left panels show the transport cost difference
        $f-f^\star$ against iteration count, and the right panels
        show the same quantity against median solver time,
        excluding rounding.
        Here $f$ is the evaluated transport cost and $f^\star$
        is the reference value computed using POT's network
        simplex solver.
        Lines show medians across image pairs; shaded bands
        indicate the 25th--75th percentiles. 
        Here, we observe that O-\ac{BDRS}
        has a performance degradation on Microscopy Images as the number of iterations increase.
        We suspect that this late-iteration degradation reflects numerical sensitivity at low effective temperatures, potentially exacerbated by the $10^{-15}$ mass assigned to zero-mass pixels, which are common in Microscopy Images and Shapes, but not the other classes.
    }
    \label{fig:momentum-ablation-logit-microscopy}
\end{figure}

\begin{figure}[tbp]
    \centering
    \includegraphics[width=\textwidth]{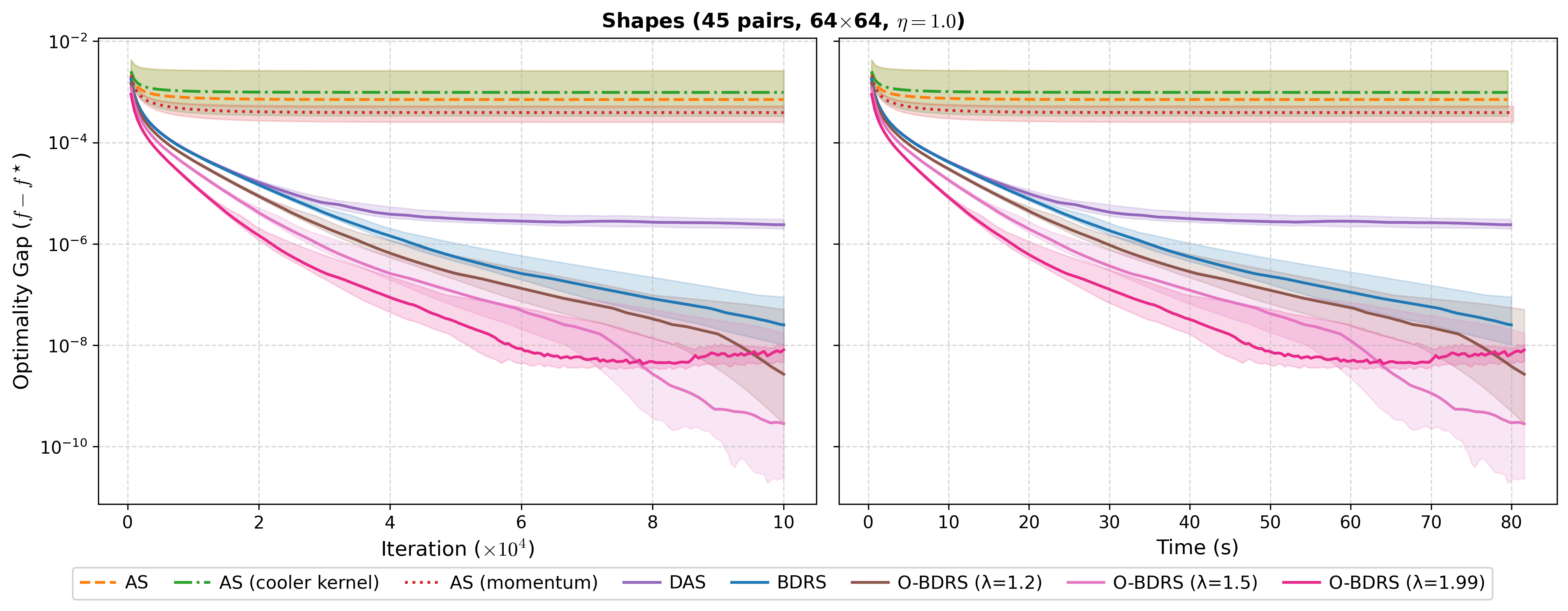}
    \par\medskip
    \includegraphics[width=\textwidth]{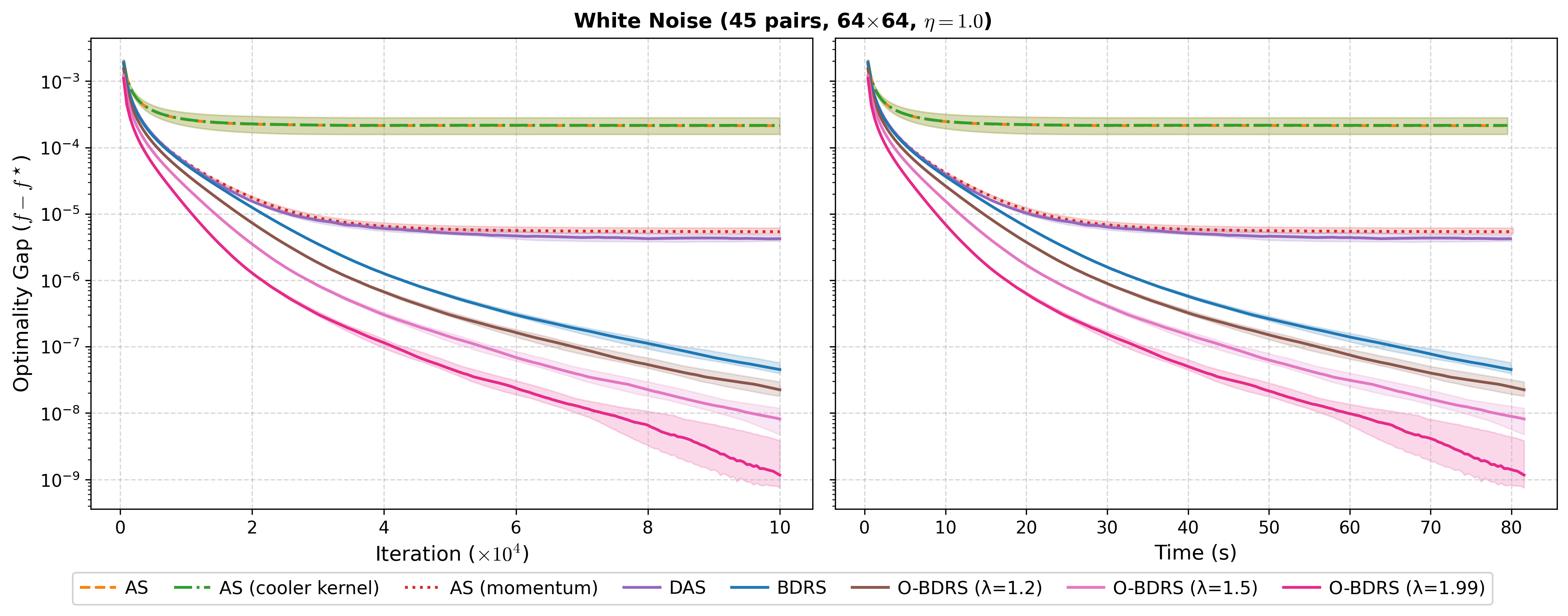}

    \caption{
        Momentum and kernel ablations on the Shapes (top)
        and White Noise (bottom) classes of DOTmark.
        Each class contains 45 image pairs at resolution
        $64\times64$, with $\eta=1$.
        The left panels show the transport cost difference
        $f-f^\star$ against iteration count, and the right panels
        show the same quantity against median solver time,
        excluding rounding.
        Here $f$ is the evaluated transport cost and $f^\star$
        is the reference value computed using POT's network
        simplex solver.
        Lines show medians across image pairs; shaded bands
        indicate the 25th--75th percentiles.
        Here, we observe that O-\ac{BDRS}
        has a performance degradation on Shapes as the number of iterations increase.
        We suspect that this late-iteration degradation reflects numerical sensitivity at low effective temperatures, potentially exacerbated by the $10^{-15}$ mass assigned to zero-mass pixels, which are common in Microscopy Images and Shapes, but not the other classes.
    }
    \label{fig:momentum-ablation-shapes-noise}
\end{figure}